\documentclass{article}
\usepackage {amscd}
\usepackage {amsmath}
\usepackage {amssymb}
\usepackage [backend=biber,style=alphabetic,maxalphanames=4,sorting=nyt]{biblatex}
\usepackage {geometry}
\usepackage {setspace}
\usepackage {tikz}
\usepackage {tikz-cd}
\usetikzlibrary{decorations.pathreplacing, calc}
\usepackage {float}
\usepackage {indentfirst}
\usepackage {enumitem}
\usepackage [amsmath,amsthm,thmmarks,hyperref,thref]{ntheorem}
\usepackage {mathabx}
\usepackage {diagbox}
\usepackage {graphicx}
\usepackage{longtable}
\usepackage{etoolbox}
\usepackage {hyperref}
\preto\tabular{\setcounter{magicrownumbers}{0}}
\preto\longtable{\setcounter{magicrownumbers}{0}}
\newcounter{magicrownumbers}
\newcommand\rownumber{\stepcounter{magicrownumbers}\arabic{magicrownumbers}}
\graphicspath{ {./images/} }

\theoremstyle{plain}
\newtheorem{theorem}{Theorem}[section]
\newtheorem{proposition}[theorem]{Proposition}
\newtheorem{corollary}[theorem]{Corollary}
\newtheorem{lemma}[theorem]{Lemma}

\theoremstyle{definition}
\newtheorem{definition}[theorem]{Definition}
\newtheorem{proposition-definition}[theorem]{Proposition-Definition}

\newtheorem{convention}[theorem]{Convention}

\theoremstyle{remark}
\newtheorem{example}[theorem]{Example}
\newtheorem{construction}[theorem]{Construction}
\newtheorem{algorithm}[theorem]{Algorithm}
\newtheorem{statement}[theorem]{}
\newtheorem*{remark}{Remark}

\newcolumntype{C}{>{$}c<{$}} % math-mode version of "l" column type

\usepackage[tickmarkheight=.5em,textwidth=\marginparwidth,textsize=small]{todonotes}

\author {Yang He \and Vyacheslav V. Shokurov}
\title {On the gap of thresholds}

\begin {document}
\begin {spacing}{1.3}
\maketitle

\begin{abstract}
    We propose an algorithm to compute the gap numbers of log canonical thresholds, anticanonical thresholds, and $\mathbb{R}$-complementary thresholds for surface pairs with a given rational hyperstandard coefficient set. As an example, we compute these gap numbers for a special hyperstandard set that appears in the effective boundedness of $n$-complements for surface pairs.
\end{abstract}

\section{Introduction}

\subsection{Background}

    The study of thresholds plays an important role in modern algebraic geometry. One important result in this direction is the ascending chain condition (or ACC) for log canonical thresholds, which was first conjectured in \cite{Shokurov93}:
    
\begin{theorem}[cf. \cite{HMX}, Theorem 1.1]\label{thm: ACC for lct}
    Fix a positive integer $d$ and a set $\Gamma \subseteq [0,1]$. Let $J \subset \mathbb{R}_{>0}$ be another set of real numbers. We define
  $$
    \mathrm{LCT}_{d}(\Gamma,J) = \{\mathrm{lct}(X/Z,\Delta;M) \mid \mathrm{dim}X=d, \Delta \in \Gamma, M \in J \}.
  $$
  If $\Gamma$ and $J$ satisfy DCC, then $\mathrm{LCT}_{d}(\Gamma,J)$ satisfies ACC.
\end{theorem}

    Another important threshold that appeared recently is the $\mathbb{R}$-complementary threshold. The set of $\mathbb{R}$-complementary thresholds also satisfies ACC according to the following theorem:

    \begin{theorem}[ACC for $\mathbb{R}$-complementary threshold, cf. \cite{Shokurov20}, Theorem 21]
    Fix a positive integer $d$ and a set $\Gamma \subseteq [0,1]$. Let $J \subset \mathbb{R}_{>0}$ be another set of real numbers. We define
    $$
    \mathbb{R}\mathrm{-CT}_{d}(\Gamma,J) = \{\mathbb{R}\mathrm{-ct}(X/Z,\Delta;M) \mid X \text{ is of Fano type},\, \mathrm{dim}X=d,\, \Delta \in \Gamma, M \in J \}.
    $$
    If $\Gamma$ and $J$ satisfy DCC, then $\mathbb{R}\mathrm{-CT}_{d}(\Gamma,J)$ satisfies ACC.
\end{theorem}

    The ACC property implies the existence of gap numbers. More explicitly, for any ACC set $S$ and any element $s \in S$, the subset $\{s' \in S \mid s' < s \}$ has a maximal number. For log canonical thresholds and $\mathbb{R}$-complementary thresholds, the gap number below 1 provides effective boundedness for many birational invariants. The first and second gap numbers of the $\mathbb{R}$-complementary thresholds for surfaces without boundary were computed in \cite{LiuSho}:

\begin{theorem}[cf. \cite{LiuSho}, Theorem 1.3]\label{thm: gap number for trivial coefficients}
    Let $X$ be an lc surface with an $\mathbb{R}$-complement and $F$ be a nonzero effective Weil divisor on $X$. Then we have either $\mathbb{R}\mathrm{-ct}(X,0;F) = 1$, or $\mathbb{R}\mathrm{-ct}(X,0;F) = \frac{12}{13}$, or $\mathbb{R}\mathrm{-ct}(X,0;F) \leq \frac{10}{11}$.
\end{theorem}
    
\subsection{Main results}
    The main goal of this article is to provide a computer-assisted algorithm to compute the explicit gap numbers for any rational hyperstandard coefficient set. As an application, we compute the gap numbers for a specific coefficient set $\Gamma(\mathcal{N}_{1})$ (see \thref{statement: setup}). This coefficient set is used to obtain the effective boundedness of surface $n$-complements.
    
    The first result concerns a gap number of local type:

\begin{theorem}\label{thm:lct gap}
    Let $\Gamma \subset [0,1]$ be a DCC set and assume $0,1 \in \Gamma$. Let
    $$
    \widetilde{\Gamma} = \{\frac{m - 1 + \sum n_{i}b_{i}}{m} \mid m \in \mathbb{Z}_{>0}, n_{i} \in \mathbb{Z}_{\geq 0}, b_{i} \in \Gamma \} \cap [0,1].
    $$
    be the set of adjunction coefficients of $\Gamma$. Let $(X,B)$ be an lc surface pair such that $B\in \Gamma$. Let $F$ be a nonzero effective $\mathbb{Q}$-Cartier Weil divisor on $X$ and write
    $$
    t=\mathrm{lct}(X,B;F).
    $$
    Then either $t=1$ or $t \leq 1 - \min\{ \frac{1}{3}, \frac{\alpha_{0}(\widetilde{\Gamma})}{2}, \alpha_{1}(\widetilde{\Gamma}) - 1 \}$. Here (cf. \thref{def: gap numbers})
    $$
    \alpha_{u}(\widetilde{\Gamma}) =\min \{ \sum\limits_{i=1}^{m} n_{i}b_{i} \mid m \in \mathbb{Z}_{>0}, n_{i} \in \mathbb{Z}_{\geq 0} , b_{i} \in \widetilde{\Gamma}, \sum\limits_{i=1}^{m} n_{i}b_{i} > u \}.
    $$
    Moreover, the same is true if we replace the lc surface pair $(X,B)$ by a local lc surface pair $(x \in X,B)$, and replace the log canonical threshold by its local version.
\end{theorem}

    For our specific coefficient set, we have

\begin{corollary}\label{corollary: lct gap special coefficients}
    Let $(X,B)$ be an lc surface pair such that $B\in \Gamma(\mathcal{N}_{1})$. Let $F$ be a nonzero effective $\mathbb{Q}$-Cartier Weil divisor on $X$ and write
    $$
    t=\mathrm{lct}(X,B;F).
    $$
    Then either $t=1$ or $t \leq 1 - \frac{1}{252}$. Moreover, the same is true if we replace the lc surface pair $(X,B)$ by a local lc surface pair $(x \in X,B)$, and replace the log canonical threshold by its local version.
\end{corollary}

    The next two results concern the gap number of global thresholds:

\begin{theorem}\label{thm: act gap}
    Let $X$ be a $\mathbb{Q}$-factorial surface such that $\rho(X)=1$. Let $(X,B)$ be a klt 0-pair. Write
    $$
    B = b_{1}F + \sum\limits_{i=2}^{n} b_{i}D_{i},
    $$
    where $F$ is a nonzero effective Weil divisor and $D_{i}$ are distinct prime divisors on $X$. Assume that $b_{i} \in \Gamma(\mathcal{N}_{1})$ for $i = 2, \cdots, n$. Then $b_{1} \leq 1 - \frac{1}{127260}$. Moreover, equality holds if and only if $(X,B)$ is one of the pairs in \thref{example: optimal bound for act} and \thref{example: optimal bound for act 2}, with $F=L_4$ and $F=D_1$, respectively..
\end{theorem}

\begin{theorem}\label{thm: Rct gap}
    Let $(X,B)$ be a surface pair such that $B\in \Gamma(\mathcal{N}_{1})$ and $(X,B)$ has an $\mathbb{R}$-complement. Let $F$ be a nonzero effective Weil divisor on $X$. Write
    $$
    t=\mathbb{R}\mathrm{-ct}(X,B ;F).
    $$
    Then either $t=1$ or $t \leq 1 - \frac{1}{127260}$.
\end{theorem}

\section{Preliminaries}

    In this section, we recall some definitions and results in birational geometry.

\begin{convention}
    Throughout the article, we work over an algebraically closed field of characteristic 0.
\end{convention}

\subsection{Log discrepancies of surface pairs}

    In this subsection, we recall some definitions in birational geometry and some results about surface pairs.

\begin{definition}
    Let $\pi: X \rightarrow U$ be a proper morphism between normal algebraic spaces.
\begin{enumerate}
    \item An \emph{$\mathbb{R}$-Weil divisor} (frequently abbreviated to \emph{$\mathbb{R}$-divisor}) $D$ on $X$ is an $\mathbb{R}$-linear combination $D = \sum d_{i}D_{i}$ of prime divisors. If further we have $0 \leq d_{i} \leq 1$ for all $i$, we say that $D$ is an \emph{$\mathbb{R}$-boundary} (frequently abbreviated to \emph{boundary}).
    
    \item An \emph{$\mathbb{R}$-Cartier divisor} $D$ is an $\mathbb{R}$-linear combination of Cartier divisors.
    
    \item Two $\mathbb{R}$-divisors $D$ and $D'$ are \emph{$\mathbb{R}$-linearly equivalent over $U$}, denoted
    \[
    D \sim_{\mathbb{R},U} D',
    \]
    if their difference is an $\mathbb{R}$-linear combination of principal divisors and an $\mathbb{R}$-Cartier divisor pulled back from $U$.
    
    \item Two $\mathbb{R}$-divisors $D$ and $D'$ are \emph{numerically equivalent over $U$}, denoted
    \[
    D \equiv_U D',
    \]
    if their difference is an $\mathbb{R}$-Cartier divisor such that
    \[
    (D-D')\cdot C = 0
    \]
    for any curve $C$ contained in a fibre of $\pi$.
    
    \item An $\mathbb{R}$-Cartier divisor $D$ is \emph{ample over $U$} (or \emph{$\pi$-ample}) if it is $\mathbb{R}$-linearly equivalent to a positive linear combination of ample (in the usual sense) Cartier divisors over $U$.
    
    \item An $\mathbb{R}$-Cartier divisor $D$ on $X$ is \emph{nef over $U$} (or \emph{$\pi$-nef}) if
    \[
    D \cdot C \geq 0
    \]
    for any curve $C \subset X$, contracted by $\pi$.
    
    \item An $\mathbb{R}$-Cartier divisor $D$ is \emph{semi-ample over $U$} (or \emph{$\pi$-semi-ample}) if there is a morphism
    \[
    f \colon X \to Y
    \]
    over $U$ such that $D$ is $\mathbb{R}$-linearly equivalent to the pullback of an ample $\mathbb{R}$-divisor over $U$.
    
    \item An $\mathbb{R}$-divisor $D$ is \emph{pseudo-effective} if the numerical class of $D$ over $U$ is a limit of divisors $D_i \geq 0$.
\end{enumerate}
    Similarly, if we replace the $\mathbb{R}$-linear combinations by $\mathbb{Q}$-linear combinations, we have the $\mathbb{Q}$-version of the above definitions, e.g., $\mathbb{Q}$-divisors.
\end{definition}

\begin{definition}[Contractions]
    A \emph{contraction}, or an \emph{algebraic fibre space}, is a proper, surjective morphism $f: X \rightarrow Y$ between normal algebraic varieties, such that $f_{*}\mathcal{O}_{X} = \mathcal{O}_{Y}$.
\end{definition}

\begin{definition}[Log pairs]
    A \emph{log pair} $(X/Z,B)$ is a projective contraction $X/Z$ together with an $\mathbb{R}$-boundary $B$.
\end{definition}

    We use the standard notions of lc, klt, and plt pairs. In particular, we assume that $K_{X} + B$ is $\mathbb{R}$-Cartier when we say that $(X,B)$ is lc, klt, or plt. Next, we recall some results for log discrepancies in dimension 2. For details, see \cite{Alexeev92_lc_surface}.

\begin{definition}[Dual graph]
    Let $Y \rightarrow X$ be the minimal resolution of the surface $X$. The \emph{weighted dual graph}, or simply \emph{dual graph} of $X$, is the following:
    \begin{enumerate}
        \item each exceptional curve $E_{i}$ corresponds to a vertex $v_{i}$;
        \item two vertices $v_{i}$ and $v_{j}$ are connected by an edge of weight $m$ if the corresponding curves intersect with $E_{i} \cdot E_{j} = m$;
        \item each vertex $v_{i}$ has a positive weight $n_{i} = - E_{i}^2$.
    \end{enumerate}
    Denote by $\Gamma$ the dual graph of $X$. For any subgraph $\Gamma' \subseteq \Gamma$, we denote by $\Delta(\Gamma')$ the absolute value of the determinant of the submatrix $(E_{i}\cdot E_{j})$, made up of columns and rows corresponding to the vertices of $\Gamma'$. We use the convention $\Delta(\varnothing)=1$.
\end{definition}

\begin{proposition}[A formula for surface log discrepancy]\label{prop:formula for ld}
    Let $(X,B)$ be a surface pair and $f:Y \rightarrow X$ be the minimal resolution of $X$. Let $P$ be a singular point of $X$ and $E_{1},\cdots , E_{n}$ be the divisors on $Y$ centered at $P$. Let $\Gamma$ be the dual graph of $P \in X$. Assume that every exceptional curve $E_{i}$ is rational and $\Gamma$ is a tree. Then we have
    $$
    a(E_{i};X,B) = \frac{\sum\limits_{j=1}^{n} \Delta(\Gamma - \mathrm{Path}(v_{i},v_{j}))c_{j}}{\Delta(\Gamma)}
    $$
    where $a(E_{i};X,B)$ is the log discrepancy of $(X,B)$ at $E_{i}$, $\mathrm{Path}(v_{i},v_{j})$ is the path between vertices $v_{i}$ and $v_{j}$ (such path is unique since $\Gamma$ is a tree), and
    $$
    c_{j} = 2 - (f^{-1}_{*}B + \sum\limits_{k\neq j}E_{k})\cdot E_{j}.
    $$
\end{proposition}

\subsection{The theory of complements}

\begin{definition}
    Let $(X/Z,D)$ be a log pair. Then it is said to be
    \begin{enumerate}
        \item \emph{log Fano} if $(X/Z,D)$ is lc and $-(K_{X} + D)$ is ample over $Z$; if $D = 0$ we just say $X$ is Fano over $Z$;
        \item \emph{weak log Fano (WLF)} if $(X/Z,D)$ is lc and $-(K_{X} + D)$ is nef and big over $Z$; if $D = 0$ we just say $X$ is weak Fano over $Z$.
        \item a \emph{0-pair} if $(X/Z,D)$ is lc and $K_{X} + D$ is numerically trivial over $Z$.
    \end{enumerate}
\end{definition}

\begin{definition}[Fano type contractions]\label{def: Fano type contraction}
    Let $X/Z$ be a projective contraction. We say that $X$ is \emph{of Fano type} over $Z$ if $(X,D)$ is klt weak log Fano over $Z$ for some choice of $D$; it is easy to see this is equivalent to the existence of a big/$Z$ $\mathbb{Q}$-boundary (resp. $\mathbb{R}$-boundary) $\Gamma$ so that $(X, \Gamma)$ is klt and $K_{X} + \Gamma \sim_{\mathbb{Q}} 0/Z$ (resp. $\sim_{\mathbb{R}}$ instead of $\sim_{\mathbb{Q}}$).
\end{definition}

\begin{remark}
    In \cite[\S 1.1]{Shokurov20}, a proper contraction $X/Z$ is said to have weak Fano type if there exists a big/$Z$ boundary $\Gamma$ such that $(X,\Gamma)$ is klt and $K_X+\Gamma\sim_{\mathbb{R},Z}0$. In that terminology, Fano type is equivalent to weak Fano type together with projectivity. Since the contractions considered here are projective, these two notions coincide in our setting. We therefore use “Fano type” throughout, including when applying results of \cite{Shokurov20} stated for weak Fano type.
\end{remark}

\begin{definition}[$\mathbb{Q}$-complements and $\mathbb{R}$-complements]
    Let $(X/Z,D)$ be a log pair. An \emph{$\mathbb{R}$-complement} (resp. \emph{$\mathbb{Q}$-complement}) of $(X/Z,D)$ is an $\mathbb{R}$-boundary (resp. $\mathbb{Q}$-boundary) $D'$ satisfying
    $$
    D' \ge D, \qquad (X,D') \text{ is lc}, \qquad K_{X}+D' \sim_{\mathbb{R}(\text{resp. }\mathbb{Q}), Z} 0.
    $$

    These definitions can be done in the more general situation: when $D$ is an \emph{$\mathbb{R}$-subboundary} (i.e., an $\mathbb{R}$-divisor $D=\sum d_i D_i$ such that $d_i \le 1$ for all $i$).
\end{definition}

\subsection{Thresholds}

    In this subsection, we recall some basic definitions about several thresholds and their ACC (ascending chain condition) property.

\begin{definition}[log canonical threshold]
    Let $(X,\Delta)$ be an lc pair. Let $M > 0$ be an $\mathbb{R}$-Cartier divisor. We define the \emph{log canonical threshold of $(X,\Delta)$ with respect to $M$} to be
    $$
    \mathrm{lct}(X/Z,\Delta;M) = \sup\{t \in \mathbb{R}_{\geq 0} \mid (X/Z,\Delta + tM) \text{ is log canonical} \}.
    $$
\end{definition}

\begin{definition}[anticanonical threshold]
    Let $(X/Z,\Delta)$ be an lc pair. Let $H$ be an $\mathbb{R}$-divisor ample over $Z$. We define the \emph{anticanonical threshold of $(X/Z,\Delta)$ with respect to $H$} to be
    $$
    \mathrm{act}(X/Z,\Delta;H) = \inf\{t \in \mathbb{R}_{\geq 0} \mid K_{X} + \Delta + tH \text{ is nef over }Z \}.
    $$
\end{definition}

\begin{definition}[$\mathbb{R}$-complementary threshold]
    Let $(X/Z,\Delta)$ be a pair admitting an $\mathbb{R}$-complement. Let $M > 0$ be a nonzero effective $\mathbb{R}$-divisor. We define the \emph{$\mathbb{R}$-complementary threshold of $(X/Z,\Delta)$ with respect to $M$} to be
    $$
    \mathbb{R}\mathrm{-ct}(X/Z,\Delta;M) = \sup\{t \in \mathbb{R}_{\geq 0} \mid (X/Z,\Delta + tM) \text{ has an } \mathbb{R}\text{-complement} \}.
    $$
\end{definition}

    We recall the ACC properties for these thresholds.

\begin{theorem}[ACC for $\mathbb{R}$-complementary threshold, cf. \cite{Shokurov20}, Theorem 21]
    Fix a positive integer $d$ and a set $\Gamma \subseteq [0,1]$. Let $J \subset \mathbb{R}_{>0}$ be another set of real numbers. We define
    $$
    \mathbb{R}\mathrm{-CT}_{d}(\Gamma,J) = \{\mathbb{R}\mathrm{-ct}(X/Z,\Delta;M) \mid X \text{ is of Fano type},\, \mathrm{dim}X=d,\, \Delta \in \Gamma, M \in J \}.
    $$
    If $\Gamma$ and $J$ satisfy DCC, then $\mathbb{R}\mathrm{-CT}_{d}(\Gamma,J)$ satisfies ACC.
\end{theorem}

\begin{theorem}[ACC for log canonical threshold, cf. \cite{HMX}, Theorem 1.1]
    Fix a positive integer $d$ and a set $\Gamma \subseteq [0,1]$. Let $J \subset \mathbb{R}_{>0}$ be another set of real numbers. We define
  $$
    \mathrm{LCT}_{d}(\Gamma,J) = \{\mathrm{lct}(X/Z,\Delta;M) \mid \mathrm{dim}X=d, \Delta \in \Gamma, M \in J \}.
  $$
  If $\Gamma$ and $J$ satisfy DCC, then $\mathrm{LCT}_{d}(\Gamma,J)$ satisfies ACC.
\end{theorem}

\begin{theorem}[ACC for anticanonical threshold, \cite{Shokurov20}, Corollary 24]
    Fix a positive integer $d$ and sets $\Gamma\subseteq[0,1]$ and $J\subseteq\mathbb{R}_{>0}$. We define
    $$
    \begin{aligned}
    \mathrm{ACT}_{d}(\Gamma,J)
    =\Bigl\{
    t=\mathrm{act}(X/Z,\Delta;H)
    \ \Bigm|\
    &\dim X=d>\dim Z,\\
    &X/Z\text{ is of Fano type},\\
    &(X,\Delta)\text{ is lc},\quad
      \Delta\in\Gamma,\quad H\in J,\\
    &K_{X}+\Delta+tH\equiv_{Z}0
    \Bigr\}.
    \end{aligned}
    $$
    Here $X\rightarrow Z$ runs through projective contractions, and $H>0$ is an effective $\mathbb{R}$-divisor ample over $Z$ whose prime components are locally free over $Z$. If $\Gamma$ and $J$ satisfy DCC, then $\mathrm{ACT}_{d}(\Gamma,J)$ satisfies ACC.
\end{theorem}

We will use the following lemma to reduce the computation of gap numbers for $\mathbb{R}$-complementary thresholds to the computation of gap numbers for log canonical thresholds and anticanonical thresholds:

\begin{lemma}[cf. {\cite[Lemma 14]{Shokurov20}}]
    Let $(X/Z,B)$ be a pair such that $X/Z$ has Fano type, $(X/Z,B)$ has an $\mathbb{R}$-complement, and $F>0$ be an effective divisor on $X$. Put
    $$
    t = \mathbb{R}\mathrm{-ct}(X/Z,B;F).
    $$
    Then one of the following holds:
    \begin{enumerate}[label=(\arabic*)]
        \item \textbf{The log canonical threshold case.} There exists a $\mathbb{Q}$-factorial closed point germ $o' \in X'$ of Fano type of dimension $\leq \mathrm{dim}X$, and effective divisors $B'$ and $F'$ on $X'$, such that 
        $$
        t = \mathbb{R}\mathrm{-ct}(X' \ni o', B';F') = \mathrm{lct}(X' \ni o', B';F').
        $$
        Moreover, if $F$ is a prime divisor and $\mathrm{dim}X'=\mathrm{dim}X$, then $F'$ is also a prime divisor.
        \item \textbf{The anticanonical threshold case.} There exists a projective contraction $X'' \rightarrow o''$, where $o''$ is a closed point, and effective divisors $B''$ and $F''$ on $X''$, such that $X''$ is $\mathbb{Q}$-factorial of Fano type of dimension $\leq \mathrm{dim}X$, $\rho(X'')=1$, and
        $$
        t = \mathbb{R}\mathrm{-ct}(X''/o'',B'';F'') = \mathrm{act}(X''/o'',B'';F'').
        $$
        Moreover, $(X'',B''+tF'')$ is an lc $0$-pair, and no lc center of $(X'',B''+tF'')$ is contained in $\mathrm{Supp}F''$. If $F$ is a prime divisor and $\dim X''=\mathrm{dim}X$, then $F''$ is also a prime divisor.
    \end{enumerate}
    In both cases, the multiplicities of $B', B''$ and of $F', F''$ are respectively the multiplicities of $B$ and of $F$.
\end{lemma}

    Although the above lemma is only stated for varieties of Fano type, it is expected for arbitrary variety. In low dimensions, we have the required instruments: the termination and semi-ampleness.

\begin{proposition}\label{prop: reduce complementary threshold}
    Let $(X/Z,B)$ be a projective surface pair admitting an $\mathbb{R}$-complement, and $F>0$ be an effective divisor on $X$. Put
    $$
    t = \mathbb{R}\mathrm{-ct}(X/Z,B;F).
    $$
    Then one of the following holds:
    \begin{enumerate}[label=(\arabic*)]
        \item \textbf{The log canonical threshold case.} There exists a $\mathbb{Q}$-factorial closed point germ $o' \in X'$ of Fano type of dimension $\leq 2$, and effective divisors $B'$ and $F'$ on $X'$, such that 
        $$
        t = \mathbb{R}\mathrm{-ct}(X' \ni o', B';F') = \mathrm{lct}(X' \ni o', B';F').
        $$
        Moreover, if $F$ is a prime divisor and $\mathrm{dim}X'=2$, then $F'$ is also a prime divisor.
        \item \textbf{The anticanonical threshold case.} There exists a projective contraction $X'' \rightarrow o''$, where $o''$ is a closed point, and effective divisors $B''$ and $F''$ on $X''$, such that $X''$ is $\mathbb{Q}$-factorial of Fano type of dimension $\leq 2$, $\rho(X'')=1$, and
        $$
        t = \mathbb{R}\mathrm{-ct}(X''/o'',B'';F'') = \mathrm{act}(X''/o'',B'';F'').
        $$
        Moreover, $(X'',B''+tF'')$ is an lc $0$-pair, and no lc center of $(X'',B''+tF'')$ is contained in $\mathrm{Supp}F''$. If $F$ is a prime divisor and $\dim X''=2$, then $F''$ is also a prime divisor.
    \end{enumerate}
    In both cases, the multiplicities of $B', B''$ and of $F', F''$ are respectively the multiplicities of $B$ and of $F$.
\end{proposition}

\begin{proof}
    The proof follows the same structure as the proof of \cite[Lemma 14]{Shokurov20}. First, if $t=0$, then we can take $o' \in X'$ to be a curve germ, $B'=o'$ and $F' = c o'$, where $c$ is a coefficient of $F$. Then $\mathrm{lct}(X' \ni o',B';F') = 0 = t$. Hence we can assume that $t>0$. 
    
    Since $\mathrm{dim} X=2$, the threshold $t$ is attained by \cite[Theorem 5.20(7)]{HLS} and the existence of good minimal models for surfaces (cf. \cite[proof of Theorem 7.2]{LiuSho}). If there is an lc center $C$ of $(X/Z,B+tF)$ contained in $\mathrm{Supp}F$, then we have $t = \mathrm{lct}(X \ni p, B; F)$, where $p$ is the generic point of $C$. After taking hyperplane sections, we reduce to the log canonical threshold case. Now we can assume that the lc centers of $(X/Z,B+tF)$ are not in $\mathrm{Supp}F$.
    
    Choose an $\mathbb{R}$-complement $(X/Z,B^{+})$ of $(X/Z,B+tF)$. Take a dlt modification $f: \widetilde{X} \rightarrow X$ and write
    $$
    K_{\widetilde{X}} + \widetilde{B^{+}} = f^{*}(K_{X}+B^{+}).
    $$
    Put $S = \lfloor \widetilde{B^{+}} \rfloor$ and $\widetilde{B} = S + \sum\limits_{D \nsubset S}\mathrm{mult}_{D}(B_{\widetilde{X}})D$, where $B_{\widetilde{X}}$ is the strict birational transformation of $B$ on $\widetilde{X}$. Let $\widetilde{F}$ be the strict birational transformation of $F$ on $\widetilde{X}$ with its components in $S$ removed. Then for every positive real number $r$ such that $(X,B+rF)$ is log canonical, we have
    $$
    f_{*}(K_{\widetilde{X}} + \widetilde{B} + r\widetilde{F}) \geq K_{X} + B + rF.
    $$
    Since the lc centers of $(X/Z,B+tF)$ are not in $\mathrm{Supp}F$, the pair $(X,B+rF)$ is log canonical for sufficiently close $r>t$. If $(\widetilde{X}/Z, \widetilde{B} + r\widetilde{F})$ has an $\mathbb{R}$-complement, then $(X/Z,B+rF)$ has an $\mathbb{R}$-complement, a contradiction. Hence we have
    $$
    t = \mathbb{R}\mathrm{-ct}(\widetilde{X}/Z, \widetilde{B}; \widetilde{F}).
    $$
    Hence, after replacing $(X/Z,B)$ and $F$ by $(\widetilde{X}/Z, \widetilde{B})$ and $\widetilde{F}$ respectively, we can assume that $(X/Z,B^{+})$ is dlt, and any prime component of $E = B^{+} - (B+tF)$ is not an lc center of $(X/Z,B^{+})$. In particular, for sufficiently small positive real number $\epsilon$, the pair $(X,B^{+}+\epsilon E)$ is log canonical.

    First, we run a $(K_{X}+B^{+}+\epsilon E)$-MMP over $Z$. Notice that
    $$
    K_{X}+B^{+}+\epsilon E \sim_{\mathbb{R},Z} \epsilon E
    $$
    and
    $$
    K_{X}+B + tF \sim_{\mathbb{R},Z} - E.
    $$
    Hence this MMP is also a $-(K_{X} + B + tF)$-MMP over $Z$. Because there are finitely many surface contractions, this MMP is also a $-(K_{X} + B + rF)$-MMP over $Z$ for a sufficiently close $r>t$. Consequently, a complement above $t$ on the resulting model lifts to a complement above $t$ on the preceding model by the negativity lemma. Hence the $\mathbb{R}$-complementary threshold is preserved in this MMP. After replacing $X$ be the outcome of this MMP, we can assume that $E$ is semi-ample over $Z$. Let $X \rightarrow Z'$ be the ample model of $E$ over $Z$. Then we have
    $$
    K_{X} + B + tF \sim_{\mathbb{R},Z'} K_{X}+B^{+} \sim_{\mathbb{R},Z} 0.
    $$
    Take a positive real number $\delta < t$ and run a $(K_{X} + B + (t - \delta)F)$-MMP on $X$ over $Z'$. Notice that 
    $$
    K_{X} + B + (t - \delta)F \sim_{\mathbb{R},Z'} -\delta F.
    $$
    Let $X'$ be the outcome of this MMP, and $B'$ (resp. $F'$, $E'$) be the strict birational transformation of $B$ (resp. $F$, $E$) on $X'$. Similar to the argument on the first MMP, the $\mathbb{R}$-complementary threshold is also preserved in this second MMP.
    
    If $X'$ is a log minimal model, then $-F'$ is semi-ample over $Z'$. If $t < \mathrm{lct}(X',B';F')$, then we can choose a sufficiently small positive real number $\epsilon$ such that $(X',B'+ (t + \epsilon)F')$ is log canonical. Moreover, for sufficiently small positive real number $\epsilon$, the divisor
    $$
    -(K_{X'}+B'+(t+\epsilon)F') \sim_{\mathbb{R}, Z} E' - \epsilon F'
    $$
    is semi-ample over $Z$. Hence $(X',B'+ (t + \epsilon)F')$ has an $\mathbb{R}$-complement, a contradiction. Hence we must have $t = \mathrm{lct}(X',B';F')$. After taking an lc center and hyperplane sections, we reduce to the log canonical threshold case.
    
    If $X'/Z''$ is a Mori fibre space of relative dimension $1$ over $Z'$, then $(X'/Z'',B'+tF')$ is a log canonical 0-pair. Moreover, $F'$ is ample over $Z''$. Hence either $t = \mathrm{lct}(X',B';F')$, or $t = \mathrm{act}(X'/Z'',B';F') < \mathrm{lct}(X',B';F')$. If $t = \mathrm{lct}(X',B';F')$, then after taking an lc center and hyperplane sections, we reduce to the log canonical threshold case. If $t = \mathrm{act}(X'/Z'',B';F') < \mathrm{lct}(X',B';F')$, then after taking a general fibre, we reduce to the anticanonical threshold case.
    
    If $X'/Z''$ is a Mori fibre space of relative dimension $2$ over $Z'$, then $Z''$ is a closed point and $\rho(X')=1$. Moreover, we have $E=0$ and $B^{+} = B + tF$. In particular, $(X',B'+tF')$ is a log canonical 0-pair. Hence either $t = \mathrm{lct}(X',B';F')$, or $t = \mathrm{act}(X',B';F') < \mathrm{lct}(X',B';F')$. If $t = \mathrm{lct}(X',B';F')$, then after taking an lc center and hyperplane sections, we reduce to the log canonical threshold case. If $t = \mathrm{act}(X',B';F') < \mathrm{lct}(X',B';F')$, then it is the anticanonical threshold case.
\end{proof}

\subsection{Classification of surfaces}

    In this subsection, we recall some results on the classification of surfaces. First, we recall the classification of log canonical surface pairs.

\begin{proposition}[cf. \cite{Alexeev92_lc_surface}]\label{prop: surface lc singularity}
    Let $(S,C+B)$ be a log canonical surface pair, where $C > 0$ is a reduced divisor on $S$, and $B$ is a boundary $\mathbb{R}$-divisor. Let $P \in C$ be a closed point. Then one of the following is true:
    \begin{enumerate}
        \item $C$ has a node at $P$, and $S$ is nonsingular at $P$.
        \item $C$ has a node at $P$, and $S$ has a singularity of type \textbf{A} on $P$. More precisely, the dual graph is of the following form:
    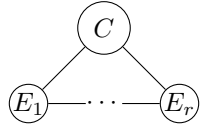
\begin{figure}[H]
    \centering
    \begin{tikzpicture}
        \node[inner sep=0] (0) {$\cdots$};
        \node[inner sep=0, draw, circle] (1) [right of=0] {$E_{r}$};
        \draw (0) -- (1);
        \node[draw, circle] (a) [above of=0]{$C$};
        \draw (1) -- (a);
        \node[inner sep=0, draw, circle] (-1) [left of=0] {$E_{1}$};
        \draw (0) -- (-1);
        \draw (-1) -- (a);
    \end{tikzpicture}
    \caption{Dual graph 1}\label{figure: dual graph 1}
    \end{figure}
    Notice that when $r=1$, we have $C \cdot E_{1} = 2$.
    \item $C$ is nonsingular at $P$, and $S$ has a singularity of type \textbf{D} on $P$. More precisely, the dual graph is of the following form:
    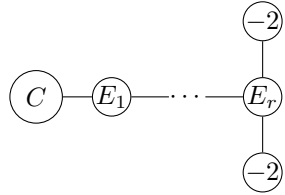
\begin{figure}[H]
        \centering
        \begin{tikzpicture}
            \node[draw, circle] (0) {$C$};
            \node[inner sep=0, draw, circle] (1) [right of=0] {$E_{1}$};
            \node[inner sep=0] (2) [right of=1] {$\cdots$};
            \node[inner sep=0, draw, circle] (3) [right of=2] {$E_{r}$};
            \node[inner sep=0, draw, circle] (3') [above of=3] {$-2$};
            \node[inner sep=0, draw, circle] (3'') [below of=3] {$-2$};
            \draw (0) -- (1);
            \draw (1) -- (2);
            \draw (2) -- (3);
            \draw (3) -- (3');
            \draw (3) -- (3'');
        \end{tikzpicture}
        \caption{Dual graph 2}\label{figure: dual graph 2}
    \end{figure}
    \item $C$ is nonsingular at $P$, and $S$ has a singularity of type \textbf{A} on $P$. More precisely, the dual graph is of the following form:
    \begin{figure}[H]
        \centering
        \begin{tikzpicture}
            \node[draw, circle] (0) {$C$};
            \node[inner sep=0, draw, circle] (1) [right of=0] {$E_{1}$};
            \node[inner sep=0] (2) [right of=1] {$\cdots$};
            \node[inner sep=0, draw, circle] (3) [right of=2] {$E_{r}$};
            \draw (0) -- (1);
            \draw (1) -- (2);
            \draw (2) -- (3);
        \end{tikzpicture}
        \caption{Dual graph 3}\label{figure: dual graph 3}
    \end{figure}
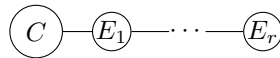
    \end{enumerate}
\end{proposition}

\begin{proposition}[cf. \protect{\cite{Shokurov93}[Section 3], \cite{Alexeev92_lc_surface}}]\label{prop: surface adjunction}
    Let $(S,C+B)$ be a log canonical surface pair, where $C > 0$ is a reduced divisor on $S$, and $B = \sum\limits_{i=1}^{n}b_{i}D_{i}$ is a boundary $\mathbb{R}$-divisor. Let $\nu: \widetilde{C} \rightarrow C$ be a normalization of $C$. Then we have
    $$
    \nu^{*}(K_{S} + C + B) = K_{\widetilde{C}} + \mathrm{Diff}_{\widetilde{C}}(B).
    $$
    Moreover, for any point $P \in \widetilde{C}$, one of the following holds:
    \begin{enumerate}[label=(\arabic*)]
        \item The curve $C$ has a node at $\nu(P)$. In this case $(S,C)$ is not plt in $\nu(P)$ and $\mathrm{mult}_{P}(\mathrm{Diff}_{\widetilde{C}}(B)) = 1$, $\mathrm{mult}_{P}(B) = 0$.
        \item The curve $C$ is nonsingular at $\nu(P)$ and $(S,C)$ is non-plt at $\nu(P)$. In this case $\nu$ is an isomorphism at $\nu(P)$ and $\mathrm{mult}_{P}(\mathrm{Diff}_{\widetilde{C}}(B)) = 1$.
        \item The pair $(S,C)$ is plt at $\nu(P)$. In this case, $C$ is nonsingular at $\nu(P)$ and $\nu$ is an isomorphism at $\nu(P)$. We have $$
        \mathrm{mult}_{P}(\mathrm{Diff}_{\widetilde{C}}(B)) = \frac{m - 1 + \sum\limits_{i=1}^{n}r_{i}b_{i}}{m}
        $$
        for some positive integer $m$ and nonnegative integers $r_{i}$. Moreover, we have:
        \begin{enumerate}
            \item The number $m$ is the Cartier index of $C$ at $\nu(P)$. It is also the Cartier index of the divisor $K_{S}+C$ at $\nu(P)$.
            \item The number $r_{i}$ is given by $\frac{r_{i}}{m} = (C \cdot D_{i})_{\nu(P)}$.
            \item The multiplicity $\mathrm{mult}_{P}(\mathrm{Diff}_{\widetilde{C}}(B)) > 0$ if and only if one of the following happens:
            \begin{enumerate}
                \item The integer $r_{i} > 0$ for some $i$. This case happens if and only if $\nu(P) \in C \cap \mathrm{Supp}B$;
                \item The integer $m > 1$. This case happens if and only if $S$ is singular at $\nu(P)$.
            \end{enumerate}
            There are only finitely many points satisfying the above condition.
        \end{enumerate}
    \end{enumerate}
    In particular, after taking the degree of both sides, we have
    $$
    (K_{S}+C+B) \cdot C = 2g(\widetilde{C}) - 2 + l + \sum\limits_{i=1}^{n}w_{i}b_{i} + \sum\limits_{i=1}^{u} \frac{m_{i}-1 + \sum\limits_{j=1}^{n}r_{ij}b_{j}}{m_{i}}
    $$
    where the integer $l$ denotes the number of points $P \in \widetilde{C}$ in Case (1) or Case (2), the integer $w_{i}$ is the intersection multiplicity of $C$ and $D_{i}$ in the smooth locus of $S$, and $P_{1},P_{2},\cdots , P_{u}$ are the points in Case (3) such that $S$ is singular at $\nu(P_{i})$.
\end{proposition}

    To prove the main theorems, we need the classification of nondegenerate surfaces of small degree.

\begin{theorem}[cf. \cite{DelPezzo1885MinimalDegreeSurfaces}]\label{thm: surface of minimal degree}
    If a projective surface $S$ of degree $d$ is in a projective space $\mathbb{P}^{d+1}$ but not in any hyperplane $\mathbb{P}^{d}$, then $S$ is isomorphic to one of the following:
    \begin{enumerate}
        \item A rational normal scroll $\mathbb{F}_{d-2l}$, embedded by the linear system $|E + (d-l)F|$.
        \item A Veronese embedding of $\mathbb{P}^2$ in $\mathbb{P}^5$.
        \item The projective plane $\mathbb{P}^2$.
        \item The rational cone $\mathbb{P}(1,1,d)$.
    \end{enumerate}
\end{theorem}

\subsection{Linear systems on Fano type surfaces}

    In this subsection, we study certain linear systems on a surface of Fano type. First, to prove the main results in this subsection, we need the following general Kodaira vanishing theorem.

\begin{theorem}[General Kodaira vanishing, cf. {\cite[Theorem 2.70]{KM}}, {\cite[Remark 2.1]{Fujino2024GeneralKodairaVanishing}}]\label{thm: generalized Kodaira vanishing}
    Let $(X,\Delta)$ be a proper klt pair. Let $N$ be a $\mathbb{Q}$-Cartier Weil divisor on $X$ such that $N \equiv M + \Delta$, where $M$ is a nef and big $\mathbb{R}$-Cartier $\mathbb{R}$-divisor. Then $H^{i}(X,\mathcal{O}_{X}(-N))=0$ for $i < \mathrm{dim}X$.
\end{theorem}

    We use the above theorem to prove the following proposition:

\begin{proposition}\label{prop: lifting sections}
    Let $(S,\Delta)$ be a proper klt weak Fano pair, and $C$ be a nonsingular rational curve on $S$ which does not contain singularities of $S$. If $C^2 \geq 0$, then the complete linear system $|C|$ is base-point free.
\end{proposition}

\begin{proof}
    Since $C$ does not contain singularities of $S$, it is a Cartier divisor of $S$. Consider the short exact sequence
    $$
    0 \rightarrow \mathcal{O}_{S} \rightarrow \mathcal{O}_{S}(C) \rightarrow \mathcal{O}_{C}(C) \rightarrow 0
    $$
    and the associated long exact sequence
    $$
    0 \rightarrow H^{0}(S,\mathcal{O}_{S}) \rightarrow H^{0}(S,\mathcal{O}_{S}(C)) \rightarrow H^{0}(C,\mathcal{O}_{C}(C)) \rightarrow H^{1}(S,\mathcal{O}_{S}) \rightarrow \cdots.
    $$
    Since $(S,\Delta)$ is klt, $S$ has rational singularities. In particular, $S$ is Cohen-Macaulay. Applying \thref{thm: generalized Kodaira vanishing} with $N = -K_{S} = -(K_{S}+\Delta) + \Delta$ and the Serre duality for Cohen-Macaulay varieties, we have $h^{1}(S,\mathcal{O}_{S}) = h^{1}(S,\mathcal{O}_{S}(K_{S})) = 0$. Hence, from the above long exact sequence, we get a short exact sequence
    \begin{equation}\label{equation: short exact sequence}
    0 \rightarrow H^{0}(S,\mathcal{O}_{S}) \rightarrow H^{0}(S,\mathcal{O}_{S}(C)) \rightarrow H^{0}(C,\mathcal{O}_{C}(C)) \rightarrow 0.
    \end{equation}
    If $C^{2} \geq 0$, then the line bundle $\mathcal{O}_{C}(C)$ is base-point free on $C$. By the surjection $H^{0}(S,\mathcal{O}_{S}(C)) \twoheadrightarrow H^{0}(C,\mathcal{O}_{C}(C))$, we conclude that the complete linear system $|C|$ has no base points on $C$. It is obvious that the base points of $|C|$ must be on $C$. Hence $|C|$ is base-point free.
\end{proof}

\begin{corollary}\label{cor: complete linear system morphism}
    Under the assumptions of \thref{prop: lifting sections} and the assumption that $C^{2} \geq 0$, the complete linear system $|C|$ defines a morphism $\varphi_{|C|}: S \rightarrow \mathbb{P}^{C^{2}+1}$. In particular, we have:
    \begin{itemize}
        \item If $C^{2} = 0$, then we have the contraction $\varphi_{|C|}: S \rightarrow \mathbb{P}^{1}$.
        \item If $C^{2} \geq 1$, then we have the birational contraction $\varphi_{|C|}: S \rightarrow \mathbb{P}^{C^{2}+1}$, and the image of $\varphi_{|C|}$ is a nondegenerate (that is, it is not contained in any hyperplane of $\mathbb{P}^{C^{2}+1}$) irreducible surface of degree $C^{2}$.
    \end{itemize}
\end{corollary}

\begin{proof}
    The existence of the morphism $\varphi_{|C|}$ directly follows from \thref{prop: lifting sections}. Since $C$ is rational, we have $h^{0}(C,\mathcal{O}_{C}(C)) = C^{2}+1$. Hence by the short exact sequence (\ref{equation: short exact sequence}) we have $h^{0}(S,\mathcal{O}_{S}(C)) = C^{2}+2$. Moreover, we have:
    \begin{itemize}
        \item If $C^{2} = 0$, then since $C$ is a fibre of $\varphi_{|C|}$, we conclude that $\varphi_{|C|}$ is a contraction.
        \item If $C^{2} \geq 1$, then since $\varphi_{|C|}$ is defined by a complete linear system, the image is nondegenerate. A nondegenerate surface in $\mathbb{P}^{C^{2}+1}$ has degree $\geq C^{2}$. Hence $\varphi_{|C|}$ is birational and the image of $\varphi_{|C|}$ has degree $C^{2}$.
    \end{itemize}
\end{proof}

\section{Gap numbers for coefficient sets}

  In this section, we discuss some general arithmetic properties of gap numbers and fix the coefficient set. First, we recall the definition of hyperstandard coefficient sets.
    
\begin{definition}[Hyperstandard sets]
    Let $\mathfrak{R} \subseteq [0,1]$. We define
    $$
    \Phi(\mathfrak{R}) = \{ 1 - \frac{r}{m} \mid r \in \mathfrak{R}, m \in \mathbb{Z}_{> 0} \}
    $$
    to be the \emph{hyperstandard multiplicities} associated to $\mathfrak{R}$. We usually assume $0,1 \in \mathfrak{R}$ without mention, so $\Phi(\mathfrak{R})$ includes $\Phi(\{ 0,1 \})$, the set of \emph{standard multiplicities}.
\end{definition}

\begin{definition}[Adjunction sets]
    Let $\mathcal{N}$ be a finite set of positive integers. We define
    \begin{align*}
    \Gamma(\mathcal{N}) &= \{1 - \frac{1}{m} + \frac{\sum \frac{l_{i}}{n_{i}+1}}{m} \mid m \in \mathbb{Z}_{>0}, n_{i} \in \mathcal{N}, 0 < l_{i} \leq n_{i} \} \cap [0,1] \\
    & = \Phi(\{1 - \sum \frac{l_{i}}{n_{i}+1} \mid n_{i} \in \mathcal{N}, 0 < l_{i} \leq n_{i} \} \cap [0,1] \}).
    \end{align*}
    We also define
    \begin{align*}
    \widetilde{\Gamma(\mathcal{N})} & = \{1 - \frac{1}{m} + \frac{\sum m_{i}b_{i}}{m} \mid m \in \mathbb{Z}_{>0}, m_{i} \in \mathbb{Z}_{\geq 0} ,b_{i} \in \Gamma(\mathcal{N}) \} \cap [0,1] \\
    & = \Phi(\{1 - \sum m_{i}b_{i} \mid m_{i} \in \mathbb{Z}_{\geq 0},b_{i} \in \Gamma(\mathcal{N}) \} \cap [0,1] \}).
    \end{align*}
    The set $\widetilde{\Gamma(\mathcal{N})}$ is exactly the coefficient set obtained from the difference in surface adjunction when the boundary coefficients lie in $\Gamma(\mathcal{N})$. Notice that $0,1 \in \Gamma(\mathcal{N})$ and $0,1 \in \widetilde{\Gamma(\mathcal{N})}$.
\end{definition}

\begin{definition}[gap numbers]\label{def: gap numbers}
    Let $\Gamma \subseteq [0,1]$ be a DCC set containing a positive real number. For any real number $u \geq 0$, by the DCC property, we can define
    $$
    \alpha_{u}(\Gamma) =\min \{ \sum\limits_{i=1}^{m} n_{i}b_{i} \mid m \in \mathbb{Z}_{>0}, n_{i} \in \mathbb{Z}_{\geq 0} , b_{i} \in \Gamma, \sum\limits_{i=1}^{m} n_{i}b_{i} > u \}.
    $$
\end{definition}

 Now we provide an algorithm to compute the gap numbers for a coefficient set $\Gamma \subset [0,1]$ with a single accumulation point 1. Notice that since $\Gamma \subset [0,1]$ has a single accumulation point 1, the set $\Gamma$ is a DCC set.

\begin{algorithm}[Computing gap numbers]\label{algorithm:gap number}
    Let $\Gamma \subset [0,1]$ be a finite set containing a positive number, or a DCC set with a single accumulation point 1, and let $u$ be a nonnegative real number.
    If $\Gamma$ is a finite set, then $\alpha_{u}(\Gamma)$ can be easily obtained from an upper bound on the multiplicities of the $\mathbb{Z}$-linear combination in \thref{def: gap numbers}. The following algorithm computes $\alpha_{u}(\Gamma)$ when $\Gamma$ is a DCC set with a single accumulation point 1.

    First, we can compute $\alpha_{u}(\Gamma)$ for any $0 \leq u < 1$. For this purpose, we find the smallest number $b \in \Gamma$ such that $b > u$. This is possible since the only accumulation point of $\Gamma$ is 1. Notice that $b < 1$. Then
    $$
    \alpha_{u}(\Gamma) = \min \{ \sum\limits_{i=1}^{m} n_{i}b_{i} \mid n_{i} \in \mathbb{Z}_{>0},\, b_{i} \in \Gamma,\, b_{i} \leq b < 1,\, \sum\limits_{i=1}^{m} n_{i}b_{i} > u \}.
    $$
    Thus the required minimum exists and is the minimum in a finite set of real numbers. By construction, in this case $\alpha_{u}(\Gamma) < 1$.

    Now suppose that $u \geq 1$. Then we can compute $\alpha_{u-1}(\Gamma)$. \thref{def: gap numbers} implies that $0 \leq u-\alpha_{u-1}(\Gamma)<1$. Thus we can compute $\alpha_{u-\alpha_{u-1}(\Gamma)}(\Gamma)$. The set
    $$
    \{ b_{i} \in \Gamma \mid b_{i} \leq \alpha_{u-\alpha_{u-1}(\Gamma)}(\Gamma)\}
    $$
    is finite because $\alpha_{u-\alpha_{u-1}(\Gamma)}(\Gamma) < 1$ and 1 is the only accumulation point of $\Gamma$. We contend that 
    $$
    \alpha_{u}(\Gamma) = \min \{ \sum\limits_{i=1}^{m} n_{i}b_{i} \mid n_{i} \in \mathbb{Z}_{>0},\, b_{i} \in \Gamma,\, b_{i} \leq \alpha_{u-\alpha_{u-1}(\Gamma)}(\Gamma),\, \sum\limits_{i=1}^{m} n_{i}b_{i} > u\}.
    $$
    Thus, to compute $\alpha_{u}(\Gamma)$, it is sufficient to find the minimum in a finite set of real numbers.

    To verify that the minimum gives $\alpha_{u}(\Gamma)$ we verify that for every presentation
    $$
    \alpha_{u}(\Gamma) = \sum\limits_{i=1}^{l}n_{i}b_{i},
    $$
    where $n_{i} \in \mathbb{Z}_{>0}$ and $b_{i} \in \Gamma$, we have
    $$
    b_{i} \leq \alpha_{u-\alpha_{u-1}(\Gamma)}(\Gamma)
    $$
    for every $i$. Indeed, suppose to the contrary and without loss of generality that $b_{1}>\alpha_{u-\alpha_{u-1}(\Gamma)}(\Gamma)$. Then
    $$
    \alpha_{u}(\Gamma) = b_{1} + (n_{1}-1)b_{1} + \sum\limits_{i=2}^{l}n_{i}b_{i} > u.
    $$
    Since $b_{1} \leq 1$, we have $(n_{1}-1)b_{1} + \sum\limits_{i=2}^{l}n_{i}b_{i} > u - 1$. By \thref{def: gap numbers} we have $(n_{1}-1)b_{1} + \sum\limits_{i=2}^{l}n_{i}b_{i} \geq \alpha_{u-1}(\Gamma)$. We have
    \begin{align*}
        \alpha_{u}(\Gamma) &= b_{1} + (n_{1}-1)b_{1} + \sum\limits_{i=2}^{l}n_{i}b_{i} \\
        & > \alpha_{u-\alpha_{u-1}(\Gamma)}(\Gamma) + \alpha_{u-1}(\Gamma) \\
        & > u-\alpha_{u-1}(\Gamma) + \alpha_{u-1}(\Gamma) = u,
    \end{align*}
    where $\alpha_{u-\alpha_{u-1}(\Gamma)}(\Gamma) + \alpha_{u-1}(\Gamma)$ is a smaller nonnegative $\mathbb{Z}$-linear combination $> u$, a contradiction to the definition of $\alpha_{u}(\Gamma)$.
\end{algorithm}

\begin{statement}[Computations with irrational numbers]
  Since the above calculation is usually assisted by computers, we have to pay attention to how to deal with irrational numbers. We can consider only computable irrational numbers in $\Gamma$ and $u$, more precisely, numbers with computable approximations by rational numbers. Then by induction, we can find an upper rational approximation $<1$ of $u-\alpha_{u-1}(\Gamma)$. This allows to find finitely many \emph{possible} presentations $\alpha_u(\Gamma)=\sum n_i b_i$, where $b_i\in \Gamma$ and $n_i$ are positive integers. Thus $\alpha_u(\Gamma)$ has computable approximations by rational numbers. However, in general, we cannot find \emph{exact} presentations in finitely many steps of computation.
 
\end{statement}

    For the proof of the main results, we need the following property of $\alpha_n(\Gamma)$: adding sufficiently large coefficients into the coefficient set $\Gamma$ does not change certain gap numbers.

\begin{proposition}\label{prop: gap number after adding large coefficients}
    Let $\Gamma \subseteq [0,1]$ be a DCC set containing $\{0,1\}$ and $n \geq 1$ be an integer. Let $\epsilon$ be a positive real number such that there exists an element $\gamma \in \Gamma$ such that $n - \alpha_{n-1}(\Gamma) < \gamma < 1 - \epsilon$. Then we have
    $$
    \alpha_{m}(\Gamma \cup \Phi(\{ \epsilon \})) = \alpha_{m}(\Gamma)
    $$
    for all nonnegative integers $m \leq n$.
\end{proposition}

    The proof uses the following lemma:

\begin{lemma}\label{lemma: gap number decrease}
    If $1 \in \Gamma$, then for any positive integers $m\leq n$, we have $\alpha_n(\Gamma)-n \leq \alpha_m(\Gamma)-m$.
\end{lemma}

\begin{proof}
    We have $\alpha_m(\Gamma)+n-m > n$. Since $1 \in \Gamma$, the number $\alpha_m(\Gamma)+n-m$ is also a nonnegative $\mathbb{Z}$-linear combination of the numbers in $\Gamma$. Hence by definition of $\alpha_n(\Gamma)$ we have $\alpha_n(\Gamma) \leq \alpha_m(\Gamma)+n-m$.
\end{proof}

\begin{proof}[Proof of \thref{prop: gap number after adding large coefficients}]
    We can use \thref{algorithm:gap number}. Indeed, if $m=0$, then
    $$
    \min\{ b \mid 0 < b \in \Gamma \} \leq \gamma < 1 - \epsilon.
    $$
    Hence
    $$
    \alpha_{0}(\Gamma) = \alpha_{0}(\Gamma \cup \Phi(\{ \epsilon \})) = \min\{ b \mid 0 < b \in \Gamma \}.
    $$ 
    Suppose by induction that the equality holds for some $m \leq n-1$. We claim that the equation holds for $m+1$. Indeed, write
    $$
    \alpha_{m+1}(\Gamma \cup \Phi(\{ \epsilon \})) = b + \sum\limits_{i=1}^{l} b_{i}
    $$
    where $b, b_{i}>0$ and $b,b_{i} \in \Gamma \cup \Phi(\{ \epsilon \})$ for all $i$. We claim that $b, b_{i} \in \Gamma$. Indeed, suppose to the contrary and without loss of generality that $0 < b \in \Phi(\{ \epsilon \})$. Then we have $b = 1 - \frac{\epsilon}{l} \geq 1 - \epsilon$ for some positive integer $l$. By our assumption, there exists an element $\gamma \in \Gamma$ such that $n - \alpha_{n-1}(\Gamma) < \gamma < 1 - \epsilon \leq b$. Notice that since $1 \in \Gamma$, we have $\alpha_{n}(\Gamma) - n \leq \alpha_{m}(\Gamma) - m$ for any positive integer $m \leq n$. Hence $\gamma > n - \alpha_{n-1}(\Gamma) \geq m+1 - \alpha_{m}(\Gamma)$ for any $m \leq n-1$. Hence we have
    $$
    m+1 < \gamma + \alpha_{m}(\Gamma) = \gamma + \alpha_{m}(\Gamma \cup \Phi(\{ \epsilon \})) \leq \gamma + \sum\limits_{i=1}^{l} b_{i} < b + \sum\limits_{i=1}^{l} b_{i}.
    $$
    This contradicts the fact that $\alpha_{m+1}(\Gamma \cup \Phi(\{ \epsilon \}))$ is the minimal nonnegative $\mathbb{Z}$-linear combination that is $>m+1$. Hence we must have $b, b_{i} \in \Gamma$ for all $i$, and
    $$
    \alpha_{m+1}(\Gamma \cup \Phi(\{ \epsilon \})) = \alpha_{m+1}(\Gamma).
    $$
\end{proof}

    Now we fix the coefficient set to study.

\begin{statement}[Set-up]\label{statement: setup}
    Let $\mathcal{N}_{1} = \{1,2,3,4,6\}$. In this article, we mainly consider the following coefficient set:
    \begin{align*}
    \Gamma(\mathcal{N}_{1}) = &\Phi(\{ 0, 1, \frac{1}{2},\frac{1}{3},\frac{2}{3},\frac{1}{4},\frac{3}{4},\frac{1}{5},\frac{2}{5},\frac{3}{5},\frac{4}{5},\frac{1}{6},\frac{1}{7},\frac{2}{7},
    \frac{3}{7},\frac{4}{7},\frac{5}{7},\frac{6}{7},\frac{1}{10},\frac{3}{10},\frac{1}{12},\frac{5}{12}, \frac{1}{14},\frac{3}{14},\frac{5}{14},\frac{1}{15},\frac{2}{15},\frac{4}{15},\frac{7}{15}, \\
    & \frac{1}{20},\frac{3}{20},\frac{7}{20},\frac{11}{20},\frac{1}{21},\frac{2}{21},\frac{4}{21},\frac{5}{21},\frac{8}{21},\frac{11}{21}, \frac{1}{28}, \frac{3}{28}, \frac{5}{28},\frac{9}{28},\frac{13}{28},\frac{17}{28},\frac{1}{35},\frac{2}{35},\frac{3}{35},\frac{4}{35},\frac{6}{35},\frac{8}{35},\frac{9}{35},\frac{11}{35},\frac{13}{35},\\
    &\frac{16}{35},\frac{18}{35},\frac{23}{35},\frac{1}{42},\frac{1}{60},\frac{13}{60},\frac{1}{70},\frac{11}{70},\frac{11}{84},\frac{23}{84},\frac{4}{105},\frac{13}{105},\frac{19}{105},
    \frac{34}{105},\frac{1}{140},\frac{9}{140},\frac{17}{140},\frac{29}{140},\frac{37}{140},\frac{57}{140}, \frac{31}{420}\}).
    \end{align*}
\end{statement}

    This coefficient set is useful in the effective boundedness of $\mathbb{R}$-complementary thresholds for surface pairs.

\begin{example}\label{exam: gap numbers}
    By \thref{algorithm:gap number}, we have $\alpha_{0}(\Gamma(\mathcal{N}_{1})) = \alpha_{0}(\widetilde{\Gamma(\mathcal{N}_{1})}) = \frac{1}{7}$, $\alpha_{1}(\Gamma(\mathcal{N}_{1})) = \alpha_{1}(\widetilde{\Gamma(\mathcal{N}_{1})}) = 1 + \frac{1}{252}$ and $\alpha_{2}(\Gamma(\mathcal{N}_{1})) = \alpha_{2}(\widetilde{\Gamma(\mathcal{N}_{1})}) = 2 + \frac{1}{127260}$.
\end{example}

\section{Gap numbers for thresholds}

The goal of this section is to find the gap number for $\mathbb{R}$-complementary threshold, log canonical threshold, and anticanonical threshold for the given coefficient set $\Gamma(\mathcal{N}_{1})$.

\begin{convention}
    Throughout this article, we always assume that $X$ is a normal projective surface.
\end{convention}

\subsection{The gap number of log canonical thresholds}

\begin{proof}[Proof of \thref{thm:lct gap}]
    Write $\delta = \min\{ \frac{1}{3}, \frac{\alpha_{0}(\widetilde{\Gamma})}{2} , \alpha_{1}(\widetilde{\Gamma}) - 1\}$. We have $0 \leq t \leq 1$. Without loss of generality, we can assume that $0 < t < 1$. Then there exists a prime divisor $E$ over $X$ such that $a(E;X,B+tF)=0$ and $\mathrm{Cent}_{X}(E) \subseteq \mathrm{Supp}F$. Assume first that $E$ is a prime divisor on $X$. Then we have $\mathrm{mult}_{E}B \in \Gamma$. If $\mathrm{mult}_{E}B > 0$, then $\mathrm{mult}_{E}B \geq \alpha_{0}(\Gamma) > \delta$ and $t \leq 1 - \delta$. If $\mathrm{mult}_{E}B = 0$, then $t = \frac{1}{\mathrm{mult}_{E}F}$, so either $t = 1$ or $t \leq \frac{1}{2} \leq 1 - \delta$.

    Now we can assume that $E$ is exceptional over $X$. Let $f: Y \rightarrow X$ be an extraction of the divisor $E$. Then
    $$
    K_{Y} + B_{Y} + tF_{Y} + E = f^{*}(K_{X} + B + tF)
    $$
    where $B_{Y}$, $F_{Y}$ are the strict birational transformations of $B$, $F$ on $Y$. We claim that $E$ is a nonsingular rational curve. Indeed, since $X$ is log canonical, according to the classification of log canonical singularities (cf. \cite{Alexeev92_lc_surface}), $E$ is either a rational curve or a nonsingular elliptic curve. Write $B = \sum\limits_{i=1}^{n}b_{i}D_{i}$. Applying \thref{prop: surface adjunction} to the pair $(Y,E + tF_{Y} + B_{Y})$ and the curve $E$, we have
    \begin{equation}\label{equation: lct adjunction}
    2g(\widetilde{E}) - 2 + l + \sum\limits_{i=1}^{n}w_{i}b_{i} + st +  \sum\limits_{i=1}^{u} \frac{m_{i}-1 + s_{i}t + \sum\limits_{j=1}^{n} r_{ij}b_{j}}{m_{i}} = 0
    \end{equation}
    for some nonnegative integers $l,w_{i},s,s_{i},r_{ij}$. Since the center of $E$ lies in the support of $F$, we have either $s>0$ or $s_{i}>0$ for some $i$. If $E$ is elliptic or $l \geq 2$, then $2g(\widetilde{E}) - 2 + l \geq 0$, and the left-hand side of Equation (\ref{equation: lct adjunction}) is greater than $0$, a contradiction. Hence $E$ is rational and $l \leq 1$.

    Now suppose to the contrary that $1 > t > 1 - \delta$. By the log canonical condition, we have $s_{i}t + \sum\limits_{j=1}^{n} r_{ij}b_{j} \leq 1$ and $s_{i} = 0 \text{ or } 1$. Moreover, if $s_{i} = 1$, then $\sum\limits_{j=1}^{n} r_{ij}b_{j}=0$. If $l = 1$, then the left-hand side of Equation (\ref{equation: lct adjunction}) becomes
    $$
    -1 + \left( \sum\limits_{i=1}^{n}w_{i}b_{i} + \sum\limits_{\substack{1 \leq i \leq u \\ s_{i}=0}} \frac{m_{i}-1 + \sum\limits_{j=1}^{n} r_{ij}b_{j}}{m_{i}} \right) + \left( \sum\limits_{\substack{1 \leq i \leq u \\ s_{i}=1}} \frac{m_{i} - 1 + t}{m_{i}} + st \right).
    $$
    Any contribution in the first bracket belongs to $\widetilde{\Gamma}$. Any contribution in the second bracket is $\geq t > 1 - \delta$. Moreover, there is at least one nonzero term from the second bracket. Hence we have 
    $$
    -1 + \left( \sum\limits_{i=1}^{n}w_{i}b_{i} + \sum\limits_{\substack{1 \leq i \leq u \\ s_{i}=0}} \frac{m_{i}-1 + \sum\limits_{j=1}^{n} r_{ij}b_{j}}{m_{i}} \right) + \left( \sum\limits_{\substack{1 \leq i \leq u \\ s_{i}=1}} \frac{m_{i} - 1 + t}{m_{i}} + st \right) > -1 + \alpha_{0}(\widetilde{\Gamma}) + 1 - \delta \geq 0,
    $$
    a contradiction. 
    
    If $l=0$, then $E$ is a nonsingular rational curve. Hence the sum can be split into two parts, and the equation becomes
    $$
    \left( \sum\limits_{i=1}^{n}w_{i}b_{i} + \sum\limits_{\substack{1 \leq i \leq u \\ s_{i}=0}} \frac{m_{i}-1 + \sum\limits_{j=1}^{n} r_{ij}b_{j}}{m_{i}} \right) + \left( \sum\limits_{\substack{1 \leq i \leq u \\ s_{i}=1}} \frac{m_{i} - 1 + t}{m_{i}} + st \right) = 2.
    $$
    Any contribution in the first bracket belongs to $\widetilde{\Gamma}$. Any contribution in the second bracket is $\geq t > 1 - \delta$. Moreover, there is at least one nonzero term in the second bracket. Since $t > 1 - \delta \geq \frac{2}{3}$, there are at most two nonzero terms in the second bracket.
    
    If there is only one nonzero term in the second bracket, then we have $\sum\limits_{i=1}^{n}w_{i}b_{i} + \sum\limits_{\substack{1 \leq i \leq u \\ s_{i}=0}} \frac{m_{i}-1 + \sum\limits_{j=1}^{n} r_{ij}b_{j}}{m_{i}} > 1$. Hence $ \sum\limits_{i=1}^{n}w_{i}b_{i} + \sum\limits_{\substack{1 \leq i \leq u \\ s_{i}=0}} \frac{m_{i}-1 + \sum\limits_{j=1}^{n} r_{ij}b_{j}}{m_{i}} \geq \alpha_{1}(\widetilde{\Gamma})$ and we have
    $$
    \left( \sum\limits_{i=1}^{n}w_{i}b_{i} + \sum\limits_{\substack{1 \leq i \leq u \\ s_{i}=0}} \frac{m_{i}-1 + \sum\limits_{j=1}^{n} r_{ij}b_{j}}{m_{i}} \right) + \left( \sum\limits_{\substack{1 \leq i \leq u \\ s_{i}=1}} \frac{m_{i} - 1 + t}{m_{i}} + st \right) > \alpha_{1}(\widetilde{\Gamma}) + 1 - \delta \geq \alpha_{1}(\widetilde{\Gamma}) + 1 - (\alpha_{1}(\widetilde{\Gamma}) - 1) = 2 
    $$
    a contradiction.

    If there are two nonzero terms in the second bracket, then we have
    $$
    \left( \sum\limits_{i=1}^{n}w_{i}b_{i} + \sum\limits_{\substack{1 \leq i \leq u \\ s_{i}=0}} \frac{m_{i}-1 + \sum\limits_{j=1}^{n} r_{ij}b_{j}}{m_{i}} \right) + \left( \sum\limits_{\substack{1 \leq i \leq u \\ s_{i}=1}} \frac{m_{i} - 1 + t}{m_{i}} + st \right) > \alpha_{0}(\widetilde{\Gamma}) + 2(1 - \delta) \geq \alpha_{0}(\widetilde{\Gamma}) + 2(1 - \frac{\alpha_{0}(\widetilde{\Gamma})}{2}) = 2 
    $$
    also a contradiction. Hence we conclude that either $t=1$ or $t \leq 1 - \delta$.
\end{proof}

\subsection{Gap number of anticanonical thresholds}

    In order to prove \thref{thm: act gap}, we first reduce the situation to several special cases, and compute the gap number for each case. Without loss of generality, we assume that $B \neq 0$. In fact, we can assume that $b_{1} > \frac{127259}{127260}$ from the very beginning and prove the theorem by contradiction. However, in order to make the conditions in each statement more precise, we minimize the assumption on $b_{1}$ in each statement.

\begin{lemma}[General reduction]\label{statement: bound genus of B1}
    Suppose that \thref{thm: act gap} holds under the following additional assumptions:
    \begin{enumerate}
        \item $X$ is of Fano type;
        \item $F = D_{1}$ is a prime divisor;
        \item $g(\widetilde{D}_{1}) = 0$ or $g(\widetilde{D}_{1}) = 1$.
    \end{enumerate}
    Then \thref{thm: act gap} holds in general.
\end{lemma}

\begin{proof}
    Since $\rho(X)=1$, either $K_{X} \sim_{\mathbb{Q}} 0$ or $X$ is of Fano type. If $K_{X} \sim_{\mathbb{Q}} 0$ then $B=0$, and there is nothing to prove. Hence we can assume that $X$ is of Fano type. In particular, we have $K_{X} + B \sim_{\mathbb{R}} 0$. 

    Assume temporarily that \thref{thm: act gap} has been proved when $F$ is prime. Consider the general case with 
    $$
    F = \sum\limits_{i=1}^{r} m_{i}F_{i}
    $$ 
    for some positive integer $r$ and $m_{i} \in \mathbb{Z}_{>0}$ for $i=1,\cdots,r$. If $m_{i} \geq 2$, then by the klt condition $b_{1} < \frac{1}{2}$ and there is nothing to prove. Hence we can assume that $F$ is reduced. Suppose to the contrary that $b_{1} > \frac{127259}{127260}$, then we can choose a positive rational number $\gamma \in \Gamma(\mathcal{N}_{1})$ such that $1 > \gamma > b_{1}$. By \thref{thm:lct gap} and \thref{corollary: lct gap special coefficients}, the pair $(X,F + \sum\limits_{i=2}^{n}b_{i}D_{i})$ is log canonical. Hence the pair $(X,\gamma F + \sum\limits_{i=2}^{n}b_{i}D_{i})$ is klt. The divisors $-(K_{X} + \frac{127259}{127260}F + \sum\limits_{i=2}^{n}b_{i}D_{i})$ and $K_{X} + \gamma F + \sum\limits_{i=2}^{n}b_{i}D_{i}$ are ample. Hence there is a positive integer $1 \leq r_{0} \leq r$ and a rational number $\frac{127259}{127260} < b'_{1} \leq \gamma$ such that
    $$
    (X, \gamma \sum\limits_{i=1}^{r_{0}-1}F_{i} + b'_{1}F_{r_{0}} + \frac{127259}{127260}\sum\limits_{i=r_{0}+1}^{r}F_{i} + \sum\limits_{i=2}^{n}b_{i}D_{i})
    $$
    is a klt 0-pair. However, applying \thref{thm: act gap} for the single prime divisor $F_{r_{0}}$ gives $b'_{1} \leq \frac{127259}{127260}$ a contradiction. Hence the inequality in \thref{thm: act gap} holds for a general effective Weil divisor $F$. 
    
    To prove the equality assertion in \thref{thm: act gap}, suppose that $b_{1}=\frac{127259}{127260}$. The klt condition implies that $F$ is reduced and does not contain $D_{i}$, $i=2,\cdots,n$. Choose one prime component $F_{1}$ of $F$ as the distinguished divisor. Every other component of $F$ has coefficient $\frac{127259}{127260} \in \Gamma(\mathcal{N}_{1})$, so the equality assertion for prime divisors applies to the same pair $(X,B)$. Each pair in \thref{example: optimal bound for act} and \thref{example: optimal bound for act 2} has exactly one component with coefficient $\frac{127259}{127260}$. Hence $F$ is prime and is that component.
    
    Now we can assume that $F=D_{1}$ is a prime divisor and write $B=\sum\limits_{i=1}^{n}b_{i}D_{i}$. Suppose, without loss of generality, that $b_{1} > \frac{12}{13}$. The pair $(X,b_{1}D_{1})$ has an $\mathbb{R}$-complement $(X,B)$. Hence
    $$
    \mathbb{R}\mathrm{-ct}(X,0;D_{1}) \geq b_{1} > \frac{12}{13}.
    $$
    By \thref{thm: gap number for trivial coefficients},
    $$
    \mathbb{R}\mathrm{-ct}(X,0; D_{1}) = 1.
    $$
    In particular, the pair $(X,D_{1})$ has an $\mathbb{R}$-complement. Since $\rho(X)=1$, we have
    $$
    (K_{X}+D_{1})D_{1} \leq 0.
    $$
    On the other side, applying \thref{prop: surface adjunction} to the pair $(X,D_{1})$, we have
    $$
    (K_{X}+D_{1}) D_{1} = 2g(\widetilde{D_{1}}) - 2 + \mathrm{deg} \mathrm{Diff}_{\widetilde{D_{1}}}(0) \geq 2g(\widetilde{D_{1}}) - 2
    $$
    where $\widetilde{D_{1}} \rightarrow D_{1}$ is the normalization of $D_{1}$. Hence we have $g(\widetilde{D_{1}})=0$ or $g(\widetilde{D_{1}})=1$.
    \end{proof}
    
    For the remainder of \S 4.2, we assume that $F=D_{1}$ is prime and write $B = \sum\limits_{i=1}^{n}b_{i}D_{i}$. We will discuss the 2 cases $g(\widetilde{D_{1}})=0$ and $g(\widetilde{D_{1}})=1$ separately in subsections \ref{subsection: D1 elliptic} and \ref{subsection: D1 rational}.

    \begin{statement}[A useful equation]\label{statement: equation from adjunction}
    Assume additionally that $b_{1}>\frac{251}{252}$. Applying \thref{thm:lct gap} to the pair $(X, \sum\limits_{i=2}^{n}b_{i}D_{i})$ and the divisor $D_{1}$, we get
    $$
    \mathrm{lct}(X, \sum\limits_{i=2}^{n}b_{i}D_{i}; D_{1}) = 1.
    $$
    
    In particular, the pair $(X,D_{1}+\sum\limits_{i=2}^{n}b_{i}D_{i})$ is log canonical. Applying \thref{prop: surface adjunction} to the pair $(X,D_{1}+\sum\limits_{i=2}^{n}b_{i}D_{i})$, we get
    $$
    (1-b_{1})D_{1}^{2} = 2g(\widetilde{D_{1}}) - 2 + \mathrm{deg} \mathrm{Diff}_{\widetilde{D_{1}}}(\sum\limits_{i=2}^{n}b_{i}D_{i}).
    $$
    We will frequently use this equation to estimate $b_{1}$.
    \end{statement}

    \subsubsection{\texorpdfstring{$D_{1}$}{D1} is an elliptic curve}\label{subsection: D1 elliptic}

    In this subsection, we prove the following:
    \begin{proposition}\label{prop: D1 elliptic}
        Under the assumptions of \thref{thm: act gap}, suppose additionally that $D_{1}$ is an elliptic curve. Then we have $b_{1} \leq 1 - \frac{1}{252}$.
    \end{proposition}

    \begin{proof}
        By \thref{statement: bound genus of B1} we have $\mathrm{Diff}_{\widetilde{D_{1}}}(0) = 0$, and $(X,D_{1})$ is an lc 0-pair. Hence $\widetilde{D_{1}} \simeq D_{1}$ is a nonsingular elliptic curve. By \thref{prop: surface lc singularity}, $X$ is nonsingular near $D_{1}$, and $D_{1}$ is a Cartier divisor.

        We claim that we have the short exact sequence
        $$
        0 \rightarrow \mathcal{O}_{X}(K_{X}) \rightarrow \mathcal{O}_{X}(K_{X} + D_{1}) \rightarrow \mathcal{O}_{D_{1}} \rightarrow 0.
        $$

        Indeed, since $X$ is nonsingular near $D_{1}$, the sequence is exact near $D_{1}$. On the other side, on $U = X \backslash D_{1}$, we simply have the trivial exact sequence $0 \rightarrow \mathcal{O}_{U}(K_{U}) \rightarrow \mathcal{O}_{U}(K_{U}) \rightarrow 0 \rightarrow 0$. By \thref{thm: generalized Kodaira vanishing} with $N \equiv -K_{X} + 0$, we have $h^{1}(X,\mathcal{O}_{X}(K_{X})) = 0$. Hence $h^{0}(X,\mathcal{O}_{X}(K_{X} + D_{1})) \neq 0$. Then there exists an effective divisor $E$ such that $E \sim K_{X} + D_{1} \sim_{\mathbb{R}}0$. Hence we conclude that $E = 0$ and $K_{X} + D_{1} \sim 0$. In particular, $K_{X} + D_{1}$ is a Cartier divisor, and the pair $(X,D_{1})$ is canonical. In particular, $X$ is also canonical. Let $\widetilde{X}$ be the crepant terminalization of $X$. Since $X$ is of Fano type, $\widetilde{X}$ is a smooth rational surface. Hence we have 
        $$
        D_{1}^2 = K_{X}^2 = K_{\widetilde{X}}^2 = 10 - \rho(\widetilde{X}) \leq 9.
        $$

        Moreover, by \thref{statement: equation from adjunction} we have
        $$
        (1-b_{1})D_{1}^{2} = \sum\limits_{i=2}^{n}n_{i}b_{i} \geq \frac{1}{7} > 0.
        $$

        Hence we have $b_{1} \leq 1 - \frac{1}{7 D_{1}^2} \leq 1 - \frac{1}{63}$. This contradicts our assumption in \thref{statement: equation from adjunction} that $b_{1} > 1 - \frac{1}{252}$. Hence $b_{1} \leq 1 - \frac{1}{252}$.
    \end{proof}

    \subsubsection{\texorpdfstring{$D_{1}$}{D1} is a rational curve} \label{subsection: D1 rational}

    In this subsection, we prove the following:
    \begin{proposition}\label{prop: D1 rational}
        Under the assumptions of \thref{thm: act gap}, suppose additionally that $D_{1}$ is a rational curve. Then we have $b_{1} \leq 1 - \frac{1}{127260}$. Moreover, the equality holds if and only if $(X,B)$ is one of \thref{example: optimal bound for act} and \thref{example: optimal bound for act 2}.
    \end{proposition}

    \begin{statement}\label{statement: lct assumption}
    Below we assume that at least $b_{1} > \frac{251}{252}$ (we will strengthen the assumption on $b_{1}$ in the following subsections). Then by \thref{statement: bound genus of B1} we have $\mathrm{deg}\mathrm{Diff}_{\widetilde{D_{1}}}(0) \leq 2$, and $(X,D_{1})$ is an lc pair. By \thref{prop: surface lc singularity}, $D_{1}$ is either a nonsingular rational curve, or a rational curve with 1 node.
    \end{statement}

\paragraph{$D_{1}$ nonsingular}
\hspace{2em}

    In this subsection, we assume that $D_{1}$ is a nonsingular rational curve.
    
    \begin{statement}[An equation to estimate $b_{1}$]\label{statement: bound by degree}
        Applying \thref{prop: surface adjunction} to the pair $(X,D_{1})$, we have
        $$
        \mathrm{deg}\mathrm{Diff}_{\widetilde{D_{1}}}(0) = l + \sum\limits_{i=1}^{u} \frac{m_{i}-1}{m_{i}} \leq 2,
        $$
        where $l$ and $m_{i}$ are described in \thref{prop: surface adjunction}. In particular, $D_{1}$ contains at most 4 singularities of $X$. Furthermore, by \thref{statement: equation from adjunction},
        $$
        (1-b_{1})D_{1}^{2} = - 2 + l + \sum\limits_{i=2}^{n}w_{i}b_{i} + \sum\limits_{i=1}^{u}\frac{m_{i}-1+\sum\limits_{j=2}^{n}r_{ij}b_{j}}{m_{i}} > 0.
        $$
        Hence
        $$
        b_{1} = 1 - \frac{- 2 + l + \sum\limits_{i=2}^{n}w_{i}b_{i} +  \sum\limits_{i=1}^{u}\frac{m_{i}-1+\sum\limits_{j=2}^{n}r_{ij}b_{j}}{m_{i}}}{D_{1}^2}.
        $$
        In the following subsections, we will use this equation to make a precise estimation on $b_1$, based on the singularities of $X$ lying on $D_{1}$.
    \end{statement}

    We need the following lemma about the complements on the Hirzebruch surfaces.

\begin{lemma}\label{lemma: complements on Hirzebruch surfaces}
    Fix a positive integer $m$. Let $E$ be the negative section of $\mathbb{F}_{m}$ and $(\mathbb{F}_{m}, cE + \sum c_{i}D_{i})$ be a klt 0-pair, where $D_{i}$ are distinct prime divisors on $\mathbb{F}_{m}$ which are not the negative section $E$. Then we have $1 > c \geq 1 - \frac{2}{m}$.
\end{lemma}

\begin{proof}
    Every prime divisor on $\mathbb{F}_{m}$ is linearly equivalent to an effective divisor of the form $uE + vF$, where 
    \begin{itemize}
        \item $E$ is the negative section of $\mathbb{F}_{m}$;
        \item $F$ is a fibre of $\mathbb{F}_{m}$;
        \item $u,v$ are integers such that one of the following happens:
        \begin{enumerate}
            \item $u=1, v=0$;
            \item $u=0, v=1$;
            \item $u,v \geq 1, v \geq mu$.
        \end{enumerate}
    \end{itemize}
    Write $D_{i} \sim u_{i}E + v_{i} F$. On Hirzebruch surfaces, we have $K_{\mathbb{F}_{m}} \sim -2E - (m+2)F$. Hence we have 
    \begin{gather*}
        c + \sum c_{i}u_{i} = 2, \\
        \sum c_{i}v_{i} = m+2.
    \end{gather*}
    Since $D_{i}$ are not the negative section $E$, we always have $v_{i} \geq mu_{i}$. Hence $m+2 \geq m (2-c)$ and $c \geq 1 - \frac{2}{m}$.
\end{proof}

\paragraph{$D_{1}$ contains no singularities of $X$}
\hfill

    In this subsection, we prove the following:
    \begin{lemma}\label{lemma: no singular point}
        Under the assumptions of \thref{thm: act gap}, suppose additionally that $D_{1}$ is a nonsingular rational curve and $D_{1}$ doesn't contain singularities of $X$. Then we have $b_{1} \leq 1 - \frac{1}{127260}$, and the equality holds only if $X \cong \mathbb{P}^2$ and every $D_{i}$ is a line on $\mathbb{P}^2$. Moreover, if $b_{1} < 1 - \frac{1}{127260}$, then $b_{1} \leq 1 - \frac{1}{104720}$.
    \end{lemma}

    \begin{proof}
    We put $d = D_{1}^{2}$. Applying \thref{cor: complete linear system morphism} to the surface pair $(X,0)$ and the curve $D_{1}$, we obtain that the complete linear system $|D_{1}|$ defines a birational contraction of $X$ to a degree $d$ surface in $P^{d+1}$. Since $\rho(X)=1$, this birational contraction must be an embedding. Hence $X$ is classified by \thref{thm: surface of minimal degree}. Recall that $\rho(X)=1$. Hence we have one of the following situations:
    \begin{enumerate}
        \item $X \cong \mathbb{P}^2$ and $d=1$;
        \item $X \cong \mathbb{P}^2$ and $d=4$;
        \item $X \cong \mathbb{P}(1,1,d)$ and $d \geq 2$.
    \end{enumerate}

    If $X \cong \mathbb{P}^2$ and $d=1$, then by \thref{statement: bound by degree} we have
    $$
    b_{1} = 1 - (- 2 + \sum\limits_{i=2}^{n}w_{i}b_{i}) \leq 1 - (-2 + \alpha_{2}(\widetilde{\Gamma(\mathcal{N}_{1})})) = 1 - \frac{1}{127260},
    $$
    and the equality holds if and only if \thref{example: optimal bound for act} holds. If $X \cong \mathbb{P}^2$ and $d=4$, then $D_{1}$ is a quadric and we can write
    $$
    \sum\limits_{i=2}^{n}d_{i}b_{i} = 3 - 2b_{1} > 1
    $$
    where $d_{i} = \frac{w_{i}}{2}$ is the degree of $D_{i}$ in $\mathbb{P}^2$. Hence $3 - 2b_{1} \geq \alpha_{1}(\Gamma(\mathcal{N}_{1})) = 1 + \frac{1}{252}$ and $b_{1} \leq 1 - \frac{1}{504}$.
    
    Hence we can assume that $X \cong \mathbb{P}(1,1,d)$ and $d \geq 2$. The minimal resolution $X^{min}$ of $X$ is $\mathbb{F}_{d}$, i.e., the Hirzebruch surface of degree $d$. Every prime divisor on $\mathbb{F}_{d}$ is linearly equivalent to an effective divisor of the form $eE + fF$, where 
    \begin{itemize}
        \item $E$ is the negative section of $\mathbb{F}_{d}$;
        \item $F$ is a fibre of $\mathbb{F}_{d}$;
        \item One of the following happens:
        \begin{enumerate}
            \item $e=1, f=0$;
            \item $e=0, f=1$;
            \item $e,f \geq 1, f \geq de$.
        \end{enumerate}
    \end{itemize}
    This allows us to list the prime divisors of $\mathbb{P}(1,1,d)$ of low degree as follows. Here the degree of a divisor is the image of its linear equivalence class under the isomorphism $\mathrm{Cl}(\mathbb{P}(1,1,d)) \cong \mathbb{Z}$ that sends ample classes to positive integers.

    \begin{enumerate}
        \item degree 1: the strict birational transformations of these prime divisors on $\mathbb{F}_{d}$ are fibres of the ruled surface $\mathbb{F}_{d} \rightarrow \mathbb{P}^1$. These prime divisors will be called \emph{vertical}. The other prime divisors will be called \emph{horizontal}.
        \item degree $\geq 2$ and $\leq d-1$: none.
        \item degree $d$: the strict birational transformations of these divisors on $\mathbb{F}_{d}$ are sections linearly equivalent to $E + dF$ in $\mathbb{F}_{d}$.
    \end{enumerate}

    Let $(X^{min} \cong \mathbb{F}_{d},cE + B^{min})$ be the crepant pullback of $(X,B)$, where the negative section $E$ is the exceptional divisor over $X$, $B = \sum\limits_{i=1}^{n} b_{i}D_{i}$, $B^{min} = \sum\limits_{i=1}^{n} b_{i}D_{i}^{min}$, and $D_{i}^{min}$ is the strict birational transformation of $D_{i}$ on $X^{min}$. Then 
    \begin{equation}\label{equation 1''}
    c= 1- \frac{2}{d} + \frac{1}{d}(\sum\limits_{\substack{1 \leq i \leq n \\ D_{i}^{min} \text{ horizontal}}} (f_{i}- e_{i}d ) b_{i} + \sum\limits_{\substack{1 \leq i \leq n \\ D_{i}^{min} \text{ vertical}}} b_{i}) < 1,
    \end{equation}
    where $e_{i}$ and $f_{i}$ are positive integers given by the linear equivalence $D_{i}^{min} \sim e_{i} E + f_{i}F$. Notice that $e_{1}=1$, $f_{1}=d$ and $f_{i} \geq e_{i}d$ for every $i$. To simplify the notation, we define
    $$
    b_{i}^{v} = \begin{cases}
        b_{i} & \text{if }D_{i}^{min} \text{ is vertical}, \\
        0 & \text{if }D_{i}^{min} \text{ is horizontal},
    \end{cases}
    $$
    and
    $$
    b_{i}^{h} = \begin{cases}
    0 & \text{if }D_{i}^{min} \text{ is vertical}, \\
    b_{i} & \text{if }D_{i}^{min} \text{ is horizontal}.
    \end{cases}
    $$
    
    The pair $(\mathbb{F}_{d},cE + B^{min})$ is still a klt 0-pair. By adjunction on general fibres, we have:

    \begin{equation}\label{equation 1}
        c + b_{1} + \sum\limits_{i = 2}^{n} e_{i}b_{i}^{h} = 2.
    \end{equation}

    In particular, we have $\sum\limits_{i = 2}^{n} e_{i}b_{i}^{h} > 0$. Hence $\sum\limits_{i = 2}^{n} e_{i}b_{i}^{h} \geq \alpha_{0}(\Gamma(\mathcal{N}_{1})) = \frac{1}{7}$. By \thref{lemma: complements on Hirzebruch surfaces},
    \begin{align*}
    b_{1} &= 2 - c - \sum\limits_{i=2}^{n} e_{i}b_{i}^{h} \\
    & \leq 1 + \frac{2}{d} - \sum\limits_{i=2}^{n} e_{i}b_{i}^{h} \\
    &\leq \frac{6}{7} + \frac{2}{d}.
    \end{align*}
    
    If $d \geq 15$, then we have $b_{1} \leq \frac{104}{105}$, contradicting the assumption in \thref{statement: lct assumption} that $b_{1} > \frac{251}{252}$. Hence we can assume that $d \leq 14$. To solve Equation (\ref{equation 1}), we can assume without loss of generality that $e_{i}=1$ and $f_{i}=d$. Indeed, a linear equivalence class of the form $e_{i} E + f_{i}F$ where $f_{i} \geq e_{i}d$ can be written as a nonnegative $\mathbb{Z}$-linear combination of the linear equivalence class of $E + dF$ and the linear equivalence class of $F$. In addition, all of these classes are linearly effective and free. Thus the computed numerical bound for $b_{1}$ still holds. Notice that the number of horizontal and vertical components can be changed, and the number $n$ of components of $B$ can be changed as well. Equation (\ref{equation 1}) can be rewritten as
    \begin{equation}\label{equation 1'}
        d \sum\limits_{i=1}^{n} b_{i}^{h} + \sum\limits_{i=1}^{n} b_{i}^{v} = d+2.
    \end{equation}
    Since $c < 1$, we have $\sum\limits_{i=1}^{n} b_{i}^{v} < 2$. We use a computer to compute all possible solutions. We briefly explain the algorithm. Indeed, Equation (\ref{equation 1}) can be rewritten as
    $$
    d \sum\limits_{i=2}^{n} b_{i}^{h} + \sum\limits_{i=2}^{n} b_{i}^{v} = 2 + d(1 - b_{1}).
    $$
    
    We can compute a ``weighted'' gap number on the left-hand side, i.e., the least linear combination $>2$ of coefficients in $\Gamma(\mathcal{N}_{1})$, such that at least one coefficient has weight $d$. Then $d(1 - b_{1})$ is bounded from below, and we obtain an upper bound of $b_{1}$. For instance, if $d = 2$, then the weighted gap number is given by $2 \times \frac{1}{7} + \frac{29}{40} + \frac{1018}{1029} = 2 + \frac{1}{41160}$, and we have $b_{1} \leq 1 -\frac{1}{41160 \times 2} = \frac{82319}{82320}$. We give a table for $b_{1}$ for all $d \leq 14$ as follows:

\begin{longtable}{|C|C|C|C|C|C|C|C|C|C|C|C|C|C|C|}
\hline
    d & 1 & 2 & 3 & 4 & 5 & 6 & 7 & 8 & 9 & 10 & 11 & 12 & 13 & 14 \\
    \hline
    b_{1} & \frac{127259}{127260} & \frac{82319}{82320} & \frac{21419}{21420} & \frac{26319}{26320} & \frac{8924}{8925} & \frac{3779}{3780} & \frac{1763}{1764} & \frac{2015}{2016} & \frac{839}{840} & \frac{839}{840} & \frac{1537}{1540} & \frac{251}{252} & \frac{453}{455} & \frac{97}{98} \\
    \hline
\end{longtable}

Hence in all cases we have $b_{1} \leq \frac{127259}{127260}$, and the equality holds only if $d = 1$. Moreover, if $b_{1} < \frac{127259}{127260}$, then $b_{1} \leq 1 - (\alpha_{2+\frac{1}{127260}}(\Gamma(\mathcal{N}_{1}))-2) = 1 - \frac{1}{104720}$, and again the equality holds only if $d = 1$.
\end{proof}

\paragraph{$D_{1}$ contains singularities of $X$}
\hfill

    Now we will prove \thref{thm: act gap} in the case where $D_{1}$ contains some singularities of $X$.
    
    \begin{lemma}\label{lemma: D1 contains singularities}
        Under the assumptions of \thref{thm: act gap}, suppose additionally that $D_{1}$ is a nonsingular rational curve and $D_{1}$ contains some singularities of $X$. Then we have $b_{1} \leq 1 - \frac{1}{127260}$. Moreover, equality holds if and only if $(X,B)$ is as in \thref{example: optimal bound for act 2}.
    \end{lemma}
    
    To establish this, first, we give a description of the singularities that $D_{1}$ passes through.

\begin{lemma}[Classification of singularities]\label{statement: local classification}
    Under the assumptions of \thref{thm: act gap}, suppose additionally that $D_{1}$ is a nonsingular rational curve and $D_{1}$ contains some singularities of $X$. Let $\pi: X^{min} \rightarrow X$ be the minimal resolution of $X$. Let $P \in D_{1}$ be a singularity of $X$, $E_{1},E_{2},\cdots,E_{s_{P}}$ be the exceptional divisors on $X^{min}$ centered at $P$, and $\Gamma$ be the dual graph of $P$. Let $E$ be an exceptional divisor on $X^{min}$ centered at $P$ that intersects $D_{1}^{min}$ and $v$ the corresponding vertex of $E$ in $\Gamma$. Then one of the following holds:
    \begin{enumerate}
    \item The dual graph $\Gamma$ is of type $\mathbf{A}$, and the exceptional divisor $E$ is at the end of $\Gamma$ intersecting $D_{1}^{min}$. This is Case (3) in \thref{prop: surface adjunction}. The dual graph is given by Figure \thref{figure: dual graph 3}, where $C = D_{1}^{min}$. Moreover, we have
    $$
    a(E;X,B) = \frac{q(1-b_{1}) + \alpha}{m}
    $$
    where $m = \Delta(\Gamma)$ is the Cartier index of $D_{1}$ at $P$, $q = \Delta(\Gamma-v_{E})$, and $1 \geq \alpha = 1 - \sum\limits_{i=2}^{n} m_{i}b_{i} \geq 0$ for some $m_{i} \in \mathbb{Z}_{\geq 0}$. The coefficient of $B^{min}$ at $E$ is
    $$
    b_{E} = 1 - a(E;X,B) = \frac{m-q + q b_{1} - \alpha}{m} > b_{1} - \frac{1}{m}.
    $$

    Notice that for fixed positive integers $m$ and $q<m$, the set
    $$
    \{ \frac{m-q + q b_{1} - (1 - \sum\limits_{i=2}^{n} r_{i}b_{i})}{m} \in (0,1) \mid r_{i} \in \mathbb{Z}_{\geq 0}, b_{i} \in \Gamma(\mathcal{N}_{1}), 0 \leq \sum\limits_{i=2}^{n} r_{i}b_{i} \leq 1 \}
    $$
    satisfies DCC, and the only accumulation point is $1 - \frac{q}{m}(1-b_{1}) > b_{1}$. Additionally, extracting $E$ reduces the self intersection of $D_{1}$ by $\frac{q}{m} < 1$.

    \item The dual graph $\Gamma$ is of type $\mathbf{D}$, and the exceptional divisor $E$ is the ending point of the chain part of $\Gamma$. This is Case (2) in \thref{prop: surface adjunction}. The dual graph is given by Figure \thref{figure: dual graph 2}, where $C = D_{1}^{min}$. Moreover, we have
    $$
    a(E;X,B) = \frac{q(1-b_{1})}{m}
    $$
    where $m = \Delta(\Gamma)$ and $q = \Delta(\Gamma-v_{E})$. Then extracting $E$ reduces the self intersection of $D_{1}$ by $\frac{q}{m} \leq 1$. This case can be treated formally as the previous case with $\alpha = 0$.

    \item The dual graph $\Gamma$ is of type $\mathbf{A}_3$ where the vertices have weights $(2,w,2)$ for some positive integer $w \geq 2$, and the exceptional divisor $E$ is in the middle of the chain. This is Case (2) in \thref{prop: surface adjunction}. The dual graph is given by Figure \thref{figure: dual graph 2}, where $C = D_{1}^{min}$. By Equation (\ref{equation: log discrepancy}), we have
    $$
    a(E;X,B) = \frac{1-b_{1}}{w-1} \leq 1 - b_{1}.
    $$
    Moreover, extracting $E$ reduces the self intersection of $D_{1}$ by $\frac{1}{w - 1} \leq 1$, and the equality holds if and only if $w=2$. This case can be treated formally the same as the previous case, since type $\mathbf{D}_{3}$ is the same as type $\mathbf{A}_{3}$.
    \end{enumerate}
\end{lemma}

\begin{proof}
    Since $(X,B)$ is klt and $D_{1}$ is nonsingular, only Case (2) and Case (3) in \thref{prop: surface adjunction} can happen. In particular, the strict birational transformation $D_{1}^{min}$ of $D_{1}$ on $X^{min}$ intersects each connected component of the exceptional locus at exactly one point with multiplicity 1. In particular, the exceptional divisor $E$ is unique. Applying \thref{prop:formula for ld} to the exceptional divisor $E$, we have
    \begin{equation}\label{equation: log discrepancy}
    a(E;X,B) = \frac{\sum\limits_{i=1}^{s_{P}} \Delta(\Gamma - \mathrm{Path}(v,v_{i}))c_{i}}{\Delta(\Gamma)},
    \end{equation}
    where
    $$
    c_{i} = 2 - (f^{-1}_{*}B + \sum\limits_{\substack{1 \leq j \leq s_{P} \\ j \neq i}}E_{j})\cdot E_{i}.
    $$
    Notice that $B$ is effective, and $c_{i}>0$ only if $v_{i}$ is an ending point of the graph $\Gamma$, that is, it has only one adjacent vertex. Since $(X,B)$ is klt, the log discrepancy $a(E_{i};X,B) > 0$ for every $i$. This is possible only in the following situations:
    
    \begin{enumerate}
    \item The dual graph $\Gamma$ is of type $\mathbf{A}$, and the exceptional divisor $E$ is at the end of $\Gamma$ intersecting $D_{1}^{min}$. This is Case (3) in \thref{prop: surface adjunction}. The dual graph is given by Figure \thref{figure: dual graph 3}, where $C = D_{1}^{min}$. By Equation (\ref{equation: log discrepancy}), we have
    $$
    a(E;X,B) = \frac{\Delta(\Gamma - v)(1-b_{1}) + \alpha}{\Delta(\Gamma)}
    $$

    where $1 \geq \alpha = 1 - \sum\limits_{i=2}^{n} m_{i}b_{i} \geq 0$ for some $m_{i} \in \mathbb{Z}_{\geq 0}$. Put $m=\Delta(\Gamma)$ and $q=\Delta(\Gamma - v)$. By \thref{prop: surface adjunction}, $m$ is the Cartier index of $D_{1}$ at $P$. Moreover, the number $\frac{q}{m}$ is the multiplicity of $D_{1}$ in $E$. The coefficient of $B^{min}$ at $E$ is
    $$
    b_{E} = 1 - a(E;X,B) = \frac{m-q + q b_{1} - \alpha}{m} > b_{1} - \frac{1}{m}.
    $$
    Let $\pi_{E}: Y \rightarrow X$ be the extraction of the divisor $E$. Then we have $\pi_{E}^{*}D_{1} = D_{1,Y} + \frac{q}{m}E_{Y}$ and $E_{Y} \cdot D_{1,Y} = 1$. Hence we have 
    $$
    D_{1}^{2} = \pi_{E}^{*}D_{1} \cdot D_{1,Y} = D_{1,Y}^{2} + \frac{q}{m}E_{Y} \cdot D_{1,Y} = D_{1,Y}^{2} + \frac{q}{m}.
    $$

    \item The dual graph $\Gamma$ is of type $\mathbf{D}$, and the exceptional divisor $E$ is the ending point of the chain part of $\Gamma$. This is Case (2) in \thref{prop: surface adjunction}. The dual graph is given by Figure \thref{figure: dual graph 2}, where $C = D_{1}^{min}$. By Equation (\ref{equation: log discrepancy}), we have
    $$
    a(E;X,B) = \frac{\Delta(\Gamma - v)(1-b_{1})}{\Delta(\Gamma)}.
    $$
    Let $m=\Delta(\Gamma)$ and $q=\Delta(\Gamma - v)$. Similar to extracting $E$ reduces the self intersection of $D_{1}$ by $\frac{q}{m} \leq 1$.

    \item The dual graph $\Gamma$ is of type $\mathbf{A}_3$, and the exceptional divisor $E$ is in the middle of the chain. Since $(X,D_{1})$ is log canonical, the vertices have weights $(2,w,2)$, where $w$ is the weight at the middle, and $D_{1}^{min}$ intersects the exceptional divisor in the middle. This is Case (2) in \thref{prop: surface adjunction}. The dual graph is given by Figure \thref{figure: dual graph 2}, where $C = D_{1}^{min}$. By Equation (\ref{equation: log discrepancy}), we have
    $$
    a(E;X,B) = \frac{1-b_{1}}{w-1} \leq 1 - b_{1}.
    $$
    
    Moreover, extracting $E$ reduces the self intersection of $D_{1}$ by $\frac{1}{w - 1} \leq 1$, and the equality holds if and only if $w=2$.
    \end{enumerate}
\end{proof}

    Next, we need to control the number of singularities of $X$ in $D_{1}$.

\begin{lemma}\label{lemma: at most 4 singularities on D1}
    Under the assumptions of \thref{thm: act gap}, suppose additionally that $D_{1}$ is a nonsingular rational curve. Then $D_{1}$ contains at most 4 singularities of $X$. Moreover, we have 
    $$
    l + \sum\limits_{i=1}^{u}\frac{m_{i}-1}{m_{i}} \leq 2
    $$
    where $l,u,m_{i}$ are defined as in \thref{prop: surface adjunction}.
\end{lemma}

\begin{proof}
    By \thref{statement: bound genus of B1}, we have
    $$
    0 \geq (K_{X}+D_{1})\cdot D_{1} \geq -2 + l + \sum\limits_{i=1}^{u} \frac{m_{i}-1}{m_{i}}
    $$
    If $X$ is singular at $\nu(P_{i})$, then we have $m_{i} \geq 2$ and $\frac{m_{i}-1}{m_{i}} \geq \frac{1}{2}$. Hence $D_{1}$ contains at most 4 singularities of $X$.
\end{proof}

\begin{remark}
    It is known that a del Pezzo surface with log terminal singularities and Picard number 1 has at most 4 singularities. For reference, see \cite[Theorem 1.2]{Belousov08}, \cite[Theorem 1.1]{Belousov09}, \cite[Theorem 1.2]{LX25}.
\end{remark}

    By \thref{lemma: at most 4 singularities on D1}, the proof of \thref{lemma: D1 contains singularities} can be divided into 4 parts, where $D_{1}$ contains 1,2,3,4 singularities of $X$, respectively. In the remainder of this subsection, we will make some common constructions for all the 4 cases. The proof of each case will appear in the subsequent subsections (see \thref{lemma: D1 contains 1 singularity}, \thref{lemma: D1 contains 2 singularities}, \thref{lemma: D1 contains 3 singularities} and \thref{lemma: D1 contains 4 singularities}).

\begin{construction}[A special partial resolution]\label{construction: extract to make Cartier}
    Notations as above. Consider the contraction $\pi': X^{min} \rightarrow X'/X$ which contracts all the exceptional divisors on $X^{min}$ over $X$ that do not intersect $D_{1}^{min}$. Let $f':X' \rightarrow X$ be the natural birational contraction. For every divisor $D_{i}$, let $D'_{i}$ be the strict birational transformation of $D_{i}$ on $X'$. By construction, $X'$ is nonsingular near $D'_{1}$. Notice that for every singularity $P$ of $X$, we have
    \begin{enumerate}[align=left]
        \item there is exactly 1 prime divisor on $X'$ centered at $P$ if $P \in D_{1}$;
        \item the morphism $f'$ is an isomorphism near $P$ if $P \notin D_{1}$.
    \end{enumerate}
    Hence $X'$ is the minimal extraction of $X$ such that it is nonsingular near the strict birational transformation of $D_{1}$. By \thref{lemma: at most 4 singularities on D1}, $D_{1}$ contains at most 4 singularities of $X$. Hence the morphism $f'$ extracts at most 4 prime divisors. Let $P_{1},\cdots, P_{u}$ be the singularities of Case (3) in \thref{prop: surface adjunction} of $X$ contained in $D_{1}$, and $P_{u+1},\cdots,P_{u+l}$ be the singularities of Case (2) in \thref{prop: surface adjunction} of $X$ contained in $D_{1}$. Let $E_{i}$ be the divisor extracted from $P_{i}$. By \thref{statement: local classification} and \thref{lemma: at most 4 singularities on D1}, we have
    \begin{equation}\label{equation: d <= d'+2}
    d' = (D_{1}')^2 \geq d - l - \sum\limits_{i=1}^{u}\frac{q_{i}}{m_{i}} \geq d - l -\sum\limits_{i=1}^{u}\frac{m_{i}-1}{m_{i}} \geq d-2.
    \end{equation}
    Since $d > 0$, we have $d' \geq -1$.
\end{construction}

\begin{statement}[Contraction defined by $|D'_{1}|$]\label{statement: contraction defined by D'1}
    Notation as above. Write
    $$
    f^{\prime *}K_{X} = K_{X'} + \sum\limits_{i=1}^{u+l} (1 - a(E_{i};X,0))E_{i}.
    $$
    Then the surface pair $(X',\sum\limits_{i=1}^{u+l} (1 - a(E_{i};X,0))E_{i})$ is klt and weak Fano. If $d' \geq 0$, then by \thref{cor: complete linear system morphism} on the surface pair $(X',\sum\limits_{i=1}^{u+l} (1 - a(E_{i};X,0))E_{i})$ and the curve $D'_{1}$, the complete linear system $|D'_{1}|$ defines a morphism. We divide the problem into several different cases:
    \begin{enumerate}[align=left, label=Case \arabic*:]
        \item $d' = -1$.
        \item $d' = 0$. In this case, $|D'_{1}|$ has dimension 1 and $|D'_{1}|$ defines a contraction $X' \rightarrow \mathbb{P}^1$.
        \item $d' \geq 1$. In this case, $|D'_{1}|$ defines a birational contraction from $X'$ to a nondegenerate surface $X_{d'}$ of degree $d'$ in $P^{d'+1}$. By \thref{thm: surface of minimal degree}, the surface $X_{d'}$ is one of the following:
        \begin{enumerate}
             \item The Hirzebruch surface $\mathbb{F}_{d'-2s}$ embedded in $P^{d'+1}$ as a rational normal scroll for some integers $d' \geq 2$ and $1 \leq s \leq \lfloor \frac{d'}{2} \rfloor$.
            \item The projective plane $\mathbb{P}^2$ embedded in $\mathbb{P}^5$ by the degree 2 Veronese embedding. In this case $d'=4$.
            \item The projective plane $\mathbb{P}^2$ in $\mathbb{P}^2$. In this case $d' = 1$.
            \item The weighted projective space $\mathbb{P}(1,1,d')$ where $d' \geq 2$, embedded in $P^{d'+1}$ as a projective cone over a nondegenerate curve of degree $d'$ in $P^{d'}$. 
        \end{enumerate}
    \end{enumerate}
\end{statement}

    Before we discuss the above cases one-by-one, we use the following lemma to simplify our discussion.
    
\begin{lemma}[Reduction]\label{lemma: reduction}
    Notation as in \thref{statement: contraction defined by D'1}. Then to prove \thref{lemma: D1 contains singularities}, it suffices to consider the following cases:
    \begin{enumerate}
        \item Case 1;
        \item Case 2;
        \item Case 3(c);
        \item Case 3(d) where
        \begin{enumerate}
            \item $2 \leq d' \leq 14$;
            \item The strict birational transformation of $D_{i}$ on $X^{min}_{d'} \cong \mathbb{F}_{d'}$ is a prime divisor linearly equivalent to $E + d'F$, or $F$, or the zero divisor. Here $E$ is the negative section of $\mathbb{F}_{d'}$ and $F$ is a general fibre of the ruled surface.
        \end{enumerate}
    \end{enumerate}
\end{lemma}

\begin{proof}
    We need to solve Case 3(a), Case 3(b) and Case 3(d) where $d' \geq 15$ or the strict birational transformation of $D_{i}$ on $X^{min}_{d'} \cong \mathbb{F}_{d'}$ is not of the given form.

    \textbf{Case 3(a):} We show that this case can only happen when $b_{1} < 1 - \frac{1}{127260}$. Indeed, in this case, there is a unique ruling on the rational normal scroll in $\mathbb{P}^{d'+1}$ such that the general fibre is a line in $\mathbb{P}^{d'+1}$, unless $d'=2$ and $s=1$. We set up the following notations:
    \begin{itemize}
        \item If $d'>2s$, let $E$ be the negative section of $\mathbb{F}_{d'-2s}$ and $F$ be a general fibre.
        \item If $d'=2s$, we choose the ruling of the scroll such that the general fibre is a line in $P^{d'+1}$. The choice of the ruling is unique unless $d'=2$ and $s=1$. Let $E$ be a zero section of this ruling, and $F$ be a general fibre of this ruling.
    \end{itemize}
    The divisor $D'_{1}$ is the pullback of a divisor in $\mathbb{F}_{d'-2s}$ which is linearly equivalent to $E + (d'-s)F$. For any divisor $E_{i}$ on $X'$ exceptional over $X$, we have $D'_{1} \cdot E_{i} = 1$. Hence the strict birational transformation of $E_{i}$ on $\mathbb{F}_{d'-2s}$ is linearly equivalent to either $E$ or $F$. We claim that there exists a divisor $E_{i_{0}}$ on $X'$ exceptional over $X$ such that the strict birational transformation of $E_{i_{0}}$ on $\mathbb{F}_{d'-2s}$ is linearly equivalent to $E$. Indeed, suppose to the contrary that the strict birational transformation of every $E_{i}$ on $\mathbb{F}_{d'-2s}$ is a fibre. Let $F$ be a general fibre of $\mathbb{F}_{d'-2s}$ and $F_{X}$ be its strict birational transformation on $X$. Then $F_{X}^{2} = 0$, a contradiction. This completes the proof of the claim. In particular, we have $s=1$, and the strict birational transformation of $E_{i_{0}}$ on $\mathbb{F}_{d'-2}$ is $E$ (or linearly equivalent to $E$ if $d'=2$). We can write $D'_{1} \sim D''_{1} + D'''_{1}$, where $D''_{1}$ is the pullback of a general divisor linearly equivalent to $E + (d'-2)F$ and $D'''_{1}$ is the pullback of a general divisor linearly equivalent to $F$. Consider the pair 
    $$
    (X', b_{1}D''_{1} + b_{1}D'''_{1} + \sum\limits_{i=2}^{n}b_{i}D'_{i} + \sum\limits_{i=1}^{u+l} (1 - a(E_{i};X,B))E_{i}).
    $$
    Notice that the linear systems $|E + (d'-2)F|$ and $|F|$ are base-point free. Hence the new pair is still a klt 0-pair. Let $D''_{1,X}$ and $D'''_{1,X}$ be the strict birational transformations of $D''_{1}$ and $D'''_{1}$ on $X$. The pair $(X,b_{1}D''_{1,X} + b_{1}D'''_{1,X} + \sum\limits_{i=2}^{n}b_{i}D_{i})$ is also a klt 0-pair. By \thref{thm:lct gap}, the pair $(X,D''_{1,X} + D'''_{1,X} + \sum\limits_{i=2}^{n}b_{i}D_{i})$ is log canonical. Applying \thref{statement: bound by degree} to the pair $(X,D''_{1,X} + b_{1}D'''_{1,X} + \sum\limits_{i=2}^{n}b_{i}D_{i})$, we have
    $$
    (1-b_{1})(D''_{1,X})^{2} = - 2 + b_{1} + \sum\limits_{i=2}^{n}w'_{i}b_{i} + \sum\limits_{i=1}^{u}\frac{m_{i}-1+\sum\limits_{j=2}^{n}r'_{ij}b_{j}}{m_{i}} > 0.
    $$
    We have $\sum\limits_{i=2}^{n}w'_{i}b_{i} + \sum\limits_{i=1}^{u}\frac{m_{i}-1+\sum\limits_{j=2}^{n}r'_{ij}b_{j}}{m_{i}} > 1$. Hence, by \thref{exam: gap numbers}, we have $\sum\limits_{i=2}^{n}w'_{i}b_{i} + \sum\limits_{i=1}^{u}\frac{m_{i}-1+\sum\limits_{j=2}^{n}r'_{ij}b_{j}}{m_{i}} \geq 1 + \frac{1}{252}$ and 
    $$
    (1-b_{1})(D''_{1,X})^{2} \geq b_{1} + \frac{1}{252} - 1.
    $$
    Now suppose to the contrary that $b_{1} \geq 1 - \frac{1}{127260}$. Then we have 
    $$
    d = D_{1}^{2} > (D''_{1,X})^{2} \geq -1 + \frac{1}{252(1-b_{1})} \geq 504.
    $$
    Hence $d' \geq d-2 > 502$. Our arguments are similar to the proof of \thref{lemma: no singular point}. For simplicity, put
    $$
    b'_{i} = \begin{cases}
        b_{i} & \text{if the strict birational transformation of }D_{i} \text{ on } X_{d'} \cong \mathbb{F}_{d'-2} \text{ is a prime divisor}, \\
        0 & \text{otherwise},
    \end{cases}
    $$
    and
    $$
    b_{i}^{\prime h} = \begin{cases}
    b_{i} & \text{if the strict birational transformation of }D_{i} \text{ on } X_{d'} \cong \mathbb{F}_{d'-2} \text{ is a horizontal prime divisor}, \\
    0 & \text{otherwise}.
    \end{cases}
    $$
    Denote by $D_{i,X_{d'}}$ the strict birational transformation of $D_{i}$ on $X_{d'} \cong \mathbb{F}_{d'-2}$, which is either a prime divisor (when $b'_{i}>0$) or the zero divisor (when $b'_{i}=0$). Write $D_{i,X_{d'}} \sim e_{i}E + f_{i}F$. By \thref{statement: local classification}, the crepant transformation of $(X,B)$ on $X_{d'} \cong \mathbb{F}_{d'-2}$ is a klt 0-pair. Similar to Equation (\ref{equation 1}) in the proof of \thref{lemma: no singular point}, we have
    \begin{equation}
    e + b_{1} + \sum\limits_{i = 2}^{n} e_{i}b_{i}^{\prime h} = 2.
    \end{equation}
    
    In particular, we have $\sum\limits_{i = 2}^{n} e_{i}b_{i}^{\prime h} > 0$. Hence $\sum\limits_{i = 2}^{n} e_{i}b_{i}^{\prime h} \geq \alpha_{0}(\Gamma(\mathcal{N}_{1})) = \frac{1}{7}$. By \thref{lemma: complements on Hirzebruch surfaces},
    \begin{align*}
    b_{1} &= 2 - e - \sum\limits_{i=2}^{n} e_{i}b_{i}^{\prime h} \\
    &\leq 1 + \frac{2}{d'-2} - \sum\limits_{i=2}^{n} e_{i}b_{i}^{\prime h} \\
    &\leq \frac{6}{7} + \frac{2}{d'-2} \\
    & < \frac{1507}{1750},
    \end{align*}
    a contradiction.
    
    \textbf{Case 3(b):} We show that this case is impossible. Indeed, in this case, $D'_{1}$ is the pullback of a conic in $\mathbb{P}^2$. However, for any divisor $E_{i}$ on $X'$ exceptional over $X$, we have $D'_{1} \cdot E_{i} = 1$, a contradiction.

    \textbf{Case 3(d):} We prove \thref{lemma: D1 contains singularities} for $d' \geq 15$. Indeed, the minimal resolution of $\mathbb{P}(1,1,d')$ is the Hirzebruch surface $\mathbb{F}_{d'}$. Let $E$ be the negative section of $\mathbb{F}_{d'}$ and $F$ be a general fibre of $\mathbb{F}_{d'}$. The strict birational transformation of $D_{1}$ on $\mathbb{F}_{d'}$ is linearly equivalent to $E + d' F$. For every divisor $E_{i}$ on $X'$ exceptional over $X$, we have $D'_{1} \cdot E_{i} = 1$. Hence the strict birational transformation of every $E_{i}$ on $\mathbb{F}_{d'}$ is a fibre. By the universal property of minimal resolution, $X^{min}$ is a sequence of blow-ups from $\mathbb{F}_{d'}$. The following diagram shows the relation between these surfaces:
    $$
    \begin{tikzcd}
        X^{min} \arrow[d] \arrow[dr] & \\
        X' \arrow[d,"f'"] \arrow[rd] & X_{d'}^{min} \cong \mathbb{F}_{d'} \arrow[d]\\
        X & X_{d'} \cong \mathbb{P}(1,1,d')
    \end{tikzcd}
    $$
    
    We claim that the strict birational transformation of $E$ on $X^{min}$ is exceptional over $X'$. Indeed, suppose to the contrary that the strict birational transformation $E_{X'}$ of $E$ on $X'$ is a prime divisor. Then we have a birational morphism $X' \rightarrow \mathbb{F}_{d'}$. Hence we have $E_{X^{min}}^{2} \leq E_{X'}^2 \leq E^{2} = -d'$. Since $X$ is a Del Pezzo surface, we have $K_{X^{min}} \cdot E_{X^{min}} \leq \pi^{*}K_{X} \cdot E_{X^{min}} = K_{X} \cdot E_{X} < 0$. Hence we have
    $$
    -2= (K_{X^{min}} + E_{X^{min}}) \cdot E_{X^{min}} = K_{X^{min}} \cdot E_{X^{min}} + E_{X^{min}}^{2} < -d',
    $$
    a contradiction. This completes the proof of the claim.
    
    Our arguments are similar to the proof of \thref{lemma: no singular point}. Every prime divisor on $\mathbb{F}_{d'}$ is linearly equivalent to an effective divisor of the form $eE + fF$, where 
    \begin{itemize}
        \item $E$ is the negative section of $\mathbb{F}_{d'}$;
        \item $F$ is a fibre of $\mathbb{F}_{d'}$;
        \item One of the following happens:
        \begin{enumerate}
            \item $e=1, f=0$;
            \item $e=0, f=1$;
            \item $e,f \geq 1, f \geq d'e$.
        \end{enumerate}
    \end{itemize} 
    For simplicity, put
    \begin{equation}\label{equation: b_{i}'}
    b'_{i} = \begin{cases}
        b_{i} & \text{if the strict birational transformation of }D_{i} \text{ on } X_{d'}^{min} \cong \mathbb{F}_{d'} \text{ is a prime divisor}, \\
        0 & \text{otherwise}.
    \end{cases}
    \end{equation}
    Denote by $D_{i,X_{d'}^{min}}$ the strict birational transformation of $D_{i}$ on $X_{d'}^{min} \cong \mathbb{F}_{d'}$, which is either a prime divisor (when $b'_{i}>0$) or the zero divisor (when $b'_{i}=0$). By \thref{statement: local classification}, the crepant transformation of $(X,B)$ on $X_{d'}^{min} \cong \mathbb{F}_{d'}$ is the pair 
    $$
    (X_{d'}^{min} \cong \mathbb{F}_{d'},eE + B_{X_{d'}^{min}} + \sum\limits_{i=1}^{u+l} \frac{m_{i}-1+\sum\limits_{j=2}^{n} r_{ij}b_{j}-q_{i}(1-b_{1})}{m_{i}}E_{i})
    $$ 
    where $B = \sum\limits_{i=1}^{n} b_{i}D_{i}$ and $B_{X_{d'}^{min}} = \sum\limits_{i=1}^{n} b'_{i}D_{i,X_{d'}^{min}}$. Notice that Case (2) of \thref{prop: surface adjunction} can be treated by setting $\sum\limits_{j=2}^{n}r_{ij}b_{j}=1$ for $i=u+1, \cdots, u+l$.
    
    Then we can adjust Equation (\ref{equation 1''}) in the proof of \thref{lemma: no singular point} to the following:
    \begin{equation}
    e= 1- \frac{2}{d'} + \frac{1}{d'}(\sum\limits_{D_{i,X_{d'}^{min}} \text{ is horizontal}} (f_{i}- e_{i}d') b'_{i} + \sum\limits_{D_{i,X_{d'}^{min}} \text{ is vertical}} b'_{i} + \sum\limits_{i=1}^{u+l} \frac{m_{i}-1+\sum\limits_{j=2}^{n} r_{ij}b_{j}-q_{i}(1-b_{1})}{m_{i}}) < 1
    \end{equation}
    where $e_{i}$ and $f_{i}$ are positive integers given by the linear equivalence $D_{i,X_{d'}^{min}} \sim e_{i} E + f_{i}F$.
    
    For simplicity, we define
    $$
    b_{i}^{\prime v} = \begin{cases}
        b_{i} & \text{if the strict birational transformation of }D_{i} \text{ on } X_{d'}^{min} \cong \mathbb{F}_{d'} \text{ is a vertical prime divisor}, \\
        0 & \text{otherwise},
    \end{cases}
    $$
    and
    $$
    b_{i}^{\prime h} = \begin{cases}
    b_{i} & \text{if the strict birational transformation of }D_{i} \text{ on } X_{d'}^{min} \cong \mathbb{F}_{d'} \text{ is a horizontal prime divisor}, \\
    0 & \text{otherwise}.
    \end{cases}
    $$
    
    Similar to Equation (\ref{equation 1}) in the proof of \thref{lemma: no singular point}, we have
    \begin{equation}\label{equation 2'}
    e + b_{1} + \sum\limits_{i = 2}^{n} e_{i}b_{i}^{\prime h} = 2.
    \end{equation}
    
    In particular, we have $\sum\limits_{i = 2}^{n} e_{i}b_{i}^{\prime h} > 0$. Hence $\sum\limits_{i = 2}^{n} e_{i}b_{i}^{\prime h} \geq \alpha_{0}(\Gamma(\mathcal{N}_{1})) = \frac{1}{7}$. Recall that for $i=1,2,\cdots,u+l$ we have 
    $$
    0 \leq \frac{m_{i}-1+\sum\limits_{j=2}^{n} r_{ij}b_{j}-q_{i}(1-b_{1})}{m_{i}} \leq 1.
    $$ 
    Hence
    \begin{align*}
    b_{1} &= 2 - e - \sum\limits_{i=2}^{n} e_{i}b_{i}^{\prime h} \\
    & = 1 + \frac{2}{d'} - \sum\limits_{i=2}^{n} e_{i}b_{i}^{\prime h} - \frac{1}{d'}(\sum\limits_{i=1}^{n}(f_{i}- e_{i}d' ) b_{i}^{\prime h} + \sum\limits_{i=1}^{n} b_{i}^{\prime v} + \sum\limits_{i=1}^{u+l} \frac{m_{i}-1+\sum\limits_{j=2}^{n} r_{ij}b_{j}-q_{i}(1-b_{1})}{m_{i}})\\
    &\leq 1 + \frac{2}{d'} - \sum\limits_{i=2}^{n} e_{i}b_{i}^{\prime h} \\
    &\leq \frac{6}{7} + \frac{2}{d'}.
    \end{align*}
    
    If $d' \geq 15$, then $b_{1} \leq \frac{104}{105}$, a contradiction. Hence we can assume that $d' \leq 14$.
    
    Without loss of possible solutions of Equation (\ref{equation 2'}), we can assume that $e_{i}=1$ and $f_{i}=d'$. Indeed, a linear equivalence class of the form $e_{i} E + f_{i}F$ where $f_{i} \geq e_{i}d'$ can be written as a nonnegative $\mathbb{Z}$-linear combination of the linear equivalence class of $E + d'F$ and the linear equivalence class of $F$. In addition, all of these classes are linearly effective and free. Thus the computed numerical bound for $b_{1}$ still holds. Notice that the number of horizontal and vertical components can be changed, and the number $n$ of components of $B$ can be changed as well. This completes the proof.
\end{proof}

    Next, we exclude the case where $(X,D_{1})$ is not plt.

\begin{lemma}[The non-plt case]\label{lemma: non-plt case}
    Notation as in \thref{statement: contraction defined by D'1}. Suppose additionally that $(X,D_{1})$ is not plt. Then either
    $$
    b_{1}\leq\frac{251}{252},
    $$
    or
    $$
    (1-b_{1})D_{1}^{2} \geq \frac{1}{252}.
    $$
    Moreover, in every case remaining after the reduction in \thref{lemma: reduction}, we have
    $$
    b_{1}\leq 1-\frac{1}{4032} < 1-\frac{1}{127260}.
    $$
    In particular, to prove \thref{lemma: D1 contains singularities}, it suffices to assume that every singular point $P$ on $D_{1}$ belongs to the first case in \thref{statement: local classification}, that is, the dual graph $\Gamma_{P}$ is of type $\mathbf{A}$, and the exceptional divisor $E$ is at the end of $\Gamma$ intersecting $D_{1}^{min}$.
\end{lemma}

\begin{proof}
    Suppose that $b_{1}>\frac{251}{252}$. By \thref{statement: equation from adjunction}, the pair $(X,D_{1}+\sum\limits_{i=2}^{n}b_{i}D_{i})$ is log canonical. Since $D_{1}$ is nonsingular and $(X,D_{1})$ is not plt, the integer $l$ in \thref{prop: surface adjunction} satisfies $l\geq 1$. By \thref{statement: bound by degree}, we have
    $$
    (1-b_{1})d = -2+l+\sum\limits_{i=2}^{n}w_{i}b_{i} +\sum\limits_{i=1}^{u} \frac{m_{i}-1+\sum\limits_{j=2}^{n}r_{ij}b_{j}}{m_{i}},
    $$
    where $d=D_{1}^{2}>0$. Consequently,
    $$
    1+(1-b_{1})d = (l-1)+\sum\limits_{i=2}^{n}w_{i}b_{i} +\sum\limits_{i=1}^{u} \frac{m_{i}-1+\sum\limits_{j=2}^{n}r_{ij}b_{j}}{m_{i}} >1.
    $$
    The right-hand side is a nonnegative $\mathbb{Z}$-linear combination of elements of $\widetilde{\Gamma(\mathcal{N}_{1})}$, since $l-1\geq 0$ and $b_{i}\in\widetilde{\Gamma(\mathcal{N}_{1})}$ for $i\geq 2$. Therefore, by \thref{exam: gap numbers},
    $$
    1+(1-b_{1})d
    \geq
    \alpha_{1}\left(\widetilde{\Gamma(\mathcal{N}_{1})}\right)
    =
    1+\frac{1}{252}.
    $$
    Hence
    $$
    (1-b_{1})d\geq\frac{1}{252},
    $$
    which proves the first assertion.

    Now suppose that we are in one of the cases remaining after \thref{lemma: reduction}. By (\ref{equation: d <= d'+2}), we have
    $$
    0<d\leq d'+2\leq 16.
    $$
    If $b_{1}\leq\frac{251}{252}$, the desired bound is immediate.
    Otherwise, the preceding inequality gives
    $$
    1-b_{1}
    \geq\frac{1}{252d}
    \geq\frac{1}{252\cdot 16}
    =\frac{1}{4032}.
    $$
    This completes the proof.
\end{proof}

\begin{statement}
    In the last case of \thref{lemma: reduction}, suppose additionally that $(X,D_{1})$ is plt, so that $l=0$. We can adjust Equation (\ref{equation 1'}) in the proof of \thref{lemma: no singular point} to obtain the equation
    $$
    d' \sum\limits_{i=1}^{n} b_{i}^{\prime h} + \sum\limits_{i=1}^{n} b_{i}^{\prime v} + \sum\limits_{i=1}^{u} \frac{m_{i}-1+\sum\limits_{j=2}^{n} r_{ij}b_{j}-q_{i}(1-b_{1})}{m_{i}} = d'+2
    $$
    where $0 \leq \frac{m_{i}-1+\sum\limits_{j=2}^{n} r_{ij}b_{j}-q_{i}(1-b_{1})}{m_{i}} \leq 1$ for every $i$, and $0 \leq \sum\limits_{i=1}^{n} b_{i}^{\prime v} + \sum\limits_{i=1}^{u} \frac{m_{i}-1+\sum\limits_{j=2}^{n} r_{ij}b_{j}-q_{i}(1-b_{1})}{m_{i}} \leq 2$. The above equation can be rewritten as
    
    \begin{equation}\label{equation 2}
    (1-b_{1})(d' + \sum\limits_{i=1}^{u} \frac{q_{i}}{m_{i}}) = d' \sum\limits_{i=2}^{n} b_{i}^{\prime h}+\sum\limits_{i=2}^{n} b_{i}^{\prime v} + \sum\limits_{i=1}^{u} \frac{m_{i}-1+\sum\limits_{j=2}^{n} r_{ij}b_{j}}{m_{i}} - 2 > 0.    
    \end{equation}
    
    Notice that $d = d' + \sum\limits_{i=1}^{u} \frac{q_{i}}{m_{i}}$. Since $\frac{m_{i}-1+\sum\limits_{j=2}^{n} r_{ij}b_{j}}{m_{i}} \in \widetilde{\Gamma(\mathcal{N}_{1})}$, gap numbers can be computed for the linear combination on the right of (\ref{equation 2}). 
\end{statement}

    We set up the convention of the dual graph on $X^{min}$, which will be used in the subsequent proof.

\begin{convention}[The dual graph, cf. \cite{Alexeev1993TwoTwoDimensionalTerminations}]\label{convention: dual graph}
    Notation as in \thref{statement: contraction defined by D'1}. We consider the weighted graph of curves on $X^{min}$ which contains two classes: 
    \begin{itemize}
        \item ``internal'', for the curves that are exceptional over $X$, and
        \item ``external'', for the curves that are either the strict birational transformations of $D_{i}$, or the strict birational transformations of some divisors on $X$ which are exceptional over $\mathbb{P}^2$. Notice that the two external cases overlap: the strict birational transformations of some $D_{i}$ may also be exceptional over $\mathbb{P}^2$.
    \end{itemize}
    A vertex in the dual graph is either labeled by its corresponding divisor, or represented by:
    \begin{itemize}
        \item an empty circle $\bigcirc$, if the corresponding divisor is exceptional over $X$ and centered at the singularity of $X$ on $D_{1}$;
        \item a crossed circle $\otimes$, if the corresponding divisor is the strict birational transformation of a divisor on $X$ that is exceptional over $\mathbb{P}^2$.
        \item a dotted circle $\odot$, if the corresponding divisor is an exceptional divisor centered at a singularity of $X$ \emph{not} on $D_{1}$.
    \end{itemize}

    In the dual graphs, we use $\cdots$ to represent a chain of vertices (which might be empty). The subscript on $\cdots$ specifies the allowed types of the omitted vertices in the corresponding chain. For example, $\cdots_{\bigcirc}$ indicates that every omitted vertex has type $\bigcirc$, while $\cdots_{\odot,\otimes}$ indicates that each omitted vertex has either type $\odot$ or type $\otimes$.
\end{convention}

\paragraph{$D_{1}$ contains 1 singularity of $X$}
\hfill

    In this subsection, we prove the following:
    
\begin{lemma}\label{lemma: D1 contains 1 singularity}
    Under the assumptions of \thref{thm: act gap}, suppose additionally that $D_{1}$ is a nonsingular rational curve and $D_{1}$ contains exactly 1 singularity of $X$. Then we have $b_{1} \leq 1 - \frac{1}{127260}$. Moreover, the equality holds only in \thref{example: optimal bound for act 2}.
\end{lemma}
    
\begin{proof}
    Let $m=m_{1}$ and $q = q_{1}$ be as in \thref{construction: extract to make Cartier}, then we have $d = d' + \frac{q}{m} < d' + 1$.

    \begin{enumerate}[align=left,label = \textbf{Case \arabic*}:]
        \item $d' \leq -1 $. In this case $d < 0$, which is impossible.
        \item $d' = 0$. In this case $d < 1$. Hence by \thref{statement: equation from adjunction} and \thref{exam: gap numbers}, $b_{1} < 1 - \frac{1}{127260}$.
        \item $d' = 1$. In this case $d < 2$ and $X_{d'} \cong \mathbb{P}^2$. We prove that this case is possible only in Case 2 of Table \ref{table: potential numerical cases for 1 singularity}, that is, when $m=2$ and $q=r_{2}=1$. Moreover, this case is given as in \thref{example: optimal bound for act 2}. Suppose that $b_{1} \geq 1-\frac{1}{127260}$. Then by \thref{statement: bound by degree} and \thref{lemma: non-plt case}, we have
        \begin{equation}\label{inequality 1}
        2 < \sum\limits_{i=2}^{n} w_{i}b_{i} + \frac{m-1+\sum\limits_{i=2}^{n} r_{i}b_{i}}{m} = 2 + (1-b_{1})(1 + \frac{q}{m}) \leq 2 + \frac{1 + \frac{q}{m}}{127260} \leq 2 + \frac{1 + \frac{m-1}{m}}{127260}.
        \end{equation}
        If $m \geq 254$, then we have
        $$
        2 + \frac{1}{63630} > \sum\limits_{i=2}^{n} w_{i}b_{i} + \frac{m-1+\sum\limits_{i=2}^{n} r_{i}b_{i}}{m} \geq \alpha_{1}(\Gamma(\mathcal{N}_{1})) + \frac{m-1}{m} \geq 1 + \frac{1}{252} + \frac{253}{254} = 2 + \frac{1}{32004},
        $$
        a contradiction. Hence we can assume that $2 \leq m \leq 253$. We use a computer to enumerate all the linear combinations satisfying inequalities (\ref{inequality 1}). The results are listed in Table \ref{table: potential numerical cases for 1 singularity} below. We give some explanations of Table \ref{table: potential numerical cases for 1 singularity}:
        \begin{itemize}
        \item Each row is distinguished by the following numerical data: $m$, $w_{i}$ (implicitly), $b_{i}$ and $\sum\limits_{i=2}^{n}r_{i}b_{i}$. Notice that two cases with different $r_{i}$ but the same $\sum\limits_{i=2}^{n}r_{i}b_{i}$ will be considered as the same case.
        \item If two cases are related by a permutation of the indices of $b_{i}, i=2,3,\cdots, n$, we will recognize them as the same case.
        \item To simplify the expression in the table, we replace the $b_{1}$ column by the column of $a = \frac{1}{(1-b_{1})(1 + \frac{q}{m})}$. Then in inequalities (\ref{inequality 1}) we can write $2 + (1-b_{1})(1 + \frac{q}{m}) = 2 + \frac{1}{a}$. Notice that the condition $b_{1} \geq 1 - \frac{1}{127260}$ is equivalent to $a(1 + \frac{q}{m}) \geq 127260$.
        \item To remove numerically redundant cases, we assume that for all $2 \leq i \leq n$, we have $w_{i}=0$ or $w_{i}=1$. Indeed, we can formally split the term $w_{i}b_{i}$ in the sum $\sum\limits_{i=2}^{n} w_{i}b_{i} + \frac{m-1+\sum\limits_{i=2}^{n} r_{i}b_{i}}{m}$ into $w_{i}$ copies of $b_{i}$. Conversely, there exists a sum $\sum\limits_{i=2}^{n} w_{i}b_{i} + \frac{m-1+\sum\limits_{i=2}^{n} r_{i}b_{i}}{m}$ with $w_{i}>1$ for some $i$ if and only if there exists a sum in Table \ref{table: potential numerical cases for 1 singularity} where $w_{i}=w_{j}=1$ and $b_{i}=b_{j}$ for some $i \neq j$ (later we will see that this doesn't happen).
        \item A number $b_{i}$ is displayed in Table \ref{table: potential numerical cases for 1 singularity} if and only if $w_{i}=1$. For the numbers $b_{i}$ where $w_{i} = 0$, we must have $r_{i} > 0$. In this case, $b_{i}$ is not displayed in the table but still contributes to the sum $\sum\limits_{i=2}^{n} r_{i}b_{i}$.
        \item In the last column, we give an upper bound of the number $r_{2}$ for each case. Notice that if $r_{i}=0$, then $D_{i}$ doesn't pass through the singular point $P_{1}$ of $X$ on $D_{1}$.
    \end{itemize}
    
    \begin{scriptsize}
    \begin{center}
    \begin{longtable}{|@{\makebox[3em][r]{\rownumber\space}}|C|C|C|C|C|C|}
    \caption{Potential numerical cases when $D_{1}$ contains $1$ singularity}
    \label{table: potential numerical cases for 1 singularity} \\
    \hline
    \multicolumn{1}{@{\makebox[3em][r]{Case No.~}}|C|}{m}
    & a
    & b_{2}
    & b_{3}
    & \sum\limits_{i=2}^{n} r_{i}b_{i}
    & r_{2}\leq
    \\
    \hline

    2 & 104720 & \frac{1}{7} & \frac{186}{187} & \frac{29}{40} & 0 \\
    \hline
    2 & 84840 & \frac{1}{7} & \frac{502}{505} & \frac{61}{84} & 1 \\
    \hline
    2 & 127260 & \frac{1}{7} & \frac{31}{36} & \frac{501}{505} & 0 \\
    \hline
    2 & 104580 & \frac{1}{7} & \frac{31}{36} & \frac{2882}{2905} & 0 \\
    \hline
    3 & 127260 & \frac{1}{7} & \frac{503}{505} & \frac{7}{12} & 0 \\
    \hline
    3 & 104580 & \frac{1}{7} & \frac{5787}{5810} & \frac{7}{12} & 0 \\
    \hline
    3 & 87570 & \frac{1}{7} & \frac{2769}{2780} & \frac{7}{12} & 0 \\
    \hline
    3 & 81900 & \frac{1}{7} & \frac{2266}{2275} & \frac{7}{12} & 0 \\
    \hline
    4 & 104720 & \frac{1}{7} & \frac{186}{187} & \frac{9}{20} & 0 \\
    \hline
    4 & 82320 & \frac{1}{7} & \frac{2047}{2058} & \frac{9}{20} & 0 \\
    \hline
    5 & 81900 & \frac{1}{7} & \frac{31}{36} & \frac{446}{455} & 0 \\
    \hline
    7 & 82320 & \frac{1}{7} & \frac{69}{80} & \frac{283}{294} & 0 \\
    \hline
    10 & 81900 & \frac{1}{7} & \frac{31}{36} & \frac{437}{455} & 0 \\
    \hline
    11 & 104720 & \frac{1}{7} & \frac{69}{80} & \frac{16}{17} & 0 \\
    \hline
    13 & 81900 & \frac{1}{7} & \frac{31}{36} & \frac{166}{175} & 0 \\
    \hline
    14 & 82320 & \frac{1}{7} & \frac{69}{80} & \frac{136}{147} & 0 \\
    \hline
    17 & 104720 & \frac{1}{7} & \frac{69}{80} & \frac{10}{11} & 0 \\
    \hline
    26 & 65520 & \frac{1}{7} & \frac{69}{80} & \frac{31}{36} & 0 \\
    \hline
    26 & 81900 & \frac{1}{7} & \frac{31}{36} & \frac{157}{175} & 0 \\
    \hline
    39 & 65520 & \frac{1}{7} & \frac{69}{80} & \frac{19}{24} & 0 \\
    \hline
    41 & 68880 & \frac{1}{7} & \frac{69}{80} & \frac{82}{105} & 0 \\
    \hline
    49 & 82320 & \frac{1}{7} & \frac{69}{80} & \frac{31}{42} & 0 \\
    \hline
    59 & 74340 & \frac{1}{7} & \frac{31}{36} & \frac{23}{30} & 0 \\
    \hline
    65 & 81900 & \frac{1}{7} & \frac{31}{36} & \frac{26}{35} & 1 \\
    \hline
    77 & 64680 & \frac{1}{7} & \frac{145}{168} & \frac{19}{35} & 1 \\
    \hline
    78 & 65520 & \frac{1}{7} & \frac{69}{80} & \frac{7}{12} & 0 \\
    \hline
    79 & 66360 & \frac{1}{3} & \frac{47}{70} & \frac{5}{8} & 0 \\
    \hline
    83 & 104580 & \frac{1}{7} & \frac{31}{36} & \frac{47}{70} & 0 \\
    \hline
    98 & 82320 & \frac{1}{7} & \frac{69}{80} & \frac{10}{21} & 1 \\
    \hline
    101 & 84840 & \frac{1}{7} & \frac{145}{168} & \frac{2}{5} & 0 \\
    \hline
    101 & 127260 & \frac{1}{7} & \frac{31}{36} & \frac{3}{5} & 0 \\
    \hline
    103 & 64890 & \frac{1}{7} & \frac{31}{36} & \frac{83}{140} & 1 \\
    \hline
    118 & 74340 & \frac{1}{7} & \frac{31}{36} & \frac{8}{15} & 0 \\
    \hline
    123 & 68880 & \frac{1}{7} & \frac{69}{80} & \frac{12}{35} & 1 \\
    \hline
    130 & 81900 & \frac{1}{7} & \frac{31}{36} & \frac{17}{35} & 2 \\
    \hline
    139 & 87570 & \frac{1}{7} & \frac{31}{36} & \frac{9}{20} & 0 \\
    \hline
    158 & 66360 & \frac{1}{3} & \frac{47}{70} & \frac{1}{4} & 0 \\
    \hline
    166 & 104580 & \frac{1}{7} & \frac{31}{36} & \frac{12}{35} & 1 \\
    \hline
    187 & 104720 & \frac{1}{7} & \frac{69}{80} & 0 & 0 \\
    \hline
    202 & 127260 & \frac{1}{7} & \frac{31}{36} & \frac{1}{5} & 0 \\
    \hline
    253 & 63756 & \frac{1}{7} & \frac{31}{36} & 0 & 0 \\
    \hline
    
    \end{longtable}
    \end{center}
    \end{scriptsize}

    Now we examine the possible values of $q$ by discussing the possible dual graphs on the minimal resolution $X^{min}$ of $X$ that are possibly geometrically realizable. The dual graph satisfies the following properties:
    
    \begin{enumerate}[align=left, label=(\arabic*)]
        \item For every $2 \leq i \leq n$, the divisor $D'_{i}$ does not pass through $E_{1} \cap D'_{1}$. In particular, we have $D'_{i} \cdot D'_{1} = w_{i}$. Indeed, if $D'_{i}$ passes through $E_{1} \cap D'_{1}$, then we have $r_{i} \geq m$. We can then replace $r_{i}$ by $r_{i} - m$ and replace $w_{i}$ by $w_{i}+1$ to create a new case satisfying inequalities (\ref{inequality 1}). However, there is no such case in Table \ref{table: potential numerical cases for 1 singularity}.
        \item For $2 \leq i \leq n$, there is an edge between the vertex of the divisor $D_{i}$ and the vertex of the divisor $D_{1}$ if and only if $i=2,3$. Indeed, since the morphism $X^{min} \rightarrow X'$ is an isomorphism near $D'_{1}$, the edge can be computed by $D'_{i} \cdot D'_{1} = w_{i}$. As explained before Table \ref{table: potential numerical cases for 1 singularity}, there exists a case where $w_{i}>1$ if and only if there exists a case in Table \ref{table: potential numerical cases for 1 singularity} where two $b_{i}$ have the same value. In all cases of Table \ref{table: potential numerical cases for 1 singularity}, we have $b_{2} \neq b_{3}$. Hence we have $w_{i} \in \{ 0,1 \}$. Moreover, as explained before Table \ref{table: potential numerical cases for 1 singularity}, we have $w_{i}=1$ if and only if the corresponding $b_{i}$ is displayed in Table \ref{table: potential numerical cases for 1 singularity}. In all cases of Table \ref{table: potential numerical cases for 1 singularity}, we have $w_{2}=w_{3}=1$ and $w_{i}=0$ for $4 \leq i \leq n$.
        \item There is an edge between the vertex of the divisor $D_{2}$ and the vertex of the divisor $E_{1}$. In particular, $r_{2} \geq 1$. Indeed, in all cases of Table \ref{table: potential numerical cases for 1 singularity}, one can check that 
        $$
        (\sum\limits_{i=2}^{n}r_{i}b_{i}) - b_{3} \notin \{\sum\limits_{i=1}^{j}l_{i}\beta_{i} \mid j \in \mathbb{Z}_{>0}, l_{i} \in \mathbb{Z}_{\geq 0}, \beta_{i} \in \Gamma(\mathcal{N}_{1}) \}.
        $$
        Hence we have $r_{3}=0$. Suppose to the contrary that there is no edge between the vertex of the divisor $D_{2}$ and the vertex of the divisor $E_{1}$, then the strict birational transformations of $D_{2},D_{3},E_{1}$ on $X_{d'} \cong \mathbb{P}^2$ are lines, and they intersect at one point $Q$. Consider the crepant transformation 
        $$
        (X_{d'} \cong \mathbb{P}^2, B_{X_{d'}} = b_{1}D_{1,X_{d'}} + b_{2}D_{2,X_{d'}} + b_{3}D_{3,X_{d'}} + b_{E_{1}}E_{1,X_{d'}})
        $$ 
        of the klt 0-pair $(X,B)$ on $X_{d'} \cong \mathbb{P}^2$, where $D_{i,X_{d'}}$ (resp. $E_{1,X_{d'}}$) is the strict birational transformation of $D_{i}$ (resp. $E_{1}$) on $X_{d'}$. It is still a klt 0-pair, so we have $b_{1} + b_{2} + b_{3} + b_{E_{1}} = 3$. However, the multiplicity of $B_{X_{d'}}$ at $Q$ is $b_{2} + b_{3} + b_{E_{1}} = 3 - b_{1} > 2$. Hence the pair $(X_{d'},B_{X_{d'}})$ is not log canonical at $Q$, a contradiction.
        \item There is an edge between the vertex of the divisor $D_{3}$ and the vertex of type $\otimes$ or $\odot$. Indeed, since the strict birational transformations of $D_{3}$ and $E_{1}$ on $\mathbb{P}^2$ intersect, there is an edge between the vertex of the divisor $D_{3}$ and the vertex of type $\bigcirc$, $\otimes$, or $\odot$. Moreover, by (3) we have $r_{3}=0$, so there is no edge between the vertex of the divisor $D_{3}$ and the vertex of type $\bigcirc$.
        \item There is exactly one vertex of type $\otimes$, and there is an edge between this vertex and a vertex of type $\bigcirc$. Indeed, a vertex of type $\otimes$ corresponds to an exceptional divisor of $X' \rightarrow X_{d'} \cong \mathbb{P}^2$. In this case, $X'$ extracts $E_{1}$ from $X$ and contracts exactly one prime divisor to $X_{d'} \cong \mathbb{P}^2$. Let $Z$ be the prime divisor on $X'$ exceptional over $\mathbb{P}^2$, and $Z_{X}$ be its strict birational transformation on $X$. Then $Z \cdot D'_{1} = 0$ and $Z_{X} \cdot D_{1} >0$. Hence the strict birational transformation of $Z$ on $X^{min}$ intersects the chain of exceptional divisors of $X^{min} \rightarrow X$. In particular, since the vertex of the prime divisor $D_{i}$ is of type $\otimes$ if and only if $D'_{i} \cdot D'_{1} =0$, we have $n \leq 4$ by (1) and (2). In fact, in all cases of Table \ref{table: potential numerical cases for 1 singularity}, we have $r_{4} \geq 1$, so $n=4$.
    \end{enumerate}

    Now we use the above properties to complete the discussion for the $d'=1$ case. We discuss the problem on the basis of the possible value of $r_{2}$.

\begin{enumerate}
    \item $r_{2} = 1$. This happens in Cases 2, 24, 25, 29, 32, 34, 35, 38 of Table \ref{table: potential numerical cases for 1 singularity}. By (3), we have $1 = r_{2} \geq q = \Delta(\Gamma - v_{E_{1}}) > 0$. Hence we conclude that $q = 1$ and $\Gamma$ contains only 1 internal vertex $v_{E_{1}}$ corresponding to $E_{1}$. Hence we have the following dual graph:
    
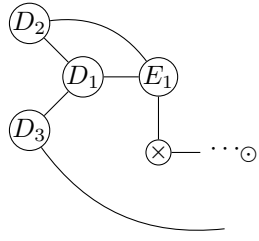
\begin{figure}[H]
    \centering
    \begin{tikzpicture}
        \node[inner sep=0, draw, circle] (0) {$D_{1}$};
        \node[inner sep=0, draw, circle] (1) [right of=0] {$E_{1}$};
        \draw (0) -- (1);
        \node[draw, circle, label=center:$\times$] (5) [below of=1] {};
        \draw (1) -- (5);
        \node[] (6) [right of=5] {$\cdots_{\odot}$};
        \draw (5) -- (6);
        \node[inner sep=0, draw, circle] (a) [above left of=0]{$D_{2}$};
        \draw (0) -- (a);
        \node[inner sep=0, draw, circle] (b) [below left of=0]{$D_{3}$};
        \draw (0) -- (b);
        \draw (a) to[bend left] (1);
        \node[] (6') [below of=6] {};
        \draw (b) to[bend right] (6');
    \end{tikzpicture}
    \end{figure}
    Here:
    \begin{itemize}
        \item For every divisor $D_{i}^{min}$ and $E_{1}^{min}$ in the dual graph, the superscript $^{min}$ is ignored.
        \item For $i,j \neq 1$, the possible edge between $D_{i}$ and $D_{j}$ is ignored.
        \item The vertex corresponding to $D_{3}$ connects to one of the vertices of type $\otimes$ or $\odot$.
    \end{itemize}
    Now we can verify all the above cases to show that $b_{1} < 1 - \frac{1}{127260}$ except for Case 2, where $b_{1} = 1 - \frac{1}{127260}$. We claim that this is \thref{example: optimal bound for act 2}. Indeed, in this case, the crepant transformation of $(X,B)$ on $X_{d'} \cong \mathbb{P}^2$ is the pair $(\mathbb{P}^2, \frac{127259}{127260}L_{1} + \frac{1}{7}L_{2} + \frac{502}{505}L_{3} + \frac{15691}{18180}L)$, where the lines $L_{1},L_{2},L_{3}$ are the strict birational transformation of $D_{1},D_{2},D_{3}$ respectively, and the line $L$ is the strict birational transformation of $E_{1}$, and $D_{4}$ is centered at $L_{3} \cap L$. Moreover, all the exceptional divisors of $\varphi_{min}: X^{min} \rightarrow \mathbb{P}^2$ are centered at $L \cap L_{3}$. Since $X^{min}$ and $\mathbb{P}^2$ are smooth, $\varphi_{min}$ is a sequence of blow-ups from $L_{3} \cap L$. Moreover, the divisor $D_{4}^{min}$ is the last blown-up divisor, which is a $(-1)$-curve. Denote by $\varphi_{d'}: X' \rightarrow \mathbb{P}^2$ the birational morphism defined by $|D'_{1}|$. Then by the projection formula, we have
    $$
    1 = L^{2} = (\varphi_{d'}^{*}(L)) \cdot E_{1}.
    $$
    A direct calculation shows that $\varphi_{d'}^{*}L = E_{1} + 3D'_{4}$, $\varphi_{d'}^{*}L_{3} = D'_{3} + D'_{4}$, and $D_{4}^{\prime 2} = - \frac{1}{3}$. In particular, we have $\mathrm{mult}_{D_{4}^{min}}\varphi_{min}^{*}L = 3$ and $\mathrm{mult}_{D_{4}^{min}}\varphi_{min}^{*}L_{3} = 1$. The only way to extract $D_{4}$ by a sequence of blow-ups is as follows:
    \begin{itemize}
        \item First, blow up the intersection $L\cap L_{3}$. Write $\pi_{1}: Y_{1} \rightarrow \mathbb{P}^2$ and let $F_{1}$ be the exceptional divisor. Then we have $\pi_{1}^{*}L = L_{Y_{1}} + F_{1}$ and $\pi_{1}^{*}L_{3} = L_{3,Y_{1}} + F_{1}$, where $L_{Y_{1}}$ is the strict birational transformation of $L$ on $Y_{1}$ and $L_{3,Y_{1}}$ is the strict birational transformation of $L_{3}$ on $Y_{1}$. We will use similar notations below without further explanation.
        \item Second, blow up the intersection $L_{Y_{1}} \cap F_{1}$. Write $\pi_{2}: Y_{2} \rightarrow Y_{1}$ and let $F_{2}$ be the exceptional divisor. Then we have $\pi_{2}^{*}\pi_{1}^{*}L = L_{Y_{2}} + F_{1,Y_{2}} + 2F_{2}$ and $\pi_{2}^{*}\pi_{1}^{*}L_{3} = L_{3,Y_{2}} + F_{1,Y_{2}} + F_{2}$.
        \item Third, blow up the intersection $L_{Y_{2}} \cap F_{2}$. Write $\pi_{3}: Y_{3} \rightarrow Y_{2}$ and let $F_{3}$ be the exceptional divisor. Then we have $\pi_{3}^{*}\pi_{2}^{*}\pi_{1}^{*}L = L_{Y_{3}} + F_{1,Y_{3}} + 2F_{2,Y_{3}} + 3F_{3}$ and $\pi_{3}^{*}\pi_{2}^{*}\pi_{1}^{*}L_{3} = L_{3,Y_{3}} + F_{1,Y_{3}} + F_{2,Y_{3}} + F_{3}$. The divisor $F_{3}$ now has the correct multiplicity.
        \item Finally we show that $Y_{3} \cong X^{min}$ and $D_{4}^{min} = F_{3}$. Indeed, since $\mathrm{mult}_{D_{4}^{min}}(\varphi_{min}^{*}L) = 3$ and $\mathrm{mult}_{D_{4}^{min}}(\varphi_{min}^{*}L_{3}) = 1$, we can conclude that either $D_{4}^{min}=F_{3}$ or $D_{4}^{min}$ can be extracted by a sequence of general blow-ups on $F_{3}$. In particular, the dual graph of the exceptional divisors of $\varphi_{min}$ has the following form:
        
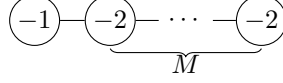
\begin{figure}[H]
        \centering
        \begin{tikzpicture}
            \node[inner sep=1.5, draw, circle] (5) {$-1$};
            \node[inner sep=1.5, draw, circle] (6) [right of=5] {$-2$};
            \draw (5) --  (6);
            \node[] (7) [right of=6] {$\cdots$};
            \draw (6) --  (7);
            \node[inner sep=1.5, draw, circle] (8) [right of=7] {$-2$};
            \draw (7) -- (8);
            \draw[decoration = {brace, amplitude=.3em, raise=.1em}, decorate] (8.south) -- node[midway,below]{$M$} (6.south);
        \end{tikzpicture}
        \end{figure}
        We have $\frac{M}{M+1} = (D'_{4})^{2} - (D^{min}_{4})^{2} = -\frac{1}{3} + 1 = \frac{2}{3}$. Hence $M=2$ and there is no extra blow-up.
    \end{itemize}

    Hence the dual graph is
    
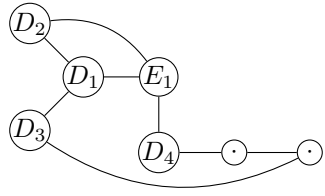
\begin{figure}[H]
    \centering
    \begin{tikzpicture}
        \node[inner sep=0, draw, circle] (0) {$D_{1}$};
        \node[inner sep=0, draw, circle] (1) [right of=0] {$E_{1}$};
        \draw (0) -- (1);
        \node[inner sep=0, draw, circle] (5) [below of=1] {$D_{4}$};
        \draw (1) -- (5);
        \node[inner sep=1.5, draw, circle] (6) [right of=5] {$\cdot$};
        \draw (5) -- (6);
        \node[inner sep=1.5, draw, circle] (7) [right of=6] {$\cdot$};
        \draw (6) -- (7);
        \node[inner sep=0, draw, circle] (a) [above left of=0]{$D_{2}$};
        \draw (0) -- (a);
        \node[inner sep=0, draw, circle] (b) [below left of=0]{$D_{3}$};
        \draw (0) -- (b);
        \draw (a) to[bend left] (1);
        \draw (b) to[bend right] (7);
    \end{tikzpicture}
    \end{figure}
    We conclude from the above dual graph that $X \cong \mathbb{P}(1,2,3)$.
    
\item $r_{2} = 2$. This only happens in Case 35 of Table \ref{table: potential numerical cases for 1 singularity}. We show that this case is impossible. Indeed, in this case, we have $r_{4}=1$ and $2 = r_{2} \geq q = \Delta(\Gamma - v_{E_{1}}) > 0$. Hence $a(1+\frac{q}{m}) = 81900(1 + \frac{q}{130}) \leq 81900(1 + \frac{1}{65}) < 127260$.
\end{enumerate}

\item $2 \leq d' \leq 14$. In this case, we have $d < d' + 1$. We prove by contradiction. Suppose to the contrary that $b_{1} \geq 1 - \frac{1}{127260}$. Then by Equation (\ref{equation 2}),
    \begin{equation}\label{inequality 1'}
    d' \sum\limits_{i=2}^{n}b_{i}^{\prime h}+\sum\limits_{i=2}^{n} b_{i}^{\prime v} + \frac{m-1+\sum\limits_{i=2}^{n} r_{i}b_{i}}{m} = 2 + (1-b_{1})(d'+\frac{q}{m}) < 2 + \frac{d'+1}{127260}.
    \end{equation}

    We use a computer to enumerate all the linear combinations satisfying inequality (\ref{inequality 1'}). There are no such cases.

\end{enumerate}

    In conclusion, we have $b_{1} \leq 1 - \frac{1}{127260}$, and the equality holds in \thref{example: optimal bound for act 2}. \NoEndMark\hfill\proofSymbol

\end{proof}

\paragraph{$D_{1}$ contains 2 singularities of $X$}
\hfill

    In this subsection, we prove the following:
    \begin{lemma}\label{lemma: D1 contains 2 singularities}
        Under the assumptions of \thref{thm: act gap}, suppose additionally that $D_{1}$ is a nonsingular rational curve and $D_{1}$ contains exactly 2 singularities of $X$. Then we have $b_{1} < 1 - \frac{1}{127260}$.
    \end{lemma}
    
\begin{proof}
    Let $m_{1},m_{2},q_{1},q_{2}$ be as in \thref{construction: extract to make Cartier}. Then we have $d = d' + \frac{q_{1}}{m_{1}} + \frac{q_{2}}{m_{2}} < d' + 2$.
    \begin{enumerate}[align=left,label = \textbf{Case \arabic*}:]
    
    \item $d' \leq -1 $. In this case, we have $d < 1$. By \thref{statement: equation from adjunction} and \thref{exam: gap numbers} we have $b_{1} < 1 - \frac{1}{127260}$.
    
    \item $d' = 0$. In this case we have $d < 2$. The linear system $|D'_{1}|$ defines a morphism from $X'$ to $\mathbb{P}^1$, where the divisor $D'_{1}$ is a fibre and the divisors $E_{i}$ are sections. We claim that this case is impossible. Indeed, by \thref{statement: bound by degree}, we have
    $$
    \frac{m_{1}-1 + \sum\limits_{i=2}^{n} r_{i}b_{i}}{m_{1}} + \frac{m_{2} - 1 + \sum\limits_{i=2}^{n} r'_{i}b_{i}}{m_{2}} + \sum\limits_{i=2}^{n} w_{i}b_{i} = 2 + (\frac{q_{1}}{m_{1}} + \frac{q_{2}}{m_{2}})(1-b_{1}) > 2.
    $$
    If $m_{1},m_{2} \geq 15$, then we have 
    $$
    2 + (\frac{q_{1}}{m_{1}} + \frac{q_{2}}{m_{2}})(1-b_{1}) = \frac{m_{1}-1 + \sum\limits_{i=2}^{n} r_{i}b_{i}}{m_{1}} + \frac{m_{2} - 1 + \sum\limits_{i=2}^{n} r'_{i}b_{i}}{m_{2}} + \sum\limits_{i=2}^{n} w_{i}b_{i} \geq \frac{14}{15}+\frac{14}{15} + \frac{1}{7} \geq 2 + \frac{1}{105}.
    $$ 
    Hence $b_{1} \leq 1 - \frac{1}{105 \times (\frac{q_{1}}{m_{1}} + \frac{q_{2}}{m_{2}})} < 1 - \frac{1}{210}$. If $m_{2} \geq 254$, then we have
    $$
    \frac{m_{1}-1 + \sum\limits_{i=2}^{n} r_{i}b_{i}}{m_{1}} + \frac{m_{2} - 1 + \sum\limits_{i=2}^{n} r'_{i}b_{i}}{m_{2}} + \sum\limits_{i=2}^{n} w_{i}b_{i} \geq \alpha_{1}(\widetilde{\Gamma}(\mathcal{N}_{1})) + \frac{253}{254} = 2 + \frac{1}{32004}.
    $$
    Hence $b_{1} \leq 1 - \frac{1}{32004 \times (\frac{q_{1}}{m_{1}} + \frac{q_{2}}{m_{2}})} < 1 - \frac{1}{64008}$.
    
    Now we assume that $m_{1} \leq 14$, $m_{1} \leq m_{2}$ and $m_{2} \leq 253$. We prove by contradiction that in this case, $b_{1}< 1 - \frac{1}{127260}$. Indeed, assume to the contrary that $b_{1} \geq 1 - \frac{1}{127260}$. Then we have 
    \begin{equation}\label{inequality 2~}
    \frac{m_{1}-1 + \sum\limits_{i=2}^{n} r_{i}b_{i}}{m_{1}} + \frac{m_{2} - 1 + \sum\limits_{i=2}^{n} r'_{i}b_{i}}{m_{2}} + \sum\limits_{i=2}^{n} w_{i}b_{i} = 2 + (\frac{q_{1}}{m_{1}} + \frac{q_{2}}{m_{2}})(1-b_{1}) \leq 2 + \frac{\frac{q_{1}}{m_{1}} + \frac{q_{2}}{m_{2}}}{127260} \leq 2 + \frac{\frac{m_{1}-1}{m_{1}} + \frac{m_{2}-1}{m_{2}}}{127260}.
    \end{equation}

    We use a computer to enumerate all the linear combinations satisfying inequalities (\ref{inequality 2~}). The results are listed in Table \ref{table: potential numerical cases for 2 singularities, X} below. We give some explanations of Table \ref{table: potential numerical cases for 2 singularities, X}:
    \begin{itemize}
        \item Each row is distinguished by the following numerical data: $m_{1}$, $m_{2}$, $w_{i}$ (implicitly), $b_{i}$, $\sum\limits_{i=2}^{n}r_{i}b_{i}$ and $\sum\limits_{i=2}^{n}r'_{i}b_{i}$. Notice that two cases with different $r_{i}$ and $r'_{i}$ but the same $\sum\limits_{i=2}^{n}r_{i}b_{i}$ and $\sum\limits_{i=2}^{n}r'_{i}b_{i}$ will be considered as the same case.
        \item If two cases are related by a permutation of the indices of $b_{i}, i=2,3,\cdots, n$, we will recognize them as the same case.
        \item To simplify the expression in the table, we replace the $b_{1}$ column by the column of $a = \frac{1}{(1-b_{1})(\frac{q_{1}}{m_{1}}+\frac{q_{2}}{m_{2}})}$. Then in inequalities (\ref{inequality 2~}) we can write $2 + (1-b_{1})(\frac{q_{1}}{m_{1}}+\frac{q_{2}}{m_{2}}) = 2 + \frac{1}{a}$. Notice that the condition $b_{1} \geq 1 - \frac{1}{127260}$ is equivalent to $a(\frac{q_{1}}{m_{1}}+\frac{q_{2}}{m_{2}}) \geq 127260$.
        \item To remove numerically redundant cases, we can assume that for all $2 \leq i \leq n$, we have $w_{i}=0$ or $w_{i}=1$. Indeed, we can formally split the term $w_{i}b_{i}$ in the sum $\frac{m_{1}-1 + \sum\limits_{i=2}^{n} r_{i}b_{i}}{m_{1}} + \frac{m_{2} - 1 + \sum\limits_{i=2}^{n} r'_{i}b_{i}}{m_{2}} + \sum\limits_{i=2}^{n} w_{i}b_{i}$ into $w_{i}$ copies of $b_{i}$. Conversely, there exists a sum $\frac{m_{1}-1 + \sum\limits_{i=2}^{n} r_{i}b_{i}}{m_{1}} + \frac{m_{2} - 1 + \sum\limits_{i=2}^{n} r'_{i}b_{i}}{m_{2}} + \sum\limits_{i=2}^{n} w_{i}b_{i}$ with $w_{i}>1$ for some $i$ if and only if there exists a sum in Table \ref{table: potential numerical cases for 2 singularities, X} where $w_{i}=w_{j}=1$ and $b_{i}=b_{j}$ for some $i \neq j$ (later we will see that this doesn't happen).
        \item A number $b_{i}$ is displayed in Table \ref{table: potential numerical cases for 2 singularities, X} if and only if $w_{i}=1$. For the numbers $b_{i}$ where $w_{i} = 0$, we must have $r_{i} > 0$ or $r'_{i}>0$. In this case, $b_{i}$ is not displayed in the table but still contributes to the sum $\sum\limits_{i=2}^{n} r_{i}b_{i}$ or $\sum\limits_{i=2}^{n} r'_{i}b_{i}$.
        \item In the last column, we give the upper bounds of the number $r_{2}$ and $r'_{2}$ for each case. Notice that if $r_{i}=0$ (resp. $r'_{i}=0$), then $D_{i}$ doesn't pass through the singular point $P_{1}$ (resp. $P_{2}$) of $X$ on $D_{1}$.
    \end{itemize}
    
    \begin{scriptsize}
    \begin{center}
    \begin{longtable}{|@{\makebox[3em][r]{\rownumber\space}}|C|C|C|C|C|C|C|C|}
    \caption{Potential numerical cases when $D_{1}$ contains $2$ singularities
    and $d'=0$}
    \label{table: potential numerical cases for 2 singularities, X}
    \\
    \hline
    \multicolumn{1}{@{\makebox[3em][r]{Case No.~}}|C|}{m_{1}}
    & m_{2} & a & b_{2} & \sum\limits_{i=2}^{n} r_{i}b_{i} & \sum\limits_{i=2}^{n} r'_{i}b_{i} & r_{2}\leq & r'_{2}\leq \\\hline
    
    2 & 3 & 127260 & \frac{1}{7} & \frac{501}{505} & \frac{7}{12} & 0 & 0 \\ \hline
    
    2 & 11 & 104720 & \frac{1}{7} & \frac{29}{40} & \frac{16}{17} & 0 & 0 \\ \hline 

    2 & 17 & 104720 & \frac{1}{7} & \frac{29}{40} & \frac{10}{11} & 0 & 0 \\ \hline
    
    2 & 187 & 104720 & \frac{1}{7} & \frac{29}{40} & 0 & 0 & 0 \\ \hline
    
    3 & 10 & 81900 & \frac{1}{7} & \frac{7}{12} & \frac{437}{455} & 0 & 0 \\ \hline
    
    3 & 13 & 81900 & \frac{1}{7} & \frac{7}{12} & \frac{166}{175} & 0 & 0 \\ \hline
    
    3 & 26 & 81900 & \frac{1}{7} & \frac{7}{12} & \frac{157}{175} & 0 & 0 \\ \hline
    
    3 & 65 & 81900 & \frac{1}{7} & \frac{7}{12} & \frac{26}{35} & 0 & 1 \\ \hline
    
    3 & 83 & 104580 & \frac{1}{7} & \frac{7}{12} & \frac{47}{70} & 0 & 0 \\ \hline
    
    3 & 101 & 127260 & \frac{1}{7} & \frac{7}{12} & \frac{3}{5} & 0 & 0 \\ \hline
    
    3 & 130 & 81900 & \frac{1}{7} & \frac{7}{12} & \frac{17}{35} & 0 & 2 \\ \hline
    
    3 & 139 & 87570 & \frac{1}{7} & \frac{7}{12} & \frac{9}{20} & 0 & 0 \\ \hline
    
    3 & 166 & 104580 & \frac{1}{7} & \frac{7}{12} & \frac{12}{35} & 0 & 1 \\ \hline
    
    3 & 202 & 127260 & \frac{1}{7} & \frac{7}{12} & \frac{1}{5} & 0 & 0 \\ \hline
    
    4 & 7 & 82320 & \frac{1}{7} & \frac{9}{20} & \frac{283}{294} & 0 & 0 \\ \hline 

    4 & 11 & 104720 & \frac{1}{7} & \frac{9}{20} & \frac{16}{17} & 0 & 0 \\ \hline
    
    4 & 14 & 82320 & \frac{1}{7} & \frac{9}{20} & \frac{136}{147} & 0 & 0 \\ \hline
    
    4 & 17 & 104720 & \frac{1}{7} & \frac{9}{20} & \frac{10}{11} & 0 & 0 \\ \hline
    
    4 & 49 & 82320 & \frac{1}{7} & \frac{9}{20} & \frac{31}{42} & 0 & 0 \\ \hline
    
    4 & 98 & 82320 & \frac{1}{7} & \frac{9}{20} & \frac{10}{21} & 0 & 1 \\ \hline
    
    4 & 187 & 104720 & \frac{1}{7} & \frac{9}{20} & 0 & 0 & 0 \\ \hline
    
    \end{longtable}
    \end{center}
    \end{scriptsize}

    Now we discuss the possible dual graphs on the minimal resolution $X^{min}$ of $X$ that are possibly geometrically realizable. The dual graph satisfies the following properties:
    \begin{enumerate}[align=left, label=(\arabic*)]
        \item For every $2 \leq i \leq n$, the divisor $D'_{i}$ does not pass through $E_{1} \cap D'_{1}$ and $E_{2} \cap D'_{1}$. In particular, we have $D'_{i} \cdot D'_{1} = w_{i}$. Indeed, if $D'_{i}$ passes through $E_{1} \cap D'_{1}$, then we have $r_{i} \geq m_{1}$. We can then replace $r_{i}$ by $r_{i} - m_{1}$ and replace $w_{i}$ by $w_{i}+1$ to create a new case satisfying inequalities (\ref{inequality 2~}). However, there is no such case in Table \ref{table: potential numerical cases for 2 singularities, X}. Similarly, $D'_{i}$ does not pass through $E_{2} \cap D'_{1}$.
        \item The divisor $D'_{2}$ is a horizontal section. Indeed, in all cases of Table \ref{table: potential numerical cases for 2 singularities, X}, we have $w_{2}=1$ and $w_{i}=0$ for all $i > 2$. By (1), we have $D'_{2} \cdot D'_{1} = w_{2} = 1$. Hence $D'_{2}$ is a horizontal section.
        \item The divisor $D'_{i}$ lies in a reducible fibre for $2 < i \leq n$. Indeed, by (1), $D'_{i} \cdot D'_{1} = w_{i} = 0$ for $i > 2$. Hence $D'_{i}$ is vertical for $i > 2$. Suppose that $D'_{i}$ does not lie on a reducible fibre for some $i > 2$. Then $D'_{i} \cdot E_{1} > 0$ and $D'_{i} \cdot E_{2} > 0$. We have $r_{i}>0$ and $r'_{i}>0$. In particular, we have
        $$
        \sum\limits_{j=2}^{n} r_{j}b_{j} - r_{2}b_{2} - b_{i}, \, \sum\limits_{j=2}^{n} r'_{j}b_{j} - r'_{2}b_{2} - b_{i} \in \{\sum\limits_{i=1}^{j}l_{i}\beta_{i} \mid j \in \mathbb{Z}_{>0}, l_{i} \in \mathbb{Z}_{\geq 0}, \beta_{i} \in \Gamma(\mathcal{N}_{1}) \}.
        $$
        We can verify in Table \ref{table: potential numerical cases for 2 singularities, X} that this case only happens in Case 12, where $r_{i}=r'_{i}=1$. Write $D'_{1} \sim hD'_{i}$ for some positive integer $h$. By the projection formula, we have
        \begin{equation}
            \begin{cases}
                (D'_{1}+\frac{q_{1}}{m_{1}}E_{1} + \frac{q_{2}}{m_{2}}E_{2}) \cdot E_{1} = 0 \\
                (D'_{1}+\frac{q_{1}}{m_{1}}E_{1} + \frac{q_{2}}{m_{2}}E_{2}) \cdot E_{2} = 0
            \end{cases}
        \end{equation}
        Hence we have 
        \begin{equation}\label{equation 17}
        E_{1}^{2} = -\frac{m_{1}}{q_{1}}, \qquad E_{2}^{2} = -\frac{m_{2}}{q_{2}}.
        \end{equation}
        Moreover, by \thref{prop: surface adjunction} (3) (b) and the projection formula, we have
        $$
        (D'_{i}+\frac{r_{i}}{m_{1}}E_{1}+\frac{r'_{i}}{m_{2}}E_{2})\cdot E_{1} = 0
        $$
        $$
        (D'_{i}+\frac{r_{i}}{m_{1}}E_{1}+\frac{r'_{i}}{m_{2}}E_{2})\cdot E_{2} = 0.
        $$
        Using (\thref{equation 17}), we have
        $$
        D'_{i} \cdot E_{1} = \frac{r_{i}}{q_{1}}, \qquad D'_{i} \cdot E_{2} = \frac{r'_{i}}{q_{2}}.
        $$
        Now we have
        $$
        1 = D'_{1} \cdot E_{1} = hD'_{i} \cdot E_{1} = \frac{hr_{i}}{q_{1}}
        $$
        $$
        1 = D'_{1} \cdot E_{2} = hD'_{i} \cdot E_{2} = \frac{hr'_{i}}{q_{2}}.
        $$
        Recall that in Case 12 of Table \ref{table: potential numerical cases for 2 singularities, X}, we have $r_{i}=r'_{i}=1$. Hence $q_{2} = q_{1} \leq m_{1}-1 = 2$ and
        $$
        a(\frac{q_{1}}{m_{1}}+\frac{q_{2}}{m_{2}}) \leq 87570 \times (\frac{2}{3}+\frac{2}{139})<127260,
        $$
        a contradiction. Hence we conclude that $D'_{i}$ lies in a reducible fibre for $i > 2$. In particular, since $\rho(X'/\mathbb{P}^1) = 2$, there is only one reducible fibre with two irreducible components. Hence $n \leq 4$. By abusing notation, we let $D'_{3}$ and $D'_{4}$ be the two irreducible components of the reducible fibre (the case where $n=3$ is formally written as $n=4$ with $b_{4}=0$).
        \item The divisors $D'_{3}$ and $D'_{4}$ intersect at exactly one point. Indeed, we have $D'_{3} \cap D'_{4} \neq \emptyset$ since they lie in the same reducible fibre. If they intersect at $\geq 2$ points, then there is a loop of curves in the minimal resolution $X^{min} \rightarrow \mathbb{P}^1$. However, by running MMP on $X^{min}$ over $\mathbb{P}^1$, we conclude that $X^{min}$ can be obtained by a sequence of blow-ups from a Hirzebruch surface. Hence the dual graph of every fibre must be a tree, a contradiction.
        \item For $i \geq 3$, either $r_{i}=0,r'_{i}=1$ or $r_{i}=1,r'_{i}=0$, including the formal case $b_{4}=0$ introduced in (3). Indeed, since $D'_{i}\cdot D'_{1} = w_{i} = 0$ by (1) and $D_{i} \cdot D_{1}>0$, at least one of $r_{i}$ and $r'_{i}$ is nonzero. 
        
        Assume first that $b_{3},b_{4}>0$. We can verify in Table \ref{table: potential numerical cases for 2 singularities, X} that we always have
        $$
        \sum\limits_{j=2}^{n} r_{j}b_{j} - r_{2}b_{2} - 2b_{i} \notin \{\sum\limits_{i=1}^{j}l_{i}\beta_{i} \mid j \in \mathbb{Z}_{>0}, l_{i} \in \mathbb{Z}_{\geq 0}, \beta_{i} \in \Gamma(\mathcal{N}_{1}) \}.
        $$
        Hence $r_{i} \leq 1$ for $i=3,4$. We claim that $r_{i},r'_{i}$ cannot be nonzero at the same time for $i=3,4$. Indeed, by the arguments in (3), $r_{i},r'_{i}>0$ only happens in Case 12 of Table \ref{table: potential numerical cases for 2 singularities, X}. In this case, we have 
        $$
        \frac{1}{3}, \frac{1}{4}, \frac{1}{5} \in \{b_{i} \mid i=3, 4, \cdots,n\}.
        $$
        Hence $n \geq 5$, contradicting (3). Now it suffices to show that $r'_{i} \leq 1$. Suppose to the contrary that $r'_{i} \geq 2$, then 
        $$
        \sum\limits_{j=2}^{n} r'_{j}b_{j} - r'_{2}b_{2} - 2b_{i} \in \{\sum\limits_{i=1}^{j}l_{i}\beta_{i} \mid j \in \mathbb{Z}_{>0}, l_{i} \in \mathbb{Z}_{\geq 0}, \beta_{i} \in \Gamma(\mathcal{N}_{1}) \}.
        $$
        One can verify in Table \ref{table: potential numerical cases for 2 singularities, X} that only Cases 8 and 10 satisfy the above condition. Without loss of generality, we can assume that $r'_{4} \geq 2$. Then we have $r_{3}=1$, $r'_{3}=0$. By \thref{prop: surface adjunction} (3) (b) and the projection formula, we have
        \begin{equation}
            \begin{cases}
                (D'_{1}+\frac{q_{1}}{m_{1}}E_{1} + \frac{q_{2}}{m_{2}}E_{2}) \cdot D'_{3} = D_{1} \cdot D_{3} = \frac{1}{m_{1}} \\
                (D'_{1}+\frac{q_{1}}{m_{1}}E_{1} + \frac{q_{2}}{m_{2}}E_{2}) \cdot D'_{4} = D_{1} \cdot D_{4} = \frac{r'_{4}}{m_{2}}
            \end{cases}
        \end{equation}
        Hence we have 
        \begin{equation}\label{equation 15}
        E_{1} \cdot D'_{3}=\frac{1}{q_{1}}, \qquad E_{1} \cdot D'_{4}=0, \qquad E_{2} \cdot D'_{3}=0,\qquad E_{2} \cdot D'_{4} = \frac{r'_{4}}{q_{2}}.
        \end{equation}
        Since $D'_{1}$ is a smooth fibre, we can write $D'_{1} \sim l_{3}D'_{3}+l_{4}D'_{4}$ for some positive integers $l_{3},l_{4}$. Taking the intersection with $E_{1}$ and $E_{2}$, we have 
        \begin{equation}\label{equation 18}
            \begin{cases}
                (l_{3}D'_{3}+l_{4}D'_{4}) \cdot E_{1} = 1 \\
                (l_{3}D'_{3}+l_{4}D'_{4}) \cdot E_{2} = 1
            \end{cases}
        \end{equation}
        Combining (\thref{equation 15}) and (\thref{equation 18}), we have
        \begin{equation}\label{equation 19}
        l_{3} = q_{1}, \qquad l_{4} = \frac{q_{2}}{r'_{4}}.
        \end{equation}
        In particular, we have $r'_{4}|q_{2}$.

        Take the minimal resolution $\pi': X^{min}\rightarrow X'$ as in \thref{construction: extract to make Cartier} and run the relative MMP for the morphism $X^{min} \rightarrow \mathbb{P}^1$. The outcome is a Hirzebruch surface $S$. The diagram is as follows:    
        $$
        \begin{tikzcd}
        & &  X^{min} \arrow[ld,swap,"\pi'"] \arrow[rd,"\psi"] & \\
        X & X' \arrow[l] \arrow[d] & & S \arrow[d] \\
        & \mathbb{P}^1 \arrow[equal,rr] & & \mathbb{P}^1
        \end{tikzcd}
        $$
        Let $\psi:X^{min}\rightarrow S$ be the morphism in the above diagram. Choose a fibre $f_{S}$ of $S\rightarrow\mathbb{P}^{1}$. Since $E_{1,S}$ and $E_{2,S}$ are sections, there exists $a\in\mathbb{Z}$ such that
        $$
        E_{1,S}-E_{2,S}+af_{S} \sim 0.
        $$
        Thus there is a nonzero rational function $\varphi\in k(S)^{*}$ with
        $$
        E_{1,S}-E_{2,S}+af_{S}=(\varphi)_{S}.
        $$
        Identify the function fields of $S$, $X^{min}$, and $X'$. Pulling back this Cartier divisor to $X^{min}$ and then pushing forward to $X'$, we obtain
        $$
        \pi'_{*}\psi^{*}(E_{1,S}-E_{2,S}+af_{S}) = (\varphi)_{X'}\sim 0.
        $$
        This divisor has integral coefficients and horizontal part $E_{1}-E_{2}$. Hence we can write
        $$
        (\varphi)_{X'} = E_{1}-E_{2}+c_{3}D'_{3}+c_{4}D'_{4} +\sum\limits_{i}a_{i}f_{i} \sim 0
        $$
        where $c_{3},c_{4},a_{i} \in \mathbb{Z}$ and the $f_{i}$ are smooth fibres. Intersecting with $E_{2}$, we have
        $$
        \frac{m_{2}}{q_{2}} + \frac{r'_{4}c_{4}}{q_{2}} + \sum a_{i} = 0
        $$
        which can be rewritten as
        $$
        m_{2} + r'_{4}c_{4} + q_{2}\sum a_{i} = 0.
        $$
        Since $r'_{4}| q_{2}$, we have $r'_{4}| m_{2}$. Hence $r'_{4}|\gcd(m_{2},q_{2}) = 1$ (cf. \cite[(3.1.11)]{Alexeev92_lc_surface}) and $r'_{4}=1$, a contradiction.

        It remains to consider the formal case introduced in (3) where $b_{4}=0$. This only happens in Cases 4 and 21 in Table \ref{table: potential numerical cases for 2 singularities, X}. In both cases, we have $r_{3}=1$ and $r'_{3}=0$. Similar to the argument in the $b_{4}>0$ case, \thref{prop: surface adjunction} (3) (b) and the projection formula give
        $$
        (D'_{1}+\frac{q_{1}}{m_{1}}E_{1} + \frac{q_{2}}{m_{2}}E_{2}) \cdot D'_{3} = D_{1} \cdot D_{3} = \frac{1}{m_{1}}
        $$
        and
        $$
        (D'_{1}+\frac{q_{1}}{m_{1}}E_{1} + \frac{q_{2}}{m_{2}}E_{2}) \cdot D'_{4} = D_{1} \cdot D_{4} = (D_{1} \cdot D_{4})_{P_{1}} + (D_{1} \cdot D_{4})_{P_{2}} = \frac{r_{4}}{m_{1}} + \frac{r'_{4}}{m_{2}}.
        $$
        Hence we have 
        $$
        E_{1} \cdot D'_{3}=\frac{1}{q_{1}}, \qquad E_{1} \cdot D'_{4}= \frac{r_{4}}{q_{1}} , \qquad E_{2} \cdot D'_{3}=0,\qquad E_{2} \cdot D'_{4} = \frac{r'_{4}}{q_{2}}.
        $$
        Hence, similar to the argument in the $b_{4}>0$ case, we have $r'_{4}=1$ and $l_{4} = \frac{q_{2}}{r'_{4}} = q_{2}$. Now we have
        $$
        1 = D'_{1} \cdot E_{1} = (l_{3}D'_{3} + l_{4}D'_{4}) \cdot E_{1} = \frac{l_{3}+r_{4}q_{2}}{q_{1}}.
        $$
        Suppose to the contrary that $r_{4}>0$, then we have $q_{2}<q_{1}$. 
        
        In Case 4 of Table \ref{table: potential numerical cases for 2 singularities, X}, we have $q_{2}<q_{1}<m_{1}=2$, a contradiction. 
        
        In Case 21 of Table \ref{table: potential numerical cases for 2 singularities, X}, we have $m_{1}=4$. Hence $q_{1} \leq 3$ and $q_{2} \leq 2$. We have
        $$
        a(\frac{q_{1}}{m_{1}}+\frac{q_{2}}{m_{2}}) \leq 104720 \times (\frac{3}{4}+\frac{2}{187})<127260,
        $$
        also a contradiction.
    \end{enumerate}

    Now we use the above properties to complete the discussion for the $d'=0$ case. In all cases of Table \ref{table: potential numerical cases for 2 singularities, X}, we have $r_{2}=0$. We discuss the problem on the basis of the possible value of $r'_{2}$.
    \begin{enumerate}
        \item $r'_{2}=0$. In this case, $D'_{2}$ is disjoint with $E_{1}$ and $E_{2}$. By the projection formula, we have 
        $$
        D'_{2} \cdot D'_{3} = D_{2}\cdot D_{3}>0, \qquad D'_{2} \cdot D'_{4} = D_{2} \cdot D_{4} > 0.
        $$
        Since $D'_{2}$ is a section, it meets the reducible fibre $D'_{3} \cup D'_{4}$ exactly at one point. Hence we conclude that the three divisors $D'_{2},D'_{3}$ and $D'_{4}$ intersect at one point $Q$. We claim that the crepant pullback of the pair $(X,D_{1} + \sum\limits_{i=2}^{n}b_{i}D_{i})$ is not lc at $Q$. Indeed, we have
        \begin{align*}
        \frac{D'_{2}\cdot D'_{3}}{D'_{2}\cdot D'_{4}} &= \frac{D_{2}\cdot D_{3}}{D_{2} \cdot D_{4}} & (\text{the projection formula})\\ 
        &= \frac{D_{1}\cdot D_{3}}{D_{1} \cdot D_{4}} & (\rho(X) = 1)\\ 
        &= \frac{\frac{1}{m_{1}}}{\frac{1}{m_{2}}} & (\thref{prop: surface adjunction} (3) (b)) \\
        &= \frac{m_{2}}{m_{1}}.
        \end{align*}
        By (\thref{equation 19}), we have $D'_{1} \sim q_{1}D'_{3} + q_{2}D'_{4}$. Hence
        $$
        D'_{2} \cdot (q_{1}D'_{3}+q_{2}D'_{4}) = D'_{2} \cdot D'_{1} = 1.
        $$
        Combining the equations
        $$
        \begin{cases}
            m_{1}D'_{2} \cdot D'_{3} - m_{2} D'_{2} \cdot D'_{4} = 0 \\
            q_{1}D'_{2} \cdot D'_{3} + q_{2} D'_{2} \cdot D'_{4} = 1
        \end{cases}
        $$
        we have
        \begin{equation}\label{equation 20}
        D'_{2} \cdot D'_{3} = \frac{m_{2}}{m_{1}q_{2} + m_{2}q_{1}}, \qquad D'_{2} \cdot D'_{4} = \frac{m_{1}}{m_{1}q_{2} + m_{2}q_{1}}.
        \end{equation}
        We also have
        \begin{equation}\label{equation 21}
        \begin{aligned}
        D'_{3} \cdot D'_{4} & = (D'_{3} + \frac{1}{m_{1}}E_{1})(D'_{4} + \frac{1}{m_{2}}E_{2}) & (E_{1} \cdot E_{2} = E_{1}\cdot D'_{4} = E_{2} \cdot D'_{3}=0) \\
        &= D_{3} \cdot D_{4} & (\text{the projection formula}) \\
        &= \frac{(D_{1}\cdot D_{3})(D_{1}\cdot D_{4})}{D_{1}^{2}} & (\rho(X) = 1) \\
        &= \frac{\frac{1}{m_{1}m_{2}}}{\frac{q_{1}}{m_{1}} + \frac{q_{2}}{m_{2}}} & (\thref{prop: surface adjunction} (3) (b)) \\
        &= \frac{1}{m_{1}q_{2} + m_{2}q_{1}}.
        \end{aligned}
        \end{equation}
        Take the minimal resolution $\widetilde{f'}: \widetilde{X'} \rightarrow X'$ of $X'$ at $Q$. Let $\widetilde{D'_{i}}$ be the strict birational transformation of $D'_{i}$ on $\widetilde{X'}$, and $\widetilde{E_{i}}$ be the strict birational transformation of $E_{i}$ on $\widetilde{X'}$. Since $X'$ is klt at $Q$, the dual graph of the exceptional divisor $\mathrm{Exc}(\widetilde{f'})$ is a tree. Moreover, $\widetilde{D'_{2}}$ is still a section of the composite fibration $\widetilde{X'} \rightarrow X' \rightarrow \mathbb{P}^1$. Hence $\widetilde{D'_{2}}$ meets exactly one prime component $\widetilde{F}$ in the reducible fibre $\mathrm{Exc}(\widetilde{f'}) \cup \widetilde{D'_{3}} \cup \widetilde{D'_{4}}$, and $\widetilde{D'_{2}} \cdot \widetilde{F} = 1$. In particular, $\widetilde{F}$ is exceptional over $X'$. We claim that the reducible fibre is a chain of curves such that $\widetilde{D'_{3}}$ and $\widetilde{D'_{4}}$ are at the end of the chain. Indeed, let $L$ be the sum of prime divisors in the path from $\widetilde{D'_{3}}$ to $\widetilde{D'_{4}}$. We have 
        $$
        (K_{\widetilde{X'}} + \widetilde{E_{1}} + \widetilde{E_{2}} + L) \cdot \widetilde{D'_{1}} = 0.
        $$
        For any prime components $L_{i}$ in $L$, adjunction gives
        $$
        (K_{\widetilde{X'}} + \widetilde{E_{1}} + \widetilde{E_{2}} + L) \cdot L_{i} \geq 0.
        $$
        For any prime components $L'$ of the reducible fibre not contained in $L$, adjunction gives
        $$
        (K_{\widetilde{X'}} + \widetilde{E_{1}} + \widetilde{E_{2}} + L) \cdot L' = -2 - L^{\prime 2} + L \cdot L' \geq L \cdot L' \geq 0.
        $$
        If the reducible fibre contains any prime component $L'$ not in $L$, then, by connectedness, there exists a prime component $L'$ such that $L \cdot L' > 0$. Since $\widetilde{D'_{1}}$ is linearly equivalent to a positive linear combination of the prime components of the reducible fibre, we have 
        $$
        (K_{\widetilde{X'}} + \widetilde{E_{1}} + \widetilde{E_{2}} + L) \cdot \widetilde{D'_{1}} > 0,
        $$
        a contradiction. This proves our claim.
        
        Let $g: Y \rightarrow X'$ be the extraction of the associated exceptional divisor $F$. Let $D_{i, Y}$ be the strict birational transformation of $D_{i}$ on $Y$. Then we have 
        $$
        D_{2,Y} \cdot F = 1, \qquad D_{2,Y} \cdot D_{3,Y} = D_{2,Y} \cdot D_{4,Y} = D_{3,Y} \cdot D_{4,Y} = 0.
        $$
        By (\thref{equation 19}), we can write $D_{1,Y} \sim q_{1}D_{3,Y}+q_{2}D_{4,Y}+F$. Now we have
        \begin{align*}
        \frac{D_{2,Y} \cdot F}{D_{3,Y} \cdot F} & = \frac{D_{2,Y} \cdot g^{*}D'_{4}}{D_{3,Y} \cdot g^{*}D'_{4}} & (D_{2,Y} \cdot D_{4,Y} = D_{3,Y} \cdot D_{4,Y} = 0) \\
        &= \frac{D'_{2} \cdot D'_{4}}{D'_{3}\cdot D'_{4}} & (\text{the projection formula}) \\
        &= m_{1} & ((\thref{equation 20}), (\thref{equation 21})). 
        \end{align*}
        Similarly, we have $\frac{D_{2,Y} \cdot F}{D_{4,Y} \cdot F} = m_{2}$. Hence
        $$
        D_{3,Y} \cdot F = \frac{1}{m_{1}}, \qquad D_{4,Y} \cdot F = \frac{1}{m_{2}}.
        $$
        Taking the adjunction of the pair $(Y,D_{1,Y}+\sum\limits_{i=2}^{n}b_{i}D_{i,Y} + \frac{m_{1}-1+\sum\limits_{i=2}^{n}r_{i}b_{i}}{m_{1}}E_{1} + \frac{m_{2}-1+\sum\limits_{i=2}^{n}r'_{i}b_{i}}{m_{2}}E_{2} + F)$ on $F$, we have
        \begin{gather*}
        (K_{Y}+D_{1,Y}+\sum\limits_{i=2}^{n}b_{i}D_{i,Y} + \frac{m_{1}-1+\sum\limits_{i=2}^{n}r_{i}b_{i}}{m_{1}}E_{1} + \frac{m_{2}-1+\sum\limits_{i=2}^{n}r'_{i}b_{i}}{m_{2}}E_{2} + F) \cdot F \\
        = -2 + \mathrm{deg}(\mathrm{Diff}_{F}(b_{2}D_{2,Y}+b_{3}D_{3,Y}+b_{4}D_{4,Y})) \\
        \geq -2 + b_{2} + (1 - \frac{(1-b_{3})}{m_{1}}) + (1 - \frac{(1-b_{4})}{m_{2}}) \\
        = -2 + b_{2} + \frac{m_{1}-1+b_{3}}{m_{1}} + \frac{m_{2}-1+b_{4}}{m_{2}} \\
        = (\frac{q_{1}}{m_{1}}+\frac{q_{2}}{m_{2}})(1-b_{1}) > 0. 
        \end{gather*}
        Here we use (\ref{inequality 2~}) and the inequality 
        $$
        (\mathrm{Diff}_{F}(bD))_{P} \geq \frac{m_{P}-1+r_{P}b}{m_{P}}  = 1 - \frac{1-r_{P}b}{m_{P}} \geq 1 - \frac{r_{P}(1-b)}{m_{P}} = 1 - (1-b)(D \cdot F)_{P}.
        $$
        Since $F^{2}<0$, we conclude that the log discrepancy $a(F,X,D_{1}+\sum\limits_{i=2}^{n}b_{i}D_{i}) < 0$. This contradicts our condition that $(X,D_{1}+\sum\limits_{i=2}^{n}b_{i}D_{i})$ is log canonical.
        \item $r'_{2} \geq 1$. This only happens in Cases 8, 11, 13, 20 in Table \ref{table: potential numerical cases for 2 singularities, X}. In this case, we have $D'_{2} \cdot E_{2} = \frac{r'_{2}}{q_{2}}$. We claim that $D'_{2}$ does not pass through $E_{2} \cap D'_{4}$. Indeed, since $D'_{3} \cap E_{2} = 0$ and $D'_{2} \cdot E_{1} = 0$, by the projection formula we have $D'_{2} \cdot D'_{3} = D_{2} \cdot D_{3} > 0$. Since $D'_{2}$ is a section, it can only intersect the reducible fibre $D'_{3} + D'_{4}$ at one point. Since this intersection point lies in $D'_{3}$, it is not $E_{2} \cap D'_{4}$. Hence $\frac{r'_{2}}{q_{2}} \in \mathbb{Z}_{>0}$. A direct computation of Cases 8, 11, 13, 20 shows that we always have $a(\frac{q_{1}}{m_{1}} + \frac{q_{2}}{m_{2}}) < 127260$.
    \end{enumerate}

    \item $d' = 1$. In this case $d < 3$ and $X_{d'} \cong \mathbb{P}^2$. Moreover, there exists a divisor $D_{i}$, $i\geq2$, intersecting $D_{1}$ at some nonsingular locus of $X$. By \thref{statement: bound by degree}, we have
    $$
    \frac{m_{1}-1 + \sum\limits_{i=2}^{n} r_{i}b_{i}}{m_{1}} + \frac{m_{2} - 1 + \sum\limits_{i=2}^{n} r'_{i}b_{i}}{m_{2}} + \sum\limits_{i=2}^{n} w_{i}b_{i} = 2 + (1 + \frac{q_{1}}{m_{1}} + \frac{q_{2}}{m_{2}})(1-b_{1}) > 2.
    $$
    Similar to the case where $d'=0$, we can assume that $m_{1} \leq 14$, $m_{1} \leq m_{2}$ and $m_{2} \leq 253$. 
    
    We prove that in this case, $b_{1} < 1 - \frac{1}{127260}$. Indeed, suppose to the contrary that $b_{1} \geq 1 - \frac{1}{127260}$. Then by \thref{statement: bound by degree} and \thref{lemma: non-plt case}, we have 
    \begin{multline}\label{inequality 2}
    \frac{m_{1}-1 + \sum\limits_{i=2}^{n} r_{i}b_{i}}{m_{1}} + \frac{m_{2} - 1 + \sum\limits_{i=2}^{n} r'_{i}b_{i}}{m_{2}} + \sum\limits_{i=2}^{n} w_{i}b_{i} = 2 + (1 + \frac{q_{1}}{m_{1}} + \frac{q_{2}}{m_{2}})(1-b_{1}) \\
    \leq 2 + \frac{1 + \frac{q_{1}}{m_{1}} + \frac{q_{2}}{m_{2}}}{127260} \leq 2 + \frac{1 + \frac{m_{1}-1}{m_{1}} + \frac{m_{2}-1}{m_{2}}}{127260}.
    \end{multline}

    We use a computer to enumerate all the linear combinations satisfying inequalities (\ref{inequality 2}). The results are listed in Table \ref{table: potential numerical cases for 2 singularities, I} below. We give some explanations of Table \ref{table: potential numerical cases for 2 singularities, I}:
    \begin{itemize}
        \item Each row is distinguished by the following numerical data: $m_{1}$, $m_{2}$, $w_{i}$ (implicitly), $b_{i}$, $\sum\limits_{i=2}^{n}r_{i}b_{i}$ and $\sum\limits_{i=2}^{n}r'_{i}b_{i}$. Notice that two cases with different $r_{i}$ and $r'_{i}$ but the same $\sum\limits_{i=2}^{n}r_{i}b_{i}$ and $\sum\limits_{i=2}^{n}r'_{i}b_{i}$ will be considered as the same case.
        \item If two cases are related by a permutation of the indices of $b_{i}, i=2,3,\cdots, n$, we will recognize them as the same case.
        \item To simplify the expression in the table, we replace the $b_{1}$ column by the column of $a = \frac{1}{(1-b_{1})(1 + \frac{q_{1}}{m_{1}}+\frac{q_{2}}{m_{2}})}$. Then in inequalities (\ref{inequality 2}) we can write $2 + (1-b_{1})(1 + \frac{q_{1}}{m_{1}}+\frac{q_{2}}{m_{2}}) = 2 + \frac{1}{a}$. Notice that the condition $b_{1} \geq 1 - \frac{1}{127260}$ is equivalent to $a(1 + \frac{q_{1}}{m_{1}}+\frac{q_{2}}{m_{2}}) \geq 127260$.
        \item To remove numerically redundant cases, we can assume that for all $2 \leq i \leq n$, we have $w_{i}=0$ or $w_{i}=1$. Indeed, we can formally split the term $w_{i}b_{i}$ in the sum $\frac{m_{1}-1 + \sum\limits_{i=2}^{n} r_{i}b_{i}}{m_{1}} + \frac{m_{2} - 1 + \sum\limits_{i=2}^{n} r'_{i}b_{i}}{m_{2}} + \sum\limits_{i=2}^{n} w_{i}b_{i}$ into $w_{i}$ copies of $b_{i}$. Conversely, there exists a sum $\frac{m_{1}-1 + \sum\limits_{i=2}^{n} r_{i}b_{i}}{m_{1}} + \frac{m_{2} - 1 + \sum\limits_{i=2}^{n} r'_{i}b_{i}}{m_{2}} + \sum\limits_{i=2}^{n} w_{i}b_{i}$ with $w_{i}>1$ for some $i$ if and only if there exists a sum in Table \ref{table: potential numerical cases for 2 singularities, I} where $w_{i}=w_{j}=1$ and $b_{i}=b_{j}$ for some $i \neq j$ (later we will see that this doesn't happen).
        \item A number $b_{i}$ is displayed in Table \ref{table: potential numerical cases for 2 singularities, I} if and only if $w_{i}=1$. For the numbers $b_{i}$ where $w_{i} = 0$, we must have $r_{i} > 0$ or $r'_{i}>0$. In this case, $b_{i}$ is not displayed in the table but still contributes to the sum $\sum\limits_{i=2}^{n} r_{i}b_{i}$ or $\sum\limits_{i=2}^{n} r'_{i}b_{i}$.
        \item In the last column, we give the upper bounds of the numbers $r_{2}$ and $r'_{2}$ for each case. Notice that if $r_{i}=0$ (resp. $r'_{i}=0$), then $D_{i}$ doesn't pass through the singular point $P_{1}$ (resp. $P_{2}$) of $X$ on $D_{1}$.
    \end{itemize}
    
    \begin{scriptsize}
    \begin{center}
    \begin{longtable}{|@{\makebox[3em][r]{\rownumber\space}}|C|C|C|C|C|C|C|C|}
    \caption{Potential numerical cases when $D_{1}$ contains $2$ singularities and $d'=1$}
    \label{table: potential numerical cases for 2 singularities, I} \\
    \hline
    \multicolumn{1}{@{\makebox[3em][r]{Case No.~}}|C|}{m_{1}} 
    & m_{2} & a & b_{2} & \sum\limits_{i=2}^{n} r_{i}b_{i} & \sum\limits_{i=2}^{n} r'_{i}b_{i} & r_{2}\leq & r'_{2}\leq \\ \hline
    
    2 & 2 & 66360 & \frac{1}{3} & \frac{12}{35} & \frac{313}{316} & 0 & 0 \\
    \hline
    2 & 2 & 82320 & \frac{1}{7} & \frac{29}{40} & \frac{1018}{1029} & 0 & 0 \\
    \hline
    2 & 2 & 65520 & \frac{1}{7} & \frac{29}{40} & \frac{463}{468} & 0 & 0 \\
    \hline
    2 & 2 & 82320 & \frac{1}{7} & \frac{1018}{1029} & \frac{29}{40} & 0 & 0 \\
    \hline
    2 & 2 & 65520 & \frac{1}{7} & \frac{463}{468} & \frac{29}{40} & 0 & 0 \\
    \hline
    2 & 2 & 66360 & \frac{1}{3} & \frac{313}{316} & \frac{12}{35} & 0 & 0 \\
    \hline
    2 & 3 & 68880 & \frac{1}{7} & \frac{29}{40} & \frac{1412}{1435} & 0 & 0 \\
    \hline
    2 & 3 & 65520 & \frac{1}{7} & \frac{29}{40} & \frac{307}{312} & 0 & 0 \\
    \hline
    2 & 3 & 127260 & \frac{1}{7} & \frac{501}{505} & \frac{7}{12} & 0 & 0 \\
    \hline
    2 & 3 & 104580 & \frac{1}{7} & \frac{2882}{2905} & \frac{7}{12} & 0 & 0 \\
    \hline
    2 & 3 & 81900 & \frac{1}{7} & \frac{2257}{2275} & \frac{7}{12} & 0 & 0 \\
    \hline
    2 & 3 & 74340 & \frac{1}{7} & \frac{878}{885} & \frac{7}{12} & 0 & 0 \\
    \hline
    2 & 3 & 59220 & \frac{1}{7} & \frac{1632}{1645} & \frac{7}{12} & 0 & 0 \\
    \hline
    2 & 4 & 82320 & \frac{1}{7} & \frac{1018}{1029} & \frac{9}{20} & 0 & 0 \\
    \hline
    2 & 4 & 65520 & \frac{1}{7} & \frac{463}{468} & \frac{9}{20} & 0 & 0 \\
    \hline
    2 & 6 & 65520 & \frac{1}{7} & \frac{29}{40} & \frac{151}{156} & 0 & 0 \\
    \hline
    2 & 7 & 82320 & \frac{1}{7} & \frac{29}{40} & \frac{283}{294} & 0 & 0 \\
    \hline
    2 & 7 & 64680 & \frac{1}{7} & \frac{61}{84} & \frac{369}{385} & 1 & 0 \\
    \hline
    2 & 8 & 53760 & \frac{1}{3} & \frac{12}{35} & \frac{431}{448} & 0 & 0 \\
    \hline
    2 & 11 & 104720 & \frac{1}{7} & \frac{29}{40} & \frac{16}{17} & 0 & 0 \\
    \hline
    2 & 11 & 64680 & \frac{1}{7} & \frac{61}{84} & \frac{229}{245} & 1 & 0 \\
    \hline
    2 & 13 & 65520 & \frac{1}{7} & \frac{29}{40} & \frac{67}{72} & 0 & 0 \\
    \hline
    2 & 14 & 82320 & \frac{1}{7} & \frac{29}{40} & \frac{136}{147} & 0 & 0 \\
    \hline
    2 & 16 & 53760 & \frac{1}{3} & \frac{12}{35} & \frac{207}{224} & 0 & 0 \\
    \hline
    2 & 17 & 104720 & \frac{1}{7} & \frac{29}{40} & \frac{10}{11} & 0 & 0 \\
    \hline
    2 & 26 & 65520 & \frac{1}{7} & \frac{29}{40} & \frac{31}{36} & 0 & 0 \\
    \hline
    2 & 32 & 53760 & \frac{1}{3} & \frac{12}{35} & \frac{95}{112} & 0 & 0 \\
    \hline
    2 & 39 & 65520 & \frac{1}{7} & \frac{29}{40} & \frac{19}{24} & 0 & 0 \\
    \hline
    2 & 41 & 68880 & \frac{1}{7} & \frac{29}{40} & \frac{82}{105} & 0 & 0 \\
    \hline
    2 & 49 & 82320 & \frac{1}{7} & \frac{29}{40} & \frac{31}{42} & 0 & 0 \\
    \hline
    2 & 64 & 53760 & \frac{1}{3} & \frac{12}{35} & \frac{39}{56} & 0 & 0 \\
    \hline
    2 & 77 & 64680 & \frac{1}{7} & \frac{61}{84} & \frac{19}{35} & 1 & 1 \\
    \hline
    2 & 78 & 65520 & \frac{1}{7} & \frac{29}{40} & \frac{7}{12} & 0 & 0 \\
    \hline
    2 & 79 & 66360 & \frac{1}{3} & \frac{12}{35} & \frac{5}{8} & 0 & 0 \\
    \hline
    2 & 98 & 82320 & \frac{1}{7} & \frac{29}{40} & \frac{10}{21} & 0 & 1 \\
    \hline
    2 & 101 & 84840 & \frac{1}{7} & \frac{61}{84} & \frac{2}{5} & 1 & 0 \\
    \hline
    2 & 103 & 57680 & \frac{1}{7} & \frac{29}{40} & \frac{9}{20} & 0 & 0 \\
    \hline
    2 & 107 & 59920 & \frac{1}{7} & \frac{29}{40} & \frac{3}{7} & 0 & 3 \\
    \hline
    2 & 123 & 68880 & \frac{1}{7} & \frac{29}{40} & \frac{12}{35} & 0 & 1 \\
    \hline
    2 & 128 & 53760 & \frac{1}{3} & \frac{12}{35} & \frac{11}{28} & 0 & 0 \\
    \hline
    2 & 158 & 66360 & \frac{1}{3} & \frac{12}{35} & \frac{1}{4} & 0 & 0 \\
    \hline
    2 & 187 & 104720 & \frac{1}{7} & \frac{29}{40} & 0 & 0 & 0 \\
    \hline
    3 & 4 & 68880 & \frac{1}{7} & \frac{1412}{1435} & \frac{9}{20} & 0 & 0 \\
    \hline
    3 & 4 & 65520 & \frac{1}{7} & \frac{307}{312} & \frac{9}{20} & 0 & 0 \\
    \hline
    3 & 5 & 81900 & \frac{1}{7} & \frac{7}{12} & \frac{446}{455} & 0 & 0 \\
    \hline
    3 & 7 & 54684 & \frac{1}{7} & \frac{7}{12} & \frac{211}{217} & 0 & 0 \\
    \hline
    3 & 10 & 81900 & \frac{1}{7} & \frac{7}{12} & \frac{437}{455} & 0 & 0 \\
    \hline
    3 & 11 & 63756 & \frac{1}{7} & \frac{7}{12} & \frac{22}{23} & 0 & 0 \\
    \hline
    3 & 13 & 81900 & \frac{1}{7} & \frac{7}{12} & \frac{166}{175} & 0 & 0 \\
    \hline
    3 & 23 & 63756 & \frac{1}{7} & \frac{7}{12} & \frac{10}{11} & 0 & 0 \\
    \hline
    3 & 26 & 81900 & \frac{1}{7} & \frac{7}{12} & \frac{157}{175} & 0 & 0 \\
    \hline
    3 & 31 & 54684 & \frac{1}{7} & \frac{7}{12} & \frac{43}{49} & 0 & 0 \\
    \hline
    3 & 41 & 51660 & \frac{1}{7} & \frac{7}{12} & \frac{88}{105} & 0 & 0 \\
    \hline
    3 & 47 & 59220 & \frac{1}{7} & \frac{7}{12} & \frac{57}{70} & 0 & 1 \\
    \hline
    3 & 59 & 74340 & \frac{1}{7} & \frac{7}{12} & \frac{23}{30} & 0 & 0 \\
    \hline
    3 & 65 & 81900 & \frac{1}{7} & \frac{7}{12} & \frac{26}{35} & 0 & 1 \\
    \hline
    3 & 82 & 51660 & \frac{1}{7} & \frac{7}{12} & \frac{71}{105} & 0 & 1 \\
    \hline
    3 & 83 & 104580 & \frac{1}{7} & \frac{7}{12} & \frac{47}{70} & 0 & 0 \\
    \hline
    3 & 94 & 59220 & \frac{1}{7} & \frac{7}{12} & \frac{22}{35} & 0 & 3 \\
    \hline
    3 & 101 & 127260 & \frac{1}{7} & \frac{7}{12} & \frac{3}{5} & 0 & 0 \\
    \hline
    3 & 103 & 64890 & \frac{1}{7} & \frac{7}{12} & \frac{83}{140} & 0 & 1 \\
    \hline
    3 & 118 & 74340 & \frac{1}{7} & \frac{7}{12} & \frac{8}{15} & 0 & 0 \\
    \hline
    3 & 130 & 81900 & \frac{1}{7} & \frac{7}{12} & \frac{17}{35} & 0 & 2 \\
    \hline
    3 & 139 & 87570 & \frac{1}{7} & \frac{7}{12} & \frac{9}{20} & 0 & 0 \\
    \hline
    3 & 152 & 47880 & \frac{1}{7} & \frac{7}{12} & \frac{2}{5} & 0 & 0 \\
    \hline
    3 & 166 & 104580 & \frac{1}{7} & \frac{7}{12} & \frac{12}{35} & 0 & 1 \\
    \hline
    3 & 190 & 47880 & \frac{1}{7} & \frac{7}{12} & \frac{1}{4} & 0 & 0 \\
    \hline
    3 & 202 & 127260 & \frac{1}{7} & \frac{7}{12} & \frac{1}{5} & 0 & 0 \\
    \hline
    3 & 217 & 54684 & \frac{1}{7} & \frac{7}{12} & \frac{1}{7} & 0 & 1 \\
    \hline
    3 & 253 & 63756 & \frac{1}{7} & \frac{7}{12} & 0 & 0 & 0 \\
    \hline
    4 & 5 & 50400 & \frac{1}{7} & \frac{481}{504} & \frac{12}{35} & 0 & 1 \\
    \hline
    4 & 6 & 65520 & \frac{1}{7} & \frac{9}{20} & \frac{151}{156} & 0 & 0 \\
    \hline
    4 & 7 & 82320 & \frac{1}{7} & \frac{9}{20} & \frac{283}{294} & 0 & 0 \\
    \hline
    4 & 11 & 104720 & \frac{1}{7} & \frac{9}{20} & \frac{16}{17} & 0 & 0 \\
    \hline
    4 & 13 & 65520 & \frac{1}{7} & \frac{9}{20} & \frac{67}{72} & 0 & 0 \\
    \hline
    4 & 14 & 82320 & \frac{1}{7} & \frac{9}{20} & \frac{136}{147} & 0 & 0 \\
    \hline
    4 & 17 & 104720 & \frac{1}{7} & \frac{9}{20} & \frac{10}{11} & 0 & 0 \\
    \hline
    4 & 26 & 65520 & \frac{1}{7} & \frac{9}{20} & \frac{31}{36} & 0 & 0 \\
    \hline
    4 & 29 & 48720 & \frac{1}{7} & \frac{9}{20} & \frac{71}{84} & 0 & 0 \\
    \hline
    4 & 39 & 65520 & \frac{1}{7} & \frac{9}{20} & \frac{19}{24} & 0 & 0 \\
    \hline
    4 & 41 & 68880 & \frac{1}{7} & \frac{9}{20} & \frac{82}{105} & 0 & 0 \\
    \hline
    4 & 49 & 82320 & \frac{1}{7} & \frac{9}{20} & \frac{31}{42} & 0 & 0 \\
    \hline
    4 & 78 & 65520 & \frac{1}{7} & \frac{9}{20} & \frac{7}{12} & 0 & 0 \\
    \hline
    4 & 87 & 48720 & \frac{1}{7} & \frac{9}{20} & \frac{15}{28} & 0 & 2 \\
    \hline
    4 & 98 & 82320 & \frac{1}{7} & \frac{9}{20} & \frac{10}{21} & 0 & 1 \\
    \hline
    4 & 103 & 57680 & \frac{1}{7} & \frac{9}{20} & \frac{9}{20} & 0 & 0 \\
    \hline
    4 & 107 & 59920 & \frac{1}{7} & \frac{9}{20} & \frac{3}{7} & 0 & 3 \\
    \hline
    4 & 123 & 68880 & \frac{1}{7} & \frac{9}{20} & \frac{12}{35} & 0 & 1 \\
    \hline
    4 & 187 & 104720 & \frac{1}{7} & \frac{9}{20} & 0 & 0 & 0 \\
    \hline
    5 & 6 & 50400 & \frac{1}{7} & \frac{12}{35} & \frac{313}{336} & 1 & 0 \\
    \hline
    5 & 8 & 50400 & \frac{1}{7} & \frac{12}{35} & \frac{229}{252} & 1 & 0 \\
    \hline
    5 & 12 & 50400 & \frac{1}{7} & \frac{12}{35} & \frac{145}{168} & 1 & 0 \\
    \hline
    5 & 24 & 50400 & \frac{1}{7} & \frac{12}{35} & \frac{61}{84} & 1 & 1 \\
    \hline
    
    \end{longtable}
    \end{center}
    \end{scriptsize}

    Now we examine the possible values of $q_{1}$ and $q_{2}$ by discussing the possible dual graphs on the minimal resolution $X^{min}$ of $X$ that are possibly geometrically realizable. The dual graph satisfies the following properties:
    \begin{enumerate}[align=left, label=(\arabic*)]
        \item For every $2 \leq i \leq n$, the divisor $D'_{i}$ does not pass through $E_{1} \cap D'_{1}$ and $E_{2} \cap D'_{1}$. In particular, we have $D'_{i} \cdot D'_{1} = w_{i}$. Indeed, if $D'_{i}$ passes through $E_{1} \cap D'_{1}$, then we have $r_{i} \geq m_{1}$. We can then replace $r_{i}$ by $r_{i} - m_{1}$ and replace $w_{i}$ by $w_{i}+1$ to create a new case satisfying inequalities (\ref{inequality 2}). However, there is no such case in Table \ref{table: potential numerical cases for 2 singularities, I}. Similarly, $D'_{i}$ does not pass through $E_{2} \cap D'_{1}$.
        \item For $2 \leq i \leq n$, there is an edge between the vertex of $D_{i}$ and the vertex of $D_{1}$ if and only if $i=2$. Indeed, since the morphism $X^{min} \rightarrow X'$ is an isomorphism near $D'_{1}$, the edge can be computed by $D'_{i} \cdot D'_{1} = w_{i}$ by (1). As explained before Table \ref{table: potential numerical cases for 2 singularities, I}, we have $w_{i} \in \{ 0,1\}$. Moreover, as explained before Table \ref{table: potential numerical cases for 2 singularities, I}, we have $w_{2}=1$ and $w_{i}=0$ for all $i > 2$.
        \item There exists three chains $\Gamma$, $\Gamma_{1}$, $\Gamma_{2}$ of divisors on $X^{min}$ of type $\bigcirc$, $\otimes$, or $\odot$, such that
        \begin{itemize}
            \item $\Gamma$ connects the vertices of the strict birational transformations of $E_{1}$ and $E_{2}$ on $X^{min}$. Such a chain $\Gamma$ is unique. Moreover, $\Gamma$ has one of the following forms:
            \begin{figure}[H]
            \centering
            \begin{tikzpicture}
            \node[inner sep=0, draw, circle] (0) {$E_{1}$};
            \node[] (1) [right of=0] {$\cdots_{\bigcirc}$};
            \node[draw, circle, label=center:$\times$] (2) [right of=1] {};
            \node[] (3) [right of=2] {$\cdots_{\bigcirc}$};
            \node[inner sep=0, draw, circle] (4) [right of=3] {$E_{2}$};
            \draw (0) -- (1);
            \draw (1) -- (2);
            \draw (2) -- (3);
            \draw (3) -- (4);
            \end{tikzpicture}
            \end{figure}
            \begin{figure}[H]
            \centering
            \begin{tikzpicture}
            \node[inner sep=0, draw, circle] (0) {$E_{1}$};
            \node[] (1) [right of=0] {$\cdots_{\bigcirc}$};
            \node[draw, circle, label=center:$\times$] (2) [right of=1] {};
            \node[] (3) [right of=2] {$\cdots_{\odot, \otimes}$};
            \node[draw, circle, label=center:$\times$] (4) [right of=3] {};
            \node[] (5) [right of=4] {$\cdots_{\bigcirc}$};
            \node[inner sep=0, draw, circle] (6) [right of=5] {$E_{2}$};
            \draw (0) -- (1);
            \draw (1) -- (2);
            \draw (2) -- (3);
            \draw (3) -- (4);
            \draw (4) -- (5);
            \draw (5) -- (6);
            \end{tikzpicture}
            \end{figure}
            In particular, $\Gamma$ contains a vertex of type $\otimes$.
            \item For $i=1,2$, $\Gamma_{i}$ connects the vertices of the strict birational transformations of $E_{i}$ and $D_{2}$ on $X^{min}$ (this also includes the case where the vertices of the strict birational transformations of $E_{i}$ and $D_{2}$ on $X^{min}$ are directly connected by an edge, which is formally treated as $\Gamma_{i} = \emptyset$). Such a chain $\Gamma_{i}$ is unique. Moreover, $\Gamma_{i}$ has one of the following forms:
            \begin{figure}[H]
            \centering
            \begin{tikzpicture}
            \node[inner sep=0, draw, circle] (0) {$D_{2}$};
            \node[inner sep=0, draw, circle] (1) [right of=0] {$E_{i}$};
            \draw (0) -- (1);
            \end{tikzpicture}
            \end{figure}
            \begin{figure}[H]
            \centering
            \begin{tikzpicture}
            \node[inner sep=0, draw, circle] (0) {$D_{2}$};
            \node[] (1) [right of=0] {$\cdots_{\odot}$};
            \node[draw, circle, label=center:$\times$] (2) [right of=1] {};
            \node[inner sep=0, draw, circle] (3) [right of=2] {$E_{i}$};
            \draw (0) -- (1);
            \draw (1) -- (2);
            \draw (2) -- (3);
            \end{tikzpicture}
            \end{figure}
            In particular, if $\Gamma_{i}$ is nonempty, then it contains a vertex of type $\otimes$.
            \item We have $\Gamma \cap \Gamma_{1} = \Gamma \cap \Gamma_{2} = \Gamma_{1} \cap \Gamma_{2} = \emptyset$.
            \item The union $\Gamma \cup \Gamma_{1} \cup \Gamma_{2}$ contains all vertices of type $\bigcirc$, $\otimes$ and $\odot$ except $E_{1}$ and $E_{2}$.
        \end{itemize}
        Indeed, the strict birational transformations of $D_{2}$, $E_{1}$ and $E_{2}$ on $\mathbb{P}^2$ are three lines intersecting each other. If they intersect at one point $Q$, then the coefficient at $Q$ is $3 - b_{1} > 2$, contradicting the condition that $(X,B)$ is a klt 0-pair. Hence they intersect in pairs. Other divisors on $X^{min}$ of type $\bigcirc$, $\otimes$ and $\odot$ are obtained by a sequence of blow-ups from $\mathbb{P}^2$. Moreover, the crepant pullback of $(X,B)$ on $X^{min}$ does not have negative coefficients. Hence, for each blown-up point $Q'$, the sum of the coefficients of the divisors passing through $Q'$ is $\geq 1$. Since the morphism $X^{min} \rightarrow \mathbb{P}^2$ is an isomorphism near the strict birational transformation of $D_{1}$, its exceptional divisors are centered at the pairwise intersections of the strict birational transformations of $D_{2}$, $E_{1}$ and $E_{2}$ on $\mathbb{P}^2$. Take $\Gamma$ (resp. $\Gamma_{i}$) to be the exceptional divisors centered at the intersection of the strict birational transformations of $E_{1}$ and $E_{2}$ (resp. $D_{2}$ and $E_{i}$) on $\mathbb{P}^2$. Then we have $\Gamma \cap \Gamma_{1} = \Gamma \cap \Gamma_{2} = \Gamma_{1} \cap \Gamma_{2} = \emptyset$, and $\Gamma \cup \Gamma_{1} \cup \Gamma_{2}$ contains all vertices of type $\bigcirc$, $\otimes$ and $\odot$ except $E_{1}$ and $E_{2}$. Moreover, by the above blow-up argument, $\Gamma$, $\Gamma_{1}$ and $\Gamma_{2}$ are chains. We can immediately see that $\Gamma$ has the given form by the definition of vertices of type $\bigcirc$, $\otimes$ and $\odot$. To see that $\Gamma_{i}$ has the given form, notice that in all cases of Table \ref{table: potential numerical cases for 2 singularities, I}, we have $b_{2} < \frac{1}{2}$. Hence, to have a nonnegative coefficient on $X^{min}$, one can only blow up at the intersection point on the strict birational transformation of $E_{i}$ for each intermediate model.
        \item There are exactly two vertices of type $\otimes$, and for every such vertex, there is an edge between this vertex and a vertex of type $\bigcirc$. Indeed, a vertex of type $\otimes$ corresponds to an exceptional divisor of $X' \rightarrow X_{d'} \cong \mathbb{P}^2$. In this case, $X'$ extracts $E_{1}$ and $E_{2}$ from $X$ and contracts two divisors to $X_{d'} \cong \mathbb{P}^2$. Let $Z$ be a prime divisor on $X'$ exceptional over $\mathbb{P}^{2}$, and $Z_{X}$ be its strict birational transformation on $X$. Then $Z \cdot D'_{1} = 0$ and $Z_{X} \cdot D_{1} >0$. Hence the strict birational transformation of $Z$ on $X^{min}$ intersects one of the two chains of exceptional divisors of $X^{min} \rightarrow X$ centered at $D_{1}$. In particular, since the vertex of a prime divisor $D_{i}$ with $2 \leq i \leq n$ is of type $\otimes$ if and only if $D'_{i} \cdot D'_{1}=0$, we have $n \leq 4$ by (1) and (2).
        \item At least one of $r_{2}$ and $r'_{2}$ is nonzero. Indeed, if $r_{2}=r'_{2}=0$, then the two chains $\Gamma_{1}$ and $\Gamma_{2}$ in (3) are not empty. By (3), each chain contains at least one vertex of type $\otimes$. Hence we have at least three vertices of type $\otimes$, contradicting (4). In particular, by Table \ref{table: potential numerical cases for 2 singularities, I}, we always have $b_{2} = \frac{1}{7}$.
    \end{enumerate}
    
    Now we use the above properties to complete the discussion for the $d'=1$ case. We discuss the problem on the basis of the possible values of $r_{2}$ and $r'_{2}$. By (5), they cannot both be zero.
    \begin{enumerate}
        \item $r_{2}=0$. In this case, we have $\Gamma_{1} \neq \emptyset$ and $\Gamma_{2} = \emptyset$. By (3) and (4), both chains $\Gamma$ and $\Gamma_{1}$ contain exactly one vertex of type $\otimes$. Suppose, without loss of generality, that the vertex of type $\otimes$ in $\Gamma_{1}$ is $D_{3}$ and the vertex of type $\otimes$ in $\Gamma$ is $D_{4}$. Then by (3), for $i=1,2$, the path from $D_{4}^{min}$ to $E_{i}^{min}$, apart from $D_{4}^{min}$ itself, consists of vertices of type $\bigcirc$, which correspond to curves contracted to $P_{i}$ by $X^{min}\rightarrow X$. Hence the dual graph has the following form:
        
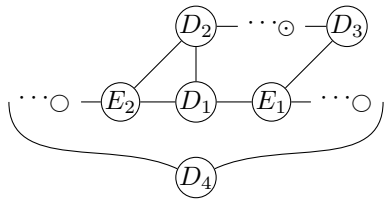
\begin{figure}[H]
        \centering
        \begin{tikzpicture}
            \node[inner sep=0, draw, circle] (0) {$D_{1}$};
            \node[inner sep=0, draw, circle] (1) [right of=0] {$E_{1}$};
            \draw (0) -- (1);
            \node[] (2) [right of=1] {$\cdots_{\bigcirc}$};
            \draw (1) -- (2);
            \node[inner sep=0, draw, circle] (a) [above of=0]{$D_{2}$};
            \draw (0) -- (a);
            \node[] (a+1) [right of=a]{$\cdots_{\odot}$};
            \draw (a) -- (a+1);
            \node[inner sep=0, draw, circle] (a+2) [right of=a+1]{$D_{3}$};
            \draw (a+1) -- (a+2);
            \draw (a+2) -- (1);
            \node[inner sep=0, draw, circle] (-1) [left of=0] {$E_{2}$};
            \draw (0) -- (-1);
            \node[] (-2) [left of=-1]{$\cdots_{\bigcirc}$};
            \draw (-1) -- (-2);
            \node[inner sep=0, draw, circle] (b) [below of=0]{$D_{4}$};
            \draw (b) to[out=25,in=-90] (2.east);
            \draw (b) to[out=155,in=-90] (-2.west);
            \draw (-1) -- (a);
        \end{tikzpicture}
        \end{figure}
        In particular, we have $r_{4},r'_{4} > 0$. Moreover, since the divisor $D_{3}$ can be obtained by a sequence of blow-ups from $\mathbb{P}^2$ at the intersection of the strict birational transformations of $D_{2}$ and $E_{1}$ on $\mathbb{P}^2$, we have $b_{3} < b_{2} = \frac{1}{7}$ by (5). Since every positive coefficient in $\Gamma(\mathcal{N}_{1})$ is at least $\frac{1}{7}$, we have $b_{3}=0$. Consequently,
        $$
        \sum\limits_{i=2}^{n}r_{i}b_{i}=r_{4}b_{4}, \qquad \sum\limits_{i=2}^{n}r'_{i}b_{i}-r'_{2}b_{2}=r'_{4}b_{4}.
        $$
        Thus both residual sums are positive integral multiples of $b_{4} \geq \frac{1}{7}$. A direct check of the relevant rows of Table \ref{table: potential numerical cases for 2 singularities, I}, for every allowed positive value of $r'_{2}$, shows that their greatest common positive rational divisor is at most $\frac{1}{20}$. This is a contradiction.
        
        \item $r'_{2}=0$. In this case, we have $\Gamma_{1} = \emptyset$ and $\Gamma_{2} \neq \emptyset$. Arguing as in (a), with $E_{1}$ and $E_{2}$ interchanged, we conclude that this case is also impossible.
        
        \item $r_{2},r'_{2}>0$. In this case, we have $\Gamma_{1} = \Gamma_{2} = \emptyset$. From Table \ref{table: potential numerical cases for 2 singularities, I}, only Cases 32 and 93 are possible, where we have $b_{2}=1/7$, $r_{2}=r'_{2}=1$, and $n=4$. After subtracting the contribution of $D_{2}$, we obtain
        $$
        \begin{array}{|c|c|c|}
        \hline
        &r_{3}b_{3}+r_{4}b_{4}
        &r'_{3}b_{3}+r'_{4}b_{4}\\ \hline
        \text{Case 32}&\frac{7}{12}&\frac{2}{5}\\\hline
        \text{Case 93}&\frac{1}{5}&\frac{7}{12} \\\hline
        \end{array}
        $$
        After a permutation of the labeling, one of the following holds:
        $$
        \begin{array}{|c|c|c|c|c|c|c|}
        \hline
        &b_{3} & b_{4} & r_{3} & r'_{3} & r_{4} & r'_{4} \\\hline
        \text{Case 32}&\frac{7}{12}&\frac{1}{5} & 1 & 0 & 0 & 2\\\hline
        \text{Case 32'}&\frac{7}{12}&\frac{2}{5} & 1 & 0 & 0 & 1\\\hline
        \text{Case 93}&\frac{1}{5}&\frac{7}{12} & 1 & 0 & 0 & 1 \\\hline
        \end{array}
        $$
        By (3), there is an edge from the vertex of the strict birational transformation of $D_{2}$ on $X^{min}$ to the vertex of the strict birational transformation of $E_{i}$ on $X^{min}$, $i=1,2$. Hence we have $q_{1} \leq r_{2} = 1$ and $q_{2} \leq r'_{2} = 1$. We conclude that $q_{1}=q_{2}=1$.
        
        In Case 32, we have
        $$
        a(1+\frac{q_{1}}{m_{1}}+\frac{q_{2}}{m_{2}}) = 64680(1+\frac{1}{2}+\frac{1}{77}) = 97860 < 127260.
        $$
        In Case 93, we have
        $$
        a(1+\frac{q_{1}}{m_{1}}+\frac{q_{2}}{m_{2}}) = 50400(1+\frac{1}{5}+\frac{1}{24}) = 62580 < 127260.
        $$
        Hence $b_{1} < 1 - \frac{1}{127260}$.
    \end{enumerate}
    This concludes the proof of Case 3.
    
    \item $2 \leq d' \leq 14$. In this case, by \thref{lemma: reduction}, we have $X_{d'} \cong \mathbb{P}(1,1,d')$. If $m_{1},m_{2} \geq 8$, then we will have $\frac{7}{8} + \frac{7}{8} + \frac{2}{7} \geq 2 + \frac{1}{28}$. Then by (\ref{equation: d <= d'+2}) and (\ref{equation 2}), we have
    $$
    b_{1} \leq 1 - \frac{1}{28d} < 1 - \frac{1}{28(d'+2)} \leq 1 - \frac{1}{448}.
    $$
    
    Hence, without loss of generality, we can assume that $m_{1} \leq 7$ and $m_{1} \leq m_{2}$. By Equation (\ref{equation 2}),
    $$
    (d' + \frac{q_{1}}{m_{1}} + \frac{q_{2}}{m_{2}})(1-b_{1}) = d' \sum\limits_{i=2}^{n}b_{i}^{\prime h}+\sum\limits_{i=2}^{n} b_{i}^{\prime v} + \frac{m_{1}-1+\sum\limits_{i=2}^{n} r_{i}b_{i}}{m_{1}} + \frac{m_{2}-1+\sum\limits_{i=2}^{n} r'_{i}b_{i}}{m_{2}} - 2 > 0.
    $$

    We disprove this case by contradiction. Suppose to the contrary that $b_{1} \geq 1 - \frac{1}{127260}$. Since 
    $$ 
    d = d' + \frac{q_{1}}{m_{1}} + \frac{q_{2}}{m_{2}} < d'+2,
    $$
    by (\ref{equation 2}), we have 
    \begin{equation}\label{inequality 2'}
    d' \sum\limits_{i=2}^{n}b_{i}^{\prime h}+\sum\limits_{i=2}^{n} b_{i}^{\prime v} + \frac{m_{1}-1+\sum\limits_{i=2}^{n} r_{i}b_{i}}{m_{1}} + \frac{m_{2}-1+\sum\limits_{i=2}^{n} r'_{i}b_{i}}{m_{2}} = 2 + (d' + \frac{q_{1}}{m_{1}} + \frac{q_{2}}{m_{2}})(1-b_{1}) < 2 + \frac{d'+2}{127260}.
    \end{equation}

    We use a computer to enumerate all the linear combinations satisfying inequality (\ref{inequality 2'}). The results are listed in Table \ref{table: potential numerical cases for 2 singularities, II} below. We give some explanations of Table \ref{table: potential numerical cases for 2 singularities, II}:
    \begin{itemize}
        \item Each row is distinguished by the following numerical data: $d'$, $m_{1}$, $m_{2}$, $b_{i}$, $\sum\limits_{i=2}^{n}r_{i}b_{i}$ and $\sum\limits_{i=2}^{n}r'_{i}b_{i}$. Notice that two cases with different $r_{i}$ and $r'_{i}$ but the same $\sum\limits_{i=2}^{n}r_{i}b_{i}$ and $\sum\limits_{i=2}^{n}r'_{i}b_{i}$ will be considered as the same case.
        \item If two cases are related by a permutation of the indices of $b_{i}, i=2,3,\cdots, n$, we will recognize them as the same case.
        \item To simplify the expression in the table, we replace the $b_{1}$ column by the column of $a = \frac{1}{(1-b_{1})(d' + \frac{q_{1}}{m_{1}}+\frac{q_{2}}{m_{2}})}$. Then in inequalities (\ref{inequality 2'}) we can write $2 + (1-b_{1})(d' + \frac{q_{1}}{m_{1}}+\frac{q_{2}}{m_{2}}) = 2 + \frac{1}{a}$. Notice that the condition $b_{1} \geq 1 - \frac{1}{127260}$ is equivalent to $a(d' + \frac{q_{1}}{m_{1}}+\frac{q_{2}}{m_{2}}) \geq 127260$.
        \item A number $b_{i}$ is displayed in Table \ref{table: potential numerical cases for 2 singularities, II} if and only if $b'_{i}>0$ (see (\ref{equation: b_{i}'})), that is, the strict birational transformation of $D_{i}$ on $X_{d'}^{min} \cong \mathbb{F}_{d'}$ is a prime divisor. For the numbers $b_{i}$ where $b'_{i} = 0$, we must have $r_{i} > 0$ or $r'_{i}>0$. In this case, $b_{i}$ is not displayed in the table but still contributes to the sum $\sum\limits_{i=2}^{n} r_{i}b_{i}$ or $\sum\limits_{i=2}^{n} r'_{i}b_{i}$.
        \item In the last column, we give the upper bounds of the numbers $r_{2}$ and $r'_{2}$ for each case. Notice that if $r_{i}=0$ (resp. $r'_{i}=0$), then $D_{i}$ doesn't pass through the singular point $P_{1}$ (resp. $P_{2}$) of $X$ on $D_{1}$.
    \end{itemize}

    \begin{center}
    \begin{longtable}{|@{\makebox[3em][r]{\rownumber\space}}|C|C|C|C|C|C|C|C|C|}
    \caption{Potential numerical cases when $D_{1}$ contains 2 singularities, II}\label{table: potential numerical cases for 2 singularities, II} \\
    \hline
    \multicolumn{1}{@{\makebox[3em][r]{Case No.~}}|C|}{d'} & m_{1} & m_{2} & a & b_{2} & \sum\limits_{i=2}^{n} r_{i}b_{i} & \sum\limits_{i=2}^{n} r'_{i}b_{i} & r_{2} \leq & r'_{2} \leq \\
    \hline
    
    2 & 2 & 3 & {32760} &
    {1/7} &
    {9/20} &
    {151/156} &
    0 &
    0\\\hline
    2 & 2 & 7 & {41160} &
    {1/7} &
    {9/20} &
    {136/147} &
    0 &
    0\\\hline
    2 & 2 & 13 & {32760} &
    {1/7} &
    {9/20} &
    {31/36} &
    0 &
    0\\\hline
    2 & 2 & 39 & {32760} &
    {1/7} &
    {9/20} &
    {7/12} &
    0 &
    0\\\hline
    2 & 2 & 49 & {41160} &
    {1/7} &
    {9/20} &
    {10/21} &
    0 &
    1\\\hline
    \end{longtable}
    \end{center}

    In particular, we always have $d'=2$, $m_{1}=2$, and $\sum\limits_{i=2}^{n} r_{i}b_{i} = \frac{9}{20}$. Now we examine the possible dual graphs on the minimal resolution $X^{min}$ of $X$. The dual graph satisfies the following properties:
    \begin{enumerate}[align=left, label=(\arabic*)]
        \item The strict birational transformation of $D_{2}$ on $X_{d'}^{min} \cong \mathbb{F}_{d'}$ is linearly equivalent to $E + 2F$, where $E$ and $F$ are the divisor classes of the negative section and the fibre of $\mathbb{F}_{d'}$, respectively. Moreover, for $3 \leq i \leq n$, the strict birational transformation of $D_{i}$ on $X_{d'}^{min} \cong \mathbb{F}_{d'}$ is zero. Indeed, as explained before Table \ref{table: potential numerical cases for 2 singularities, II}, for $2 \leq i \leq n$, a number $b_{i}$ is displayed if and only if $b'_{i}>0$. By Table \ref{table: potential numerical cases for 2 singularities, II}, there is only one nonzero $b'_{i}$ for $2 \leq i \leq n$, namely $b'_{2}$. Hence, by \thref{lemma: reduction}, the strict birational transformation of $D_{2}$ on $X_{d'}^{min} \cong \mathbb{F}_{d'}$ is linearly equivalent to $E + 2F$, and the strict birational transformation of $D_{i}$ on $X_{d'}^{min} \cong \mathbb{F}_{d'}$ is zero for $3 \leq i \leq n$.
        \item There is an edge between the vertex of the divisor $D_{2}$ and the vertex of the divisor $E_{1}$. In particular, we have $r_{2}>0$. Indeed, the strict birational transformations of $D_{2}$ and $E_{1}$ on $\mathbb{P}(1,1,2)$ intersect transversally at a smooth point $Q$. Suppose, to the contrary, that their strict birational transformations on $X^{min}$ do not intersect. Then we need to blow up at $Q$ at least once to obtain $X^{min}$. Since $\varphi_{min}$ is an isomorphism over $D_{1}^{min}$,  $Q$ does not lie on the strict birational transformation of $D_{1}$ on $\mathbb{P}(1,1,2)$. Moreover, the strict birational transformations of $E_{1}$ and $E_{2}$ only meet at the cone vertex. By (1), the boundary near $Q$ therefore consists precisely of the strict birational transformations of $D_{2}$ and $E_{1}$ with their respective coefficients. However, by \thref{statement: local classification}, the coefficient of the first blown-up divisor at $Q$ in the crepant pullback of the pair $(X,B)$ on $X^{min}$ is
        $$
        b_{2} + \frac{m_{1}-1+\sum\limits_{i=2}^{n}r_{i}b_{i} + b_{1}-1}{m_{1}} - 1 = \frac{1}{7} + \frac{1 + \frac{9}{20} + b_{1}-1}{2} - 1 = -\frac{37}{280} - \frac{1-b_{1}}{2}<0.
        $$
        This contradicts the effectivity of the crepant boundary on the minimal resolution $X^{min}$, as explained in the $d'=1$ case.
    \end{enumerate}
    However, in all cases in Table \ref{table: potential numerical cases for 2 singularities, II}, we have $r_{2}=0$, contradicting (2). This completes the proof of \thref{lemma: D1 contains 2 singularities}.
    \end{enumerate}
\end{proof}

\paragraph{$D_{1}$ contains 3 singularities of $X$}
\hfill

    In this subsection, we prove the following:
\begin{lemma}\label{lemma: D1 contains 3 singularities}
    Under the assumptions of \thref{thm: act gap}, suppose additionally that $D_{1}$ is a nonsingular rational curve containing exactly three singularities of $X$, and that we are in one of the cases remaining after the reduction in \thref{lemma: reduction}. Then we have $b_{1}\leq 1-\frac{1}{2520}$.
\end{lemma}

\begin{proof}
    Let $m_{1},m_{2},m_{3},q_{1},q_{2},q_{3}$ be as in \thref{construction: extract to make Cartier}. By \thref{lemma: at most 4 singularities on D1}, we have
    $$
    l + \sum\limits_{i=1}^{3-l}\frac{m_{i}-1}{m_{i}} \leq 2.
    $$
    If $(X,D_{1})$ is non-plt, then we must have $l=1$ and $m_{1}=m_{2}=2$. By \thref{statement: bound by degree}, we have
    $$
    (1-b_{1})D_{1}^{2} = \sum\limits_{i=2}^{n}w_{i}b_{i} + \frac{1}{2}\sum\limits_{i=1}^{2}\sum\limits_{j=2}^{n}r_{ij}b_{j} \geq \frac{1}{14}.
    $$
    Since we are in one of the cases retained by \thref{lemma: reduction}, we have $d' \leq 14$. By (\ref{equation: d <= d'+2}), we have
    $$
    D_{1}^{2}\leq d'+2\leq16.
    $$
    Hence
    $$
    b_{1} \leq 1 - \frac{1}{14D_{1}^{2}} \leq 1 - \frac{1}{224}.
    $$
    
    Next we assume that $(X,D_{1})$ is plt. Since $\sum\limits_{i=1}^{3} \frac{m_{i}-1}{m_{i}} \leq 2$, the possible values for the 3-tuple $(m_{1},m_{2},m_{3})$ are 
    $$
    (2,2,\lambda)
    $$
    for some positive integer $\lambda \geq 2$, and
    $$
    (2,3,3),\, (2,3,4), \, (2,3,5), \, (2,3,6), \, (2,4,4), \, (3,3,3).
    $$ 
    Moreover, we have $d-d'= \frac{q_{1}}{m_{1}} + \frac{q_{2}}{m_{2}} + \frac{q_{3}}{m_{3}} \leq 2$.
    \begin{enumerate}[align=left,label = \textbf{Case \arabic*}:]

    \item $d' \leq 1$. In this case $d \leq 3$. Suppose first that $(m_{1},m_{2},m_{3})=(2,2,\lambda)$. We claim that there exists a prime divisor $D_{i}, i \neq 1$ meeting $D_{1}$ either in the smooth locus of $X$ or at $P_{1}$ or $P_{2}$. Indeed, otherwise $w_{i}=r_{1i}=r_{2i}=0$ for all $i \neq 1$. Since $(X,D_{1}+\sum\limits_{i=2}^{n}b_{i}D_{i})$ is lc, we have $\sum\limits_{j=2}^{n}r_{3j}b_{j}\leq1$. By \thref{statement: bound by degree}, we have
    $$
    0 < d(1-b_{1}) = -2 + \frac{1}{2} + \frac{1}{2} + \frac{\lambda - 1 + \sum\limits_{j=2}^{n}r_{3j}b_{j}}{\lambda} \leq 0,
    $$
    a contradiction. Such a component $D_{i}$ contributes at least $\frac{1}{14}$ at a point of index $2$, and at least $\frac{1}{7}$ at a smooth point. If $\lambda \geq 15$, then by \thref{statement: bound by degree},
    $$
    d(1-b_{1}) = -2 + \sum\limits_{i=1}^{3} \frac{m_{i}-1+\sum\limits_{j=2}^{n} r_{ij}b_{j}}{m_{i}} + \sum\limits_{i=2}^{n} w_{i}b_{i} \geq -2 + \frac{1}{2} + \frac{1}{2} + \frac{\lambda-1}{\lambda} + \frac{1}{14} \geq \frac{1}{210}.
    $$
    In this case $b_{1} \leq 1 - \frac{1}{630}$. Hence we can assume that $\lambda \leq 14$. 
    
    Now we use a computer to enumerate all the possible cases of the triple $(m_{1},m_{2},m_{3})$. The results are listed in the following table:

\begin{footnotesize}
\begin{center}
\begin{longtable}{|@{\makebox[3em][r]{\rownumber\space}}|C|C|C|C|C|C|}
\caption{Potential numerical cases when $D_{1}$ contains 3 singularities}\label{table: potential numerical cases for 3 singularities} \\
\hline
\multicolumn{1}{@{\makebox[3em][r]{Case No.~}}|C|}{(m_{1},m_{2},m_{3})} & \sum\limits_{i=1}^{3} \frac{m_{i}-1+\sum\limits_{j=2}^{n} r_{ij}b_{j}}{m_{i}} + \sum\limits_{i=2}^{n} w_{i}b_{i} - 2 & b_{2} & \sum\limits_{j=2}^{n} r_{1j}b_{j} & \sum\limits_{j=2}^{n} r_{2j}b_{j} & \sum\limits_{j=2}^{n} r_{3j}b_{j}\\
\hline

(2,2,14)& {1/98} & 0 &
0& {1/7}& {1/7} \\ 
\hline

(2,2,13)& {1/182} & 0 &
0& {1/7}& {1/7} \\ 
\hline

(2,2,12)& {1/210} & 0 &
0& {1/7}& {1/5} \\ 
\hline

(2,2,11) & {1/308} & 0 & 0 & {1/7} & {1/4} \\
\hline

(2,2,10) & {1/210} & 0 & 0 & {1/7} & {1/3} \\
\hline

(2,2,9) & {1/252} & 0 & 0 & {1/7} & {11/28} \\
\hline

(2,2,8) & {3/1120} & 0 & 0 & {1/7} & {9/20} \\
\hline

(2,2,7) & {1/392} & 0 & 0 & {1/4} & {1/7} \\
\hline

(2,2,6) & {1/504} & 0 & 0 & {1/7} & {7/12} \\
\hline

(2,2,5) & {1/700} & 0 & 0 & {1/7} & {13/20} \\
\hline

(2,2,4) & {1/420} & 0 & 0 & {1/3} & {12/35} \\
\hline

(2,2,3) & {1/840} & 0 & 0 & {1/4} & {22/35} \\
\hline

(2,2,2) & {1/504} & 0 & 0 & {1/7} & {31/36} \\
\hline

(2,3,3) & {1/210} & 0 & {12/35} & 0 & 0 \\
\hline

(2,3,4) & {1/420} & 0 & 0 & 0 & {12/35} \\
\hline

(2,3,5) & {1/150} & 0 & 0 & 0 & {1/5} \\
\hline

(2,3,6) & {1/42} & 0 & 0 & 0 & {1/7} \\
\hline

(2,4,4) & {1/28} & 0 & 0 & 0 & {1/7} \\
\hline

(3,3,3) & {1/21} & 0 & 0 & 0 & {1/7} \\
\hline

\end{longtable}
\end{center}
\end{footnotesize}

    For every possible case in Table \ref{table: potential numerical cases for 3 singularities}, we have
    $$
    d(1-b_{1}) = -2 + \sum\limits_{i=1}^{3} \frac{m_{i}-1+\sum\limits_{j=2}^{n} r_{ij}b_{j}}{m_{i}} + \sum\limits_{i=2}^{n} w_{i}b_{i} \geq \frac{1}{840}.
    $$
    
    Hence we have $b_{1} \leq 1 - \frac{1}{840 \times 3} = 1-\frac{1}{2520}$.

    \item $2 \leq d' \leq 14$. If $(m_{1},m_{2},m_{3}) \neq (2,2,2)$ and $(m_{1},m_{2},m_{3}) \neq (2,2,3)$, then we have
    $$
    (1-b_{1})D_{1}^2 = d' \sum\limits_{i=2}^{n}b_{i}^{h}+\sum\limits_{i=1}^{n} b_{i}^{v} + \sum\limits_{i=1}^{3} \frac{m_{i}-1+\sum\limits_{j=2}^{n} r_{ij}b_{j}}{m_{i}} - 2 \geq \frac{2}{7} + \frac{1}{2} + \frac{1}{2} + \frac{3}{4} - 2 = \frac{1}{28}.
    $$
    
    For the triples $(2,2,2)$ and $(2,2,3)$, a direct computation shows that
    $$
    (1-b_{1})D_{1}^2 = d' \sum\limits_{i=2}^{n}b_{i}^{h}+\sum\limits_{i=2}^{n} b_{i}^{v} + \sum\limits_{i=1}^{3} \frac{m_{i}-1+\sum\limits_{j=2}^{n} r_{ij}b_{j}}{m_{i}} - 2 \geq \frac{3}{280}.
    $$

    Since $d' \leq 14$ we have $D_{1}^2 < 17$. Hence $b_{1} < 1 - \frac{3}{280 \times 17} = \frac{4757}{4760}$.
    \end{enumerate}
\end{proof}

\paragraph{$D_{1}$ contains 4 singularities of $X$}
\hfill

    In this subsection, we prove the following:
\begin{lemma}\label{lemma: D1 contains 4 singularities}
    Under the assumptions of \thref{thm: act gap}, suppose additionally that $D_{1}$ is a nonsingular rational curve containing exactly four singularities of $X$, and that we are in one of the cases remaining after the reduction in \thref{lemma: reduction}. Then we have $b_{1}\leq 1-\frac{1}{252}$.
\end{lemma}
    
\begin{proof}
    Let $m_{i},q_{i}$ be as in \thref{construction: extract to make Cartier}. Since $\sum\limits_{i=1}^{4}\frac{m_{i}-1}{m_{i}} \leq 2$, we must have $m_{i}=2$ for all $i$ and $d = d' + 2$. The equation in \thref{statement: equation from adjunction} becomes
    $$
    d(1-b_{1}) = -2 + \sum\limits_{i=2}^{n} w_{i}b_{i} + \frac{1 + \sum\limits_{i=2}^{n} r_{i}b_{i}}{2} + \frac{1 + \sum\limits_{i=2}^{n} r'_{i}b_{i}}{2} + \frac{1 + \sum\limits_{i=2}^{n} r''_{i}b_{i}}{2} + \frac{1 + \sum\limits_{i=2}^{n} r'''_{i}b_{i}}{2} \geq \frac{1}{14}.
    $$
    Since $d' \leq 14$, we have $d \leq 16$ and $b_{1} \leq 1 - \frac{1}{224}$. This contradicts our assumption in \thref{statement: equation from adjunction} that $b_{1} > 1 - \frac{1}{252}$. Hence $b_{1} \leq 1 - \frac{1}{252}$.
\end{proof}

    Now we conclude the above discussion.

\begin{proof}[Proof of \thref{lemma: D1 contains singularities}]
    By \thref{lemma: reduction}, it suffices to consider the cases listed there. By \thref{lemma: non-plt case}, the non-plt case satisfies a strictly stronger bound, so we may assume that $(X,D_{1})$ is plt. By \thref{lemma: at most 4 singularities on D1}, the curve $D_{1}$ contains between one and four singularities of $X$. The assertion now follows from \thref{lemma: D1 contains 1 singularity}, \thref{lemma: D1 contains 2 singularities}, \thref{lemma: D1 contains 3 singularities}, and \thref{lemma: D1 contains 4 singularities}. The equality case is preserved by the reductions in \thref{lemma: reduction} and \thref{lemma: non-plt case}. The cases with at least two singularities give strict inequality, so equality can occur only in the one-singularity case. By \thref{lemma: D1 contains 1 singularity}, this is precisely \thref{example: optimal bound for act 2}.
\end{proof}

\paragraph{Nodal case}
\hfill

    In this subsection, we prove the following:
    \begin{lemma}\label{lemma: D1 nodal}
        Under the assumptions of \thref{thm: act gap}, suppose additionally that $D_{1}$ is a rational curve with 1 node. Then we have $b_{1} \leq 1 - \frac{1}{252}$.
    \end{lemma}

\begin{proof}
    Let $P \in D_{1}$ be the node. By \thref{prop: surface lc singularity}, $X$ is either nonsingular at $P$ or has a singularity of type $\mathbf{A}_{r}$ at $P$ for some positive integer $r$. Since $(X,D_{1})$ has an $\mathbb{R}$-complement, we have
    $$
    0 \geq (K_{X}+D_{1}) \cdot D_{1} = -2 + 2 + \sum\limits_{i=1}^{u} \frac{m_{i}-1}{m_{i}} \geq 0.
    $$
    Hence $\mathrm{Diff}_{\widetilde{D_{1}}}(0) = P_{1}+P_{2}$ where $P_{1},P_{2}$ maps to $P$, and $X$ is nonsingular near $D_{1}$ except for $P$. Moreover, for any other divisor $D_{i}, i\neq 1$, $D_{i}$ doesn't pass through $P$. Hence we can write the equation in \thref{statement: equation from adjunction} as
    \begin{equation}\label{equation nodal}
    (1-b_{1})d = -2 + 2 + \sum\limits_{i=2}^{n} w_{i}b_{i} \geq \frac{1}{7}.
    \end{equation}

    If $X$ is nonsingular at $P$, then $D_{1}$ is Cartier. Then, similar to the proof of \thref{prop: D1 elliptic}, we have the short exact sequence
    $$
    0 \rightarrow \mathcal{O}_{X}(K_{X}) \rightarrow \mathcal{O}_{X}(K_{X} + D_{1}) \rightarrow \mathcal{O}_{D_{1}} \rightarrow 0.
    $$
    By \thref{thm: generalized Kodaira vanishing} with $N \equiv -K_{X} + 0$, we have $h^{1}(X,\mathcal{O}_{X}(K_{X})) = 0$. Hence $h^{0}(X,\mathcal{O}_{X}(K_{X} + D_{1})) \neq 0$. We conclude that $K_{X} + D_{1} \sim 0$. In particular, $K_{X}$ is Cartier and $X$ is canonical. Let $\widetilde{X}$ be the crepant terminalization of $X$. Since $X$ is of Fano type, $\widetilde{X}$ is a smooth rational surface. Hence we have
    $$
    d = D_{1}^{2} = K_{X}^{2} = K_{\widetilde{X}}^{2} = 10 - \rho(\widetilde{X}) \leq 9.
    $$
    Hence, in this case, we have $b_{1} \leq 1 - \frac{1}{7d} \leq 1 - \frac{1}{63}$.
    
    If $X$ is singular at $P$, then the dual graph $\Gamma$ is Figure \ref{figure: dual graph 1}, where $C = D_{1}^{min}$. Similar to \thref{construction: extract to make Cartier}, we produce an intermediate model $X^{min} \rightarrow X' \rightarrow X$ by contracting the exceptional divisors that don't intersect $D^{min}_{1}$. Then $D'_{1}$ is a smooth Cartier divisor on $X'$. If $r=1$, then $X'$ extracts the unique exceptional divisor $E_{1}$ centered at $P$. If $r \geq 2$, then $X'$ extracts the two exceptional divisors $E_{1}$ and $E_{2}$ which are at the end of $\Gamma$. By \thref{prop:formula for ld} we have
    $$
    a(E_{i};X,B) = \frac{1+\Delta(\Gamma - v_{i})}{\Delta(\Gamma)}(1-b_{1}) \leq (1-b_{1})
    $$
    where the inequality $1+\Delta(\Gamma - v_{i}) \leq \Delta(\Gamma)$ comes from the dual graph of type $\mathbf{A}_{r}$. Put
    $$
    e_{i}:=1-a(E_{i};X,B).
    $$
    This is the coefficient of $E_{i}$ in the crepant transformation of $(X,B)$ on $X'$. The preceding inequality and the klt condition give
    $$
    b_{1}\leq e_{i}<1.
    $$
    The images of $E_{i}$ on the subsequent birational targets retain these coefficients. Write $d' = D_{1}^{\prime 2}$. By the projection formula, we have
    $$
    d-d'=\sum\limits_{i}\frac{1+\Delta(\Gamma - v_{i})}{\Delta(\Gamma)}(D'_{1}\cdot E_{i}) \leq 2.
    $$  
    If $d' \leq 0$, then $d \leq 2$. By (\ref{equation nodal}), we have
    $$
    b_{1} \leq 1 - \frac{1}{7d} \leq 1 - \frac{1}{14}.
    $$
    Hence we can assume that $d' \geq 1$. By \thref{cor: complete linear system morphism}, we can use the linear system $|D'_{1}|$ to obtain a projective morphism from $X'$ to a surface of degree $d'$ in $\mathbb{P}^{d'+1}$. 
    
    If $r \geq 2$, then $X'$ has two divisors exceptional over $X$. By \thref{thm: surface of minimal degree}, the surface $X_{d'}$ is one of the following:
    \begin{enumerate}
        \item A rational normal scroll $\mathbb{F}_{d'-2s}$ for some integer $1 \leq s \leq \lfloor \frac{d'}{2} \rfloor$. In this case $D'_{1}$ is the pullback of a divisor linearly equivalent to $E + (d'-s)F$, where $E$ is the negative section of $\mathbb{F}_{d'-2s}$ (or a zero section if $d' = 2s$) and $F$ is a general fibre. Since $D'_{1} \cdot E_{1} = D'_{1} \cdot E_{2} =1$ and $E_{1} \cdot E_{2} > 0$, the strict birational transformation of $E_{1}$ and $E_{2}$ are a fibre and the negative section (or a zero section if $d' = 2s$), and $s = 1$. Without loss of generality, we assume that $E_{1}$ becomes the negative section. Taking the adjunction of the crepant transformation of $(X, B)$ on $X_{d'} \cong \mathbb{F}_{d'-2s}$ for a general fibre of $\mathbb{F}_{d'-2s}/\mathbb{P}^1$, we have
        $$
        2 = b_{1} + e_{1} + \sum\limits_{i=2}^{n}w_{i}b_{i}.
        $$
        Since $b_{1}\leq e_{1} < 1$, we have $\sum\limits_{i=2}^{n}w_{i}b_{i} > 0$.
        Hence
        $$
        2 = b_{1} + e_{1} + \sum\limits_{i=2}^{n}w_{i}b_{i} \geq 2b_{1} + \frac{1}{7}.
        $$
        We obtain $b_{1} \leq \frac{13}{14}$.
        \item A Veronese embedding of $\mathbb{P}^2$ in $\mathbb{P}^5$. In this case, $D'_{1}$ is the pullback of a conic in $\mathbb{P}^2$. However, we have $D'_{1} \cdot E_{1} = D'_{1} \cdot E_{2} =1$, a contradiction. Hence this case is not possible.
        \item The projective plane $\mathbb{P}^2$, in which case $d' = 1$. In this case, we have
        $$
        3 = b_{1} + e_{1} + e_{2} + \sum\limits_{i=2}^{n}w_{i}b_{i}.
        $$
        Since $b_{1} \leq e_{i} < 1$ for $i=1,2$, we have $\sum\limits_{i=2}^{n}w_{i}b_{i} > 0$.
        Hence
        $$
        3 = b_{1} + e_{1} + e_{2} + \sum\limits_{i=2}^{n}w_{i}b_{i} \geq 3b_{1} + \sum\limits_{i=2}^{n}w_{i}b_{i} \geq 3b_{1} + \frac{1}{7}.
        $$
        We obtain $b_{1} \leq \frac{20}{21}$.
        \item The rational cone $\mathbb{P}(1,1,d')$ where $d' \geq 2$. Let $\mathbb{F}_{d'}$ be the minimal resolution of $X_{d'}$ and $E$ be the negative section of $\mathbb{F}_{d'}$. Since $d' \geq 2$, the strict birational transformation of $D_{1}$ on $\mathbb{F}_{d'}$ is linearly equivalent to $E + d' F$. Moreover, for every divisor $E_{i}$ on $X'$ exceptional over $X$, the strict birational transformation of $E_{i}$ on $\mathbb{F}_{d'}$ is a fibre. Let $e$ be the coefficient of $E$ in the crepant transformation of $(X,B)$ on $\mathbb{F}_{d'}$. Then $e<1$ and $b_{1} \leq e_{i} < 1$ for $i=1,2$. As in the proof of \thref{lemma: reduction}, we can split the classes of the additional surviving boundary components into classes of type $E+d'F$ and $F$, where $F$ is the fibre class. Each resulting term retains the coefficient of its original component. Hence we can assume that for $2 \leq i \leq n$, the strict birational transformation of $D_{i}$ on $X^{min}_{d'}$ is either 0, or linearly equivalent to $E+d'F$ or $F$. Similar to the proof of \thref{lemma: reduction}, for $2 \leq i \leq n$, we put
        $$
        b_{i}^{\prime h} = \begin{cases}
        b_{i} & \text{if the strict birational transformation of }D_{i} \text{ on } X^{min}_{d'} \cong \mathbb{F}_{d'} \text{ is a horizontal prime divisor}, \\
        0 & \text{otherwise}.
        \end{cases}
        $$
        and
        $$
        b_{i}^{\prime v} = \begin{cases}
        b_{i} & \text{if the strict birational transformation of }D_{i} \text{ on } X^{min}_{d'} \cong \mathbb{F}_{d'} \text{ is a vertical prime divisor}, \\
        0 & \text{otherwise}.
        \end{cases}
        $$
        Intersecting the crepant $0$-pair with the negative section, we obtain
        $$
        e = 1-\frac{2}{d'} +\frac{1}{d'} \left(e_{1}+e_{2}+\sum\limits_{i=2}^{n}b_{i}^{\prime v}\right) \geq 1-\frac{2(1-b_{1})}{d'}.
        $$
        On the other hand, adjunction on a general fibre gives
        $$
        2=e+b_{1}+\sum\limits_{i=2}^{n}b_{i}^{\prime h}.
        $$
        Since $e<1$ and $b_{1}<1$, the horizontal sum is positive. Hence, by \thref{exam: gap numbers}, we have
        $$
        2 =e+b_{1}+\sum\limits_{i=2}^{n}b_{i}^{\prime h} \geq e+b_{1}+\frac{1}{7} \geq 1-\frac{2(1-b_{1})}{d'}+b_{1}+\frac{1}{7}.
        $$
        It follows that
        $$
        (1-b_{1})\left(1+\frac{2}{d'}\right)\geq\frac{1}{7}.
        $$
        Since $d'\geq2$, we conclude that
        $$
        b_{1} \leq 1-\frac{d'}{7(d'+2)} \leq \frac{13}{14}.
        $$
    \end{enumerate}

    If $r=1$, then $X'$ has only one divisor $E_{1}$ exceptional over $X$, and $D'_{1} \cdot E_{1} = 2$. By \thref{thm: surface of minimal degree}, the surface $X_{d'}$ is one of the following:
    \begin{enumerate}
        \item A rational normal scroll $\mathbb{F}_{d'-2s}$ for some integer $1 \leq s \leq \lfloor \frac{d'}{2} \rfloor$. In this case $D'_{1}$ is the pullback of a divisor linearly equivalent to $E + (d'-s)F$, where $E$ is the negative section of $\mathbb{F}_{d'-2s}$ (or a zero section if $d' = 2s$) and $F$ is a general fibre. Since $D'_{1} \cdot E_{1} = 2$, the strict birational transformation of $E_{1}$ becomes a section of $\mathbb{F}_{d'-2s}$. Taking the adjunction on a general fibre, we have
        $$
        2 = b_{1} + e_{1} + \sum\limits_{i=2}^{n}w_{i}b_{i}.
        $$
        Since $b_{1} \leq e_{1} < 1$, we have $\sum\limits_{i=2}^{n}w_{i}b_{i} > 0$.
        Hence
        $$
        2 = b_{1} + e_{1} + \sum\limits_{i=2}^{n}w_{i}b_{i} \geq 2b_{1} + \sum\limits_{i=2}^{n}w_{i}b_{i} \geq 2b_{1} + \frac{1}{7}.
        $$
        We obtain $b_{1} \leq \frac{13}{14}$.
        \item A Veronese embedding of $\mathbb{P}^2$ in $\mathbb{P}^5$. In this case, $D'_{1}$ is the pullback of a conic in $\mathbb{P}^2$. Since $D'_{1} \cdot E_{1} = 2$, the strict birational transformation of $E_{1}$ is a line in $\mathbb{P}^2$. We have
        $$
        3 = 2b_{1} + e_{1} + \sum\limits_{i=2}^{n}w_{i}b_{i}.
        $$
        Since $b_{1} \leq e_{1} < 1$, we have $\sum\limits_{i=2}^{n}w_{i}b_{i} > 0$.
        Hence
        $$
        3 = 2b_{1} + e_{1} + \sum\limits_{i=2}^{n}w_{i}b_{i} \geq 3b_{1} + \sum\limits_{i=2}^{n}w_{i}b_{i} \geq 3b_{1} + \frac{1}{7}.
        $$
        We obtain $b_{1} \leq \frac{20}{21}$.
        \item The projective plane $\mathbb{P}^2$, in which case $d' = 1$. Since $D'_{1} \cdot E_{1} = 2$, the strict birational transformation of $E_{1}$ is a conic in $\mathbb{P}^2$. In this case, we have
        $$
        3 = b_{1} + 2e_{1} + \sum\limits_{i=2}^{n}w_{i}b_{i}.
        $$
        Since $b_{1} \leq e_{1} < 1$, we have $\sum\limits_{i=2}^{n}w_{i}b_{i} > 0$.
        Hence
        $$
        3 = b_{1} + 2e_{1} + \sum\limits_{i=2}^{n}w_{i}b_{i} \geq 3b_{1} + \sum\limits_{i=2}^{n}w_{i}b_{i} \geq 3b_{1} + \frac{1}{7}.
        $$
        We obtain $b_{1} \leq \frac{20}{21}$.
        \item The rational cone $\mathbb{P}(1,1,d')$ where $d' \geq 2$. Let $\mathbb{F}_{d'}$ be the minimal resolution of $X_{d'}$ and $E$ be the negative section of $\mathbb{F}_{d'}$. Since $d' \geq 2$, the strict birational transformation of $D_{1}$ on $\mathbb{F}_{d'}$ is linearly equivalent to $E + d' F$. Since $D'_{1} \cdot E_{1} = 2$, we have $d' = 2$, and the strict birational transformation of $E_{1}$ on $\mathbb{F}_{d'}$ is linearly equivalent to $E + 2F$. Let $e$ be the coefficient of $E$ in the crepant transformation of $(X,B)$ on $\mathbb{F}_{d'}$. Then $e<1$ and $b_{1}\leq e_{1}<1$. As in the proof of \thref{lemma: reduction}, we can split the classes of the additional surviving boundary components into classes of type $E+d'F$ and $F$, where $F$ is the fibre class. Each resulting term retains the coefficient of its original component. Hence we can assume that for $2 \leq i \leq n$, the strict birational transformation of $D_{i}$ on $X^{min}_{d'}$ is either 0, or linearly equivalent to $E+d'F$ or $F$. Similar to the proof of \thref{lemma: reduction}, for $2 \leq i \leq n$, we put
        $$
        b_{i}^{\prime h} = \begin{cases}
        b_{i} & \text{if the strict birational transformation of }D_{i} \text{ on } X^{min}_{d'} \cong \mathbb{F}_{d'} \text{ is a horizontal prime divisor}, \\
        0 & \text{otherwise}.
        \end{cases}
        $$
        and
        $$
        b_{i}^{\prime v} = \begin{cases}
        b_{i} & \text{if the strict birational transformation of }D_{i} \text{ on } X^{min}_{d'} \cong \mathbb{F}_{d'} \text{ is a vertical prime divisor}, \\
        0 & \text{otherwise}.
        \end{cases}
        $$
        Intersecting the crepant $0$-pair with the negative section, we obtain
        $$
        e=\frac{1}{2}\sum\limits_{i=2}^{n}b_{i}^{\prime v}.
        $$
        Adjunction on a general fibre therefore gives
        $$
        2 =b_{1}+e_{1} +\frac{1}{2}\sum\limits_{i=2}^{n}b_{i}^{\prime v} +\sum\limits_{i=2}^{n}b_{i}^{\prime h}.
        $$
        Since $b_{1}<1$ and $e_{1}<1$, the sum
        $$
        \frac{1}{2}\sum\limits_{i=2}^{n}b_{i}^{\prime v} +\sum\limits_{i=2}^{n}b_{i}^{\prime h}
        $$
        is positive. Every positive coefficient is at least $\frac{1}{7}$, so this sum is at least $\frac{1}{14}$. Using $e_{1}\geq b_{1}$, we obtain
        $$
        2\geq b_{1}+e_{1}+\frac{1}{14} \geq 2b_{1}+\frac{1}{14}.
        $$
        Hence
        $$
        b_{1}\leq\frac{27}{28}.
        $$
    \end{enumerate}
    Summing up the above cases, we always have $b_{1} \leq 1 - \frac{1}{28}$. This contradicts our assumption in \thref{statement: equation from adjunction} that $b_{1} > 1 - \frac{1}{252}$. Hence $b_{1} \leq 1 - \frac{1}{252}$.
\end{proof}

    Now we conclude the above discussion.

\begin{proof}[Proof of \thref{prop: D1 rational}]
    The proposition follows from \thref{lemma: no singular point}, \thref{lemma: D1 contains singularities}, and \thref{lemma: D1 nodal}.
\end{proof}

\begin{proof}[Proof of \thref{thm: act gap}]
    The theorem follows from \thref{prop: D1 elliptic} and \thref{prop: D1 rational}.
\end{proof}

\subsection{Gap number of \texorpdfstring{$\mathbb{R}$}{R}-complementary threshold}

\begin{proof}[Proof of \thref{thm: Rct gap}]
    Write $F = \sum\limits_{i=1}^{r}m_{i}F_{i}$, where $r$ is a positive integer and $m_{i} \in \mathbb{Z}_{>0}$ for $i=1,\cdots,r$. If $m_{i} \geq 2$ for some $i$, then we have $t \leq \frac{1}{m_{i}} \leq \frac{1}{2}$. Hence we can assume that $F= \sum\limits_{i=1}^{r}F_{i}$. We prove the theorem by induction on the number of prime components of $F$. We first prove the theorem when $F$ is a prime Weil divisor. By \thref{prop: reduce complementary threshold}, we only need to find the following numbers:
    \begin{itemize}
        \item The log canonical threshold $\mathrm{lct}(X' \ni o', B';F')$ for a $\mathbb{Q}$-factorial closed point germ $o' \in X'$ of Fano type of dimension $\leq 2$, and effective divisors $B'$ and $F'$ on $X'$ whose multiplicities are respectively the multiplicities of $B$ and $F$. If $\mathrm{dim}X'=0$, then there is nothing to prove. If $\mathrm{dim}X'=1$, then $o' \in X'$ is a smooth curve germ, and we have $t = 1 - b_{i}$ for some $b_{i} \in \Gamma(\mathcal{N}_{1})$. In particular, if $t < 1$ then $t \leq 1 - \frac{1}{7} = \frac{6}{7}$. If $\mathrm{dim}X' = 2$, then the result follows from \thref{thm:lct gap} and \thref{corollary: lct gap special coefficients}.
        \item The anticanonical threshold $\mathrm{act}(X''/o'',B'';F'')$ for a projective contraction $X'' \rightarrow o''$, where $o''$ is a closed point, and effective divisors $B''$ and $F''$ on $X''$ whose multiplicities are respectively the multiplicities of $B$ and $F$, such that $X''$ is $\mathbb{Q}$-factorial of Fano type of dimension $\leq 2$ and $\rho(X'')=1$. In particular, $F''$ is an effective Weil divisor. If $\mathrm{dim}X'' = 1$, then $X'' \cong \mathbb{P}^1$ and we have 
        $$
        t = \frac{2 - \sum\limits_{i=1}^{u}b_{i}}{m}
        $$
        for some $b_{i} \in \Gamma(\mathcal{N}_{1}), i=1,2,\cdots,u$ and positive integer $m = \mathrm{deg}F''$. In particular, if $t < 1$, then 
        $$
        t \leq \max\{ 2 - \alpha_{1}(\Gamma(\mathcal{N}_{1})), 1 - \frac{\alpha_{0}(\Gamma(\mathcal{N}_{1}))}{2}, \frac{2}{3} \} = \max\{ \frac{251}{252}, \frac{13}{14}, \frac{2}{3} \} = \frac{251}{252}.
        $$
        
        If $\mathrm{dim}X'' = 2$, then any lc center of the pair $(X'', B'' + tF'')$ is not contained in $\mathrm{Supp}F''$. If $(X'', B'' + tF'')$ is klt, then the result follows from \thref{thm: act gap}. Otherwise, there is an lc center not contained in $F''$. We claim that in this case, we have $t \leq \frac{251}{252}$. Indeed, write $B'' = \sum\limits_{i} b''_{i}D''_{i}$. If the lc center is a divisor $E$ on $X''$, then the adjunction on $E$ gives
        $$
        \sum\limits_{i} \frac{m_{i}-1 + s_{i}t + \sum\limits_{j} r_{ij}b''_{j}}{m_{i}} = 2
        $$
        for some positive integers $m_{i}$ and nonnegative integers $r_{ij}$ and $s_{i}$. We can assume that $t > \frac{6}{7}$. Then $s_{i}=0$ or $s_{i}=1$. Moreover, if $s_{i}=1$ then $r_{ij}=0$ for all $j$. Since $\rho(X'')=1$, we have $F'' \cap E \neq \emptyset$. Hence $s_{i} > 0$ for some $i$. The adjunction formula can be rewritten as
        $$
        \left(\sum\limits_{i,s_{i}=1} \frac{m_{i}-1 + t}{m_{i}}\right) + \left( \sum\limits_{i, s_{i}=0} \frac{m_{i}-1 + \sum\limits_{j} r_{ij}b''_{j}}{m_{i}} \right) = 2.
        $$
        There are at most two nonzero terms in the first bracket. If there is exactly one nonzero term $\frac{m_{i_{0}}-1 + t}{m_{i_{0}}}$, then we have
        $$
        t \leq \frac{m_{i_{0}}-1 + t}{m_{i_{0}}} = 2 - \left( \sum\limits_{i, s_{i}=0} \frac{m_{i}-1 + \sum\limits_{j} r_{ij}b''_{j}}{m_{i}} \right) \leq 2 - \alpha_{1}(\widetilde{\Gamma(\mathcal{N}_{1})}) = 1 - \frac{1}{252}.
        $$
        If there are two nonzero terms, then similarly we have $2t \leq 2 - \frac{1}{7}$. Hence $t \leq \frac{13}{14}$. 
        
        Next, we assume that the lc center is a closed point. Let $f: Y \rightarrow X''$ be the extraction of an lc place $E$. Let $B''_{Y}$ be the strict birational transformation of $B''$ on $Y$, and $F''_{Y}$ be the strict birational transformation of $F''$ on $Y$. We run the $(K_{Y}+B''_{Y}+E)$-MMP on $Y$ and consider the first extremal contraction. Notice that
        $$
        K_{Y}+B''_{Y}+E \sim_{\mathbb{Q}} -tF''_{Y} = -tf^{*}F''.
        $$
        Hence the extremal contraction does not contract $E$. If it is a divisorial contraction $Y \rightarrow Y'$, then on $Y'$ we have $F''_{Y'} \cap E_{Y'} \neq \emptyset$, where $F''_{Y'}$ and $E_{Y'}$ are the strict birational transformation of $F''$ and $E$ on $Y'$ respectively. In this case, the same adjunction argument as above gives $t \leq \frac{251}{252}$. If the first extremal contraction is a Mori fibration $Y \rightarrow Z$, then taking the adjunction on the general fibre, we have
        $$
        t + 1 + \sum\limits_{i} w_{i}b_{i} = 2
        $$
        for some nonnegative integers $w_{i}$. Hence we have $t \leq \frac{6}{7}$.
    \end{itemize}
    This completes the proof when $F$ is a prime divisor.

    Suppose the theorem is true when $F$ contains $\leq r-1$ prime components for some positive integer $r \geq 2$. If $B$ and $F$ have some common prime component $D_{i}$, then $t \leq 1 - b_{i} \leq \frac{6}{7}$. Hence we can assume that $B$ and $F$ have no common prime component. Suppose, to the contrary, that $\frac{127259}{127260} < t < 1$. Then we have
    $$
    \mathbb{R}\mathrm{-ct}(X,B+\frac{127259}{127260}F_{1};\sum\limits_{i=2}^{r}F_{i}) > \frac{127259}{127260}.
    $$
    By the induction hypothesis, we have
    $$
    \mathbb{R}\mathrm{-ct}(X,B+\frac{127259}{127260}F_{1};\sum\limits_{i=2}^{r}F_{i}) = 1.
    $$
    Take a positive rational number $q \in \Gamma(\mathcal{N}_{1})$ such that $1> q > t$. Then we have
    $$
    t \geq \mathbb{R}\mathrm{-ct}(X,B+q\sum\limits_{i=2}^{r}F_{i};F_{1}) \geq \frac{127259}{127260}.
    $$
    By the induction hypothesis, we have
    \begin{equation}\label{equation final}
        \mathbb{R}\mathrm{-ct}(X,B+q\sum\limits_{i=2}^{r}F_{i};F_{1}) = \frac{127259}{127260}.
    \end{equation}
    Take a positive real number $\frac{127259}{127260} < s < t$ and an $\mathbb{R}$-complement
    $$
    B^{+} \geq B + sF.
    $$
    Take a positive number $1 > u > q$ and an $\mathbb{R}$-complement
    $$
    B^{\prime +} \geq B + \frac{127259}{127260}F_{1} + u \sum\limits_{i=2}^{r}F_{i}.
    $$
    Put $0 < \lambda = \frac{u-q}{u-s} < 1$. Then $(X,\lambda B^{+} + (1 - \lambda)B^{\prime +})$ is a log canonical 0-pair, and we have
    $$
    \lambda B^{+} + (1 - \lambda)B^{\prime +} \geq B + (\lambda s + (1-\lambda)\frac{127259}{127260})F_{1} + q \sum\limits_{i=2}^{r}F_{i}.
    $$
    Hence
    $$
    \mathbb{R}\mathrm{-ct}(X,B+q\sum\limits_{i=2}^{r}F_{i};F_{1}) \geq (\lambda s + (1-\lambda)\frac{127259}{127260}) > \frac{127259}{127260},
    $$
    contradicting (\ref{equation final}). This completes the induction proof.
\end{proof}

\section{Optimal bounds}

    Finally, we show that the gap numbers given in \thref{thm:lct gap}, \thref{thm: act gap} and \thref{thm: Rct gap} are optimal.

\begin{example}[Optimal bound for lct]
Consider the pair $(\mathbb{P}^2, \frac{1}{7}L_{1} + \frac{31}{36}L_{2} + \frac{251}{252}L_{3})$, where $L_{1},L_{2},L_{3}$ are three lines intersecting at one point. Then the pair is log canonical but not klt at the point of intersection. This shows that the bound given by \thref{thm:lct gap} is optimal.
\end{example}

\begin{example}[Optimal bound for act, I]\label{example: optimal bound for act}
    Consider the pair 
    $$
    \left(\mathbb{P}^2, \frac{1}{7}L_{1} + \frac{31}{36}L_{2} + \frac{503}{505}L_{3} + \frac{127259}{127260}L_{4} \right),
    $$ 
    where $L_{1},L_{2},L_{3},L_{4}$ are general lines on $\mathbb{P}^2$. Then the pair is a klt 0-pair.
\end{example}

\begin{example}[Optimal bound for act, II]\label{example: optimal bound for act 2}
    Consider the pair 
    $$
    \left(\mathbb{P}(1,2,3), \frac{127259}{127260}D_{1} + \frac{1}{7}D_{2} + \frac{502}{505}D_{3} + \frac{7}{12}D_{4} \right)
    $$
    where the divisors $D_{1},D_{3},D_{4}$ are toric invariant divisors of degree $3,2,1$, respectively, and the divisor $D_{2}$ is a general hypersurface of degree 3. Then the pair is a klt 0-pair.
\end{example}

\printbibliography

@article {HMX,
    AUTHOR = {Hacon, Christopher D. and McKernan, James and Xu, Chenyang},
     TITLE = {A{CC} for log canonical thresholds},
   JOURNAL = {Ann. of Math. (2)},
  FJOURNAL = {Annals of Mathematics. Second Series},
    VOLUME = {180},
      YEAR = {2014},
    NUMBER = {2},
     PAGES = {523--571},
      ISSN = {0003-486X},
   MRCLASS = {14E05 (14C20 14E30)},
  MRNUMBER = {3224718},
MRREVIEWER = {Alexandr V. Pukhlikov},
       DOI = {10.4007/annals.2014.180.2.3},
       URL = {https://doi.org/10.4007/annals.2014.180.2.3},
}

@book{KM,
  author    = {Koll{\'a}r, J{\'a}nos and Mori, Shigefumi},
  title     = {Birational Geometry of Algebraic Varieties},
  series    = {Cambridge Tracts in Mathematics},
  volume    = {134},
  publisher = {Cambridge University Press},
  address   = {Cambridge},
  year      = {1998},
  doi       = {10.1017/CBO9780511662560},
  isbn      = {978-0-521-63277-5},
  note      = {With the collaboration of C. H. Clemens and A. Corti}
}

@misc{Shokurov20,
  author        = {Shokurov, V. V.},
  title         = {Existence and boundedness of {$n$}-complements},
  year          = {2020},
  month         = dec,
  howpublished  = {arXiv preprint},
  eprint        = {2012.06495},
  archivePrefix = {arXiv},
  primaryClass  = {math.AG},
  doi           = {10.48550/arXiv.2012.06495},
  url           = {https://arxiv.org/abs/2012.06495}
}

@article{LiuSho,
      title={Optimal bounds on surfaces}, 
      author={Jihao Liu and V. V. Shokurov},
      year={2023},
      eprint={2305.19248},
      archivePrefix={arXiv},
      primaryClass={math.AG},
      url={https://arxiv.org/abs/2305.19248}, 
}

@article{Shokurov93,
doi = {10.1070/IM1993v040n01ABEH001862},
url = {https://dx.doi.org/10.1070/IM1993v040n01ABEH001862},
year = {1993},
month = {2},
publisher = {},
volume = {40},
number = {1},
pages = {95},
author = {V V Shokurov},
title = {3-FOLD LOG FLIPS},
journal = {Izvestiya: Mathematics}
}

@inbook{Alexeev92_lc_surface,
    author = {Valery Alexeev},
    title = {Classification of log canonical surface singularities: arithmetical proof.},
    publisher = {Société mathématique de France},
    year = {1992},
    chapter = {3},
    series = {Astérisque},
    booktitle = {Flips and abundance for algebraic threefolds},
    volume = {211}
}

@article{Belousov08,
    author = {Grigory Belousov},
    title = {Del Pezzo surfaces with log terminal singularities},
    journal = {Mathematical Notes},
    year = {2008},
    issue = {1},
    pages = {152--161}
}

@article{Belousov09,
    author = {Grigory Belousov},
    title = {The maximal number of singular points on log del Pezzo surfaces},
    journal = {Journal of Mathematical Sciences, the University of Tokyo},
    year = {2009},
    issue = {2},
    pages = {231--238}
}

@article{LX25,
    author  = {Liu, Jihao and Xie, Lingyao},
    title   = {Number of {S}ingular {P}oints on {P}rojective {S}urfaces},
    journal = {Chinese Annals of Mathematics, Series B},
    volume  = {46},
    number  = {5},
    pages   = {713--724},
    year    = {2025},
    doi     = {10.1007/s11401-025-0035-y},
    url     = {https://doi.org/10.1007/s11401-025-0035-y}
}

@article{Alexeev1993TwoTwoDimensionalTerminations,
  author  = {Alexeev, Valery},
  title   = {Two two-dimensional terminations},
  journal = {Duke Mathematical Journal},
  volume  = {69},
  number  = {3},
  pages   = {527--545},
  year    = {1993},
  month   = mar,
  doi     = {10.1215/S0012-7094-93-06922-0},
  url     = {https://doi.org/10.1215/S0012-7094-93-06922-0}
}

@misc{Fujino2024GeneralKodairaVanishing,
  author = {Fujino, Osamu},
  title  = {On general {K}odaira vanishing theorem},
  year   = {2024},
  note   = {Version 0.02, dated 2024/12/23},
  url    = {https://www.math.kyoto-u.ac.jp/~fujino/memo-on-kv.pdf}
}

@article{DelPezzo1885MinimalDegreeSurfaces,
  author  = {{del Pezzo}, Pasquale},
  title   = {Sulle superficie di ordine {$n$} immerse nello spazio di {$n+1$} dimensioni},
  journal = {Rendiconti della Reale Accademia delle Scienze Fisiche e Matematiche di Napoli},
  volume  = {24},
  year    = {1885},
  pages   = {212--216}
}

@misc{HLS,
  author        = {Han, Jingjun and Liu, Jihao and Shokurov, V. V.},
  title         = {{ACC} for minimal log discrepancies of exceptional singularities},
  year          = {2019},
  eprint        = {1903.04338},
  archivePrefix = {arXiv},
  primaryClass  = {math.AG},
  note          = {Version 2, March 2020},
  url           = {https://arxiv.org/abs/1903.04338v2}
}

\end{spacing}
\end{document}